\documentclass[11pt]{amsart}
\usepackage{amsmath}
\usepackage{amssymb}
\usepackage{amsthm}
\usepackage[all]{xy}
\usepackage{graphicx}
\usepackage[margin=1in]{geometry}
\usepackage{color}
\usepackage{hyperref}
\usepackage{verbatim}
\usepackage[mathscr]{eucal}
\usepackage{pb-diagram,pb-xy}

\newtheorem*{thmA}{Theorem A}
\newtheorem*{thmB}{Theorem B}

\newtheorem*{corC}{Corollary C}

\numberwithin{equation}{section}

\newcommand{\R}{\ensuremath{\mathbb{R}}}
\newcommand{\N}{\ensuremath{\mathbb{N}}}

\newcommand{\Z}{\ensuremath{\mathbb{Z}}}

\newtheorem{theorem}{Theorem}[section]
\newtheorem{Addendum}{Addendum}[section]

\newtheorem{corollary}[theorem]{Corollary}

\newtheorem{proposition}[theorem]{Proposition}
\newtheorem{lemma}[theorem]{Lemma}
\newtheorem{claim}[theorem]{Claim}

\theoremstyle{definition} 
\newtheorem{defn}[theorem]{Definition}
\newtheorem{remark}[theorem]{Remark} 
\newtheorem{Notation}[theorem]{Notational Convention}

\newtheorem{convention}{Convention}[section]
\newtheorem{Construction}{Construction}[section]

\newcommand{\Th}{{\mathsf T}\!{\mathsf h}}
\newcommand{\MT}{{\mathsf M}{\mathsf T}}

\newcommand{\mb}[1]{\mathbf{#1}}

\newcommand{\Cob}{{\mathbf{Cob}}}

\DeclareMathOperator*{\hofibre}{hofibre}

\DeclareMathOperator*{\Int}{Int}

\DeclareMathOperator*{\Diff}{Diff}

\DeclareMathOperator*{\Emb}{Emb}

\DeclareMathOperator*{\Maps}{Maps}

\DeclareMathOperator*{\Top}{\text{Top}}

\DeclareMathOperator*{\Id}{\text{Id}}

\DeclareMathOperator*{\Bun}{\mathrm{Bun}}

\DeclareMathOperator*{\Ob}{\text{Ob}}

\DeclareMathOperator*{\Ker}{Ker}

\DeclareMathOperator*{\Image}{Im}

\DeclareMathOperator*{\loc}{loc}

\DeclareMathOperator*{\hocolim}{hocolim}
\DeclareMathOperator*{\std}{std}

\DeclareMathOperator*{\stb}{stb}

\DeclareMathOperator*{\Surg}{Surg}
\DeclareMathOperator*{\sdl}{sdl}
\DeclareMathOperator*{\Trc}{Trc}
\DeclareMathOperator*{\rg}{rg}
\DeclareMathOperator*{\tr}{tr}

\DeclareMathOperator*{\Spin}{Spin}
\DeclareMathOperator*{\KO}{KO}
\DeclareMathOperator*{\inddiff}{ind-diff}
\DeclareMathOperator*{\tor}{tor}
\DeclareMathOperator*{\torp}{torp}
\DeclareMathOperator*{\psc}{psc}
\DeclareMathOperator*{\mf}{mf}

\DeclareMathOperator*{\trg}{trg}
\DeclareMathOperator*{\bas}{bas}

\author{Nathan Perlmutter}
\address{Stanford University Department of Mathematics, Building 380, Stanford, California,  94305, USA}
\email{nperlmut@stanford.edu}

\title[Parametrized Morse Theory and Positive Scalar Curvature]{Parametrized Morse Theory and Positive Scalar Curvature}

\begin{document}
\maketitle

\begin{abstract}
We use the cobordism category constructed in \cite{P 17} to the study the homotopy type of the space of positive scalar curvature metrics on a spin manifold of dimension $\geq 5$.
Our methods give an alternative proof and extension of a recent theorem of Botvinnik, Ebert, and Randal-Williams from \cite{BERW 16}. 
\end{abstract}

\setcounter{tocdepth}{1}
\tableofcontents
\vspace*{-5mm}

\section{Introduction} \label{section: Introduction}
Fix a tangential structure $\theta: B \longrightarrow BO$ and integers $k < d$.
In previous work \cite{P 17} we determined the homotopy type of the topological category $\Cob^{\mf, k}_{\theta, d+1}$.
An object of $\Cob^{\mf, k}_{\theta, d+1}$ is given by a closed, $d$-dimensional submanifold $M \subset \R^{\infty}$, equipped with a $\theta$-structure $\hat{\ell}_{M}$, for which its underlying map $\ell_{M}: M \longrightarrow B$ is $k$-connected. 
A morphism $M \rightsquigarrow N$ is an embedded $\theta$-cobordism $W \subset [0, 1]\times\R^{\infty}$ between $\{0\}\times M$ and $\{1\}\times N$, subject to the following condition:
\begin{itemize} 
\item
The height function, 
$W \hookrightarrow [0, 1]\times\R^{\infty} \stackrel{\text{proj.}} \longrightarrow [0, 1],$
is a Morse function with the property that all critical points $c \in W$ satisfy the condition: 
$k < \text{index}(c) < d-k+1.$ 
\end{itemize}
Composition is defined in the usual way by concatenation of cobordisms. 
In \cite{P 17} we identified the homotopy type of the classifying space $B\Cob^{\mf, k}_{\theta, d+1}$ with that of the infinite loopspace of a certain Thom-spectrum, $\mb{hW}^{k}_{\theta, d+1}$, associated to the space of Morse jets on $\R^{d+1}$. 
The main theorem from \cite{P 17} is stated below:
\begin{thmA}[N.P 2017]
Let $k < d/2$ and suppose that the tangential structure $\theta: B \longrightarrow BO$ is chosen so that the space $B$ satisfies Wall's finiteness condition $F(k)$ (see \cite{W 65}). 
Then there is a weak homotopy equivalence,
$B\Cob^{\mf, k}_{\theta, d+1} \simeq \Omega^{\infty-1}\mb{hW}^{k}_{\theta, d+1}.$
\end{thmA}
In this paper we use the above theorem to study the homotopy type of the space of positive scalar curvature metrics on a closed spin-manifold of dimension $\geq 5$.
In particular, we use our work from \cite{P 17} to give an alternative proof and extension of the results of Botvinnik, Ebert, and Randal-Williams from \cite{BERW 16}.
To state our main theorem, we need to implement some definitions and notation. 
For what follows, let us assume that $d \geq 5$, and that the tangential structure $\theta: B \longrightarrow BO$ is such that the space $B$ is $2$-connected.
In this case, the map $\theta$ factors up to homotopy through the projection $B\Spin \longrightarrow BO$, and thus every $\theta$-manifold has a canonical induced spin-structure. 
Fix a closed $d$-dimensional $\theta$-manifold $M$. 
We let $\mathcal{R}^{+}(M)$ denote the space of Riemannian metrics on $M$ with positive scalar curvature.  
We let,
$$\inddiff: \mathcal{R}^{+}(M) \longrightarrow \Omega^{\infty+d+1}\KO,$$
denote the \textit{index difference map} first constructed by Hitchen in \cite{H 74} (see also \cite{E 16} for a definition). 
The spectrum $\mb{hW}^{k}_{\theta, d+1}$ receives a map 
$i_{d}: \Sigma^{-1}\MT\theta_{d} \longrightarrow \mb{hW}^{k}_{\theta, d+1}$
where $\MT\theta_{d}$ is the \textit{Madsen-Tillmann-Weiss} spectrum associated to the tangential structure $\theta_{d}: B(d) \longrightarrow BO(d)$, which is obtained by restricting $\theta$ to $BO(d) \subset BO$.
$\MT\theta_{d}$ is defined to be the Thom-spectrum associated to the virtual bundle $-\theta_{d}^{*}\gamma^{d}$ over $B(d)$, where $\gamma^{d}$ is the cannonical bundle on $BO(d)$. 
The spin-structure on the bundle $\theta_{d}^{*}\gamma^{d}$ endows $\MT\theta_{d}$ with a $\KO$-orientation, which yields an infinite loop map $\mathcal{A}_{d}: \Omega^{\infty}\MT\theta_{d} \longrightarrow \Omega^{\infty+d}\KO$.
The main theorem of this paper is stated below:
\begin{thmB} \label{theorem: Main theorem}
Let $d \geq 5$ and let $2 \leq k < d/2$ be an integer.
Suppose that $B$ is $2$-connected and that it satisfies Wall's finiteness condition $F(k)$. 
Let $M$ be a closed $\theta$-manifold for which the map $\ell_{M}: M \longrightarrow B$ is $k$-connected. 
There exists a map 
$\rho: \Omega^{\infty}\mb{hW}^{k}_{\theta, d+1} \longrightarrow \mathcal{R}^{+}(M)$
such that:
\begin{enumerate} \itemsep.2cm
\item[(i)] In the case that $d = 2n \geq 6$, the composite, 
$$
\xymatrix{
\Omega^{\infty+1}\MT\theta_{2n} \ar[rr]^{\Omega^{\infty}i_{2n}} && \Omega^{\infty}\mb{hW}^{k}_{\theta, 2n+1} \ar[r]^{\rho} & \mathcal{R}^{+}(M) \ar[rr]^{\inddiff} && \Omega^{\infty+2n+1}\KO,
}
$$
is homotopic to, $\Omega\mathcal{A}_{2n}:  \Omega^{\infty+1}\MT\theta_{2n} \longrightarrow \Omega^{\infty+2n+1}\KO$.
\item[(ii)] In the case that $d = 2n-1 \geq 5$, the composite, 
$$
\xymatrix{
\Omega^{\infty+2}\MT\theta_{2n-1} \ar[rr]^{\Omega^{\infty+1}i_{2n-1}} && \Omega^{\infty+1}\mb{hW}^{k}_{\theta, 2n} \ar[r]^{\Omega\rho} & \Omega\mathcal{R}^{+}(M) \ar[rr]^{\Omega\inddiff} && \Omega^{\infty+2n+1}\KO,
}
$$
is homotopic to, $\Omega^{2}\mathcal{A}_{2n-1}:  \Omega^{\infty+1}\MT\theta_{2n-1} \longrightarrow \Omega^{\infty+2n+1}\KO$.
\end{enumerate}
\end{thmB}

Theorem B can be used to detect non-trivial homotopy groups in the space of $\psc$ metrics $\mathcal{R}^{+}(M)$. 
Following \cite[Section 5.2]{BERW 16}, in the case that $\theta = \Spin: B\Spin \longrightarrow BO,$ the effect of the map $\Omega\mathcal{A}_{d}$ on homotopy groups can be computed. 
In that paper it is proven that the induced map $\pi_{k}(\Omega\mathcal{A}_{d}): \pi_{k}(\Omega^{\infty+1}\MT\theta) \longrightarrow \pi_{k}(\Omega^{\infty+d+1}\KO)$ is non-trivial for all degrees $k \in \Z_{\geq 0}$. 
Combining this fact about the induced map $\pi_{k}(\Omega\mathcal{A}_{d})$ with Theorem B it follows that the induced map $\pi_{k}(\inddiff)$ must be nontrivial for all $k \in \Z_{\geq 0}$ as well. 
We have the following corollary.
\begin{corC} \label{corollary: surjection on rational homotopy}
Let $d \geq 5$ and let $M$ be a simply-connected, closed, spin manifold. 
Then for any integer $k \in \Z_{\geq 0}$, the induced map 
$\pi_{k}(\inddiff): \pi_{k}(\mathcal{R}^{+}(M)) \longrightarrow \pi_{k}(\Omega^{\infty+d+1}\KO)$
is nonzero.
In particular, it induces a surjection on all rational homotopy groups. 
\end{corC}

The above theorems recover and extend the results of Botvinnik, Ebert, and Randal-Williams from \cite{BERW 16}.
Indeed, the results of \cite{BERW 16} apply to manifolds of dimensions $\geq 6$, while the five-dimensional case falls out of scope of the methods used in their paper. 
Our Theorem B and Corollary C hold for the case $d = 5$.

We remark that the spectrum $\mb{hW}^{k}_{\theta, d+1}$ contains more structure than the spectrum $\MT\theta_{d}$.
Because of this we believe that the map $\rho: \Omega^{\infty}\mb{hW}^{k}_{\theta, d+1} \longrightarrow \mathcal{R}^{+}(M)$ could be used to extract extra homotopical information about $\mathcal{R}^{+}(M)$ that is missed by just studying the map $\Omega^{\infty+1}\MT\theta_{d} \longrightarrow \mathcal{R}^{+}(M)$.
We save this question for another project.

\subsection{Outline of the methods} \label{subsection: outline of the methods}
Let us now describe our methods. 
We give an outline for how to use the cobordism category $\Cob_{\theta, d+1}^{\mf, k}$ together with Theorem A to define the map $\rho$ from the statement of Theorem B.
Our construction is powered by the parametrized version of the \textit{Gromov-Lawson construction} \cite{GL 80},  developed in the work of Chernysh \cite{Ch 04} and Walsh \cite{Wa 13}, \cite{Wa 11}.
Let $M$ and $N$ be closed manifolds of dimension $d \geq 5$, and let $W: M \rightsquigarrow N$ be a cobordism equipped with a Morse function $h: W \longrightarrow [0, 1]$ that satisfies the following two conditions:
\begin{enumerate} \itemsep.2cm
\item[(i)] $0$ and $1$ are regular values for $h$ with $h^{-1}(0) = M$ and $h^{-1}(1) = N$; 
\item[(ii)] all critical points $c \in W$ of $h$ satisfy: $2 < \text{index}(c) < d+1-2$. 
\end{enumerate}
Using a version of the parametrized Gromov-Lawson construction in \cite{Wa 11}, it can be shown that the pair $(W, h)$ induces a weak map 
\begin{equation} \label{equation: weak gromov lawson map}
\mb{G}(W, h): \mathcal{R}^{+}(M) \; \longrightarrow \; \mathcal{R}^{+}(N),
\end{equation}
which by the results of Walsh and Chernysh is a weak homotopy equivalence. 
We emphasize that the map (\ref{equation: weak gromov lawson map}) depends both on the cobordism $W$ and on the choice of Morse function $h$.
Furthermore, the map can only be defined in the case that the critical points of $h$ satisfy condition (ii) above. 

It can be shown that the (weak) map $\mb{G}(W, h)$ varies continuously in the data $(W, h)$. 
For $k \geq 2$, we may then use this construction to define a continuous functor 
\begin{equation} \label{equation: positive scalar curvature functor}
\mathcal{R}^{+}: \Cob^{\mf, k}_{\theta, d+1} \; \longrightarrow \; \Top,
\end{equation}
defined on objects by sending $M \in \Ob\Cob^{\mf, k}_{\theta, d+1}$ to the space of $\psc$ metrics $\mathcal{R}^{+}(M)$; and defined on morphisms by sending $W: M \rightsquigarrow N$ to the map
$
\mb{G}(W, h_{W}): \mathcal{R}^{+}(M) \longrightarrow \mathcal{R}^{+}(N),
$
where $h_{W}$ is the height function associated to the morphism $W$.
This functor is well-defined because the height function $h_{W}$ satisfies conditions (i) and (ii) by the definition of the category $\Cob^{\mf, k}_{\theta, d+1}$.
We emphasize that it would not be possible to define such a functor out of the category $\Cob^{\mf, k}_{\theta, d+1}$ if $k < 2$.

Using this functor $\mathcal{R}^{+}$ from (\ref{equation: positive scalar curvature functor}) we may form the \textit{transport category}, $\mathcal{R}^{+}\wr\Cob^{\mf, k}_{\theta, d+1}$. 
An object of $ \mathcal{R}^{+}\wr\Cob^{\mf, k}_{\theta, d+1}$ is a pair $(M, g)$ with $M \in \Ob\Cob^{\mf, k}_{\theta, d+1}$ and $g \in \mathcal{R}^{+}(M)$; a morphism $(M, g) \longrightarrow (N, g')$ is an element $W \in \Cob^{\mf, k}_{\theta, d+1}(M, N)$ with $\mb{G}(W, h_{W})(g) = g'$. 
Since the maps (\ref{equation: positive scalar curvature functor}) are weak homotopy equivalences, it follows that the projection map 
\begin{equation} \label{equation: projection functor}
B(\mathcal{R}^{+}\wr\Cob^{\mf, k}_{\theta, d+1}) \longrightarrow B\Cob^{\mf, k}_{\theta, d+1}
\end{equation}
is a quasi-fibration and that the fibre over $M \in \Ob\Cob^{\mf, k}_{\theta, d+1} \subset B\Cob^{\mf, k}_{\theta, d+1}$ is given by the space of $\psc$ metrics, $\mathcal{R}^{+}(M)$. 
Choosing points $M \in \Ob\Cob^{\mf, k}_{\theta, d+1}$ and $g_{0} \in \mathcal{R}^{+}(M)$ determines a fibre-transport map
$$j_{M, g_{0}}: \Omega_{M}B\Cob^{\mf, k}_{\theta, d+1} \longrightarrow \mathcal{R}^{+}(M).$$ 
Combining this with the weak homotopy equivalence 
$\Omega^{\infty-1}\mb{hW}^{k}_{\theta, d+1} \simeq B\Cob^{\mf, k}_{\theta, d+1}$ 
of Theorem A yields a (weak) map,
\begin{equation} \label{equation: map into R +}
\rho: \Omega^{\infty}\mb{hW}^{k}_{\theta, d+1} \; \longrightarrow \; \mathcal{R}^{+}(M).
\end{equation}
This is our definition of the map from the statement of Theorem A.

\begin{remark} \label{remark: well-defindedness} 
We emphasize that the map (\ref{equation: weak gromov lawson map}) obtained by the Gromov-Lawson construction is only a weak map and is not well defined `on the nose'. 
To rectify this, one must work with a different (but nonetheless homotopy equivalent) model of the category $\Cob_{\theta, d+1}^{\mf, k}$.
This alternative model for $\Cob_{\theta, d+1}^{\mf, k}$ was also used in our earlier work \cite{P 17}, defined in Section 4 of that paper. 
\end{remark}

The construction of the map $\rho$ from (\ref{equation: map into R +}) is carried out in Section \ref{subsection: the fibre transport}, using the constructions from the earlier sections of the paper. 
Once this map is constructed, it remains to prove the statements in Theorem B.
Theorem B breaks down into two cases: $d = 2n$ and $d = 2n-1$. 
In Section \ref{section: fibre sequences and stabilization} we prove the even-dimensional case. 
This yields an alternative proof of \cite[Theorem B]{BERW 16}, originally proven by Botvinnik, Ebert, and Randal-Williams.
Our methods enable us to circumvent the arguments from \cite[Section 4]{BERW 16}, which contains the bulk of the homotopy theory used in their paper.
However, that part of their paper is essentially replaced with Theorem A and our techniques from \cite{P 17} and thus our proof isn't actually much simpler then theirs, just conceptually different.
Our proof also still requires most of the index theoretic results from \cite[Section 3]{BERW 16} and in the related paper by Ebert \cite{E 16}.

The rest of our paper is devoted to proving the odd-dimensional case of Theorem B (including the five-dimensional case), and this is where the advantage of our construction begins to show itself.
In particular it will give us access to five-dimensional manifolds. 
The strategy is to use the even-dimensional case to prove the odd-dimensional case. 
Let us describe the idea.
Let $d = 2n-1$, and let $W$ be a compact $2n$-dimensional manifold with non-empty boundary, $M = \partial W$.
Let $\mathcal{R}^{+}(W)$ be the space of $\psc$ metrics on $W$ that factor as a product near the boundary (with respect to some pre-determined collar). 
For a metric $g \in \mathcal{R}^{+}(M)$, we let $\mathcal{R}^{+}(W)_{g} \subset \mathcal{R}^{+}(W)$ be the subspace consisting of those metrics $h$ for which $h|_{M} = g$.
In \cite{Ch 06}, Chernysh proved that the restriction map, 
$$
r: \mathcal{R}^{+}(W) \longrightarrow \mathcal{R}^{+}(M), \quad h \mapsto h|_{M},
$$
is a quasi-fibration.
The fibre of $r$ over an element $g \in \mathcal{R}^{+}(M)$ is the space $\mathcal{R}^{+}(W)_{g}$. 
Choosing a basepoint in $\mathcal{R}^{+}(W)$ yields a fibre-transport map, 
$$
T: \Omega\mathcal{R}^{+}(M) \longrightarrow \mathcal{R}^{+}(W)_{g}.
$$
Now, 
in \cite[Theorem 3.6.1]{BERW 16} (using \cite{E 16}) it is proven that the following diagram 
\begin{equation} \label{equation: commutativity of index difference}
\xymatrix{
\Omega\mathcal{R}^{+}(M) \ar[dr]_{-\inddiff} \ar[rr]^{T} && \mathcal{R}^{+}(W)_{g} \ar[dl]^{\inddiff} \\
& \Omega^{\infty+2n+1}\KO
}
\end{equation}
is homotopy commutative. 
This is the key result that we use to relate the even-dimensional case of Theorem B to the odd-dimensional case. 
There is still more work to be done. 
With homotopy commutativity of (\ref{equation: commutativity of index difference}) in place, the next step is to construct a map, 
$$
\tau: \Omega B\Cob_{\theta, 2n}^{\mf, k} \longrightarrow B\Cob_{\theta, 2n+1}^{\mf, k},
$$
that makes the following diagram commute up to homotopy,
\begin{equation} \label{equation: commutative square of cobordism categories}
\xymatrix{
 \Omega^{2}B\Cob_{\theta, 2n}^{\mf, k}  \ar[d]^{\Omega\rho} \ar[rr]^{\Omega\tau} && \Omega B\Cob_{\theta, 2n+1}^{\mf, k} \ar[d]^{\rho}  \\
 \Omega\mathcal{R}^{+}(M) \ar[rr]^{T} && \mathcal{R}^{+}(W)_{g}.
}
\end{equation}
Once these maps are defined and commutativity is established, the proof of Theorem B part (ii) will follow by essentially composing the commutative diagrams (\ref{equation: commutativity of index difference}) and (\ref{equation: commutative square of cobordism categories}), and then observing that the $\KO$-orientations on $ \Omega^{2}B\Cob_{\theta, 2n-1}^{\mf, k}$ and $\Omega B\Cob_{\theta, 2n}^{\mf, k} $ (induced by those on $\Omega^{\infty+1}\mb{hW}^{k}_{\theta, 2n}$ and $\Omega^{\infty}\mb{hW}^{k}_{\theta, 2n+1}$ via Theorem A) are compatible. 
The map $\tau$ is a fibre-transport map of a different sort, arising from a sequence of continuous functors, 
\begin{equation} \label{equation: genauer sequence intro}
\Cob_{\theta, 2n}^{\mf, k} \longrightarrow \Cob_{\theta, 2n}^{\partial, \mf, k} \longrightarrow \Cob_{\theta, 2n-1}^{\mf, k},
\end{equation}
that geometrically realizes to a fibre sequence.
This fibre sequence is analogous to the fibre sequence of cobordism categories studied by Genauer in \cite{G 12}. 
In particular, the middle term is a cobordism category whose objects are given by manifolds with boundary, and morphisms are relative cobordisms (which are manifolds with corners) equipped with Morse functions.
Sections 6 through 9 are devoted to constructing this fibre sequence and proving commutativity of the diagram (\ref{equation: commutative square of cobordism categories}).
This includes a number of technical steps regarding surgery, Morse functions, and spaces of $\psc$ metrics on manifolds with boundary. 

\subsection{Organization} \label{subsection: organization}
In Section \ref{section: preliminary constructions} we give basic preliminary definitions and recall constructions from other papers that we will need to use. 
Section \ref{section: Spaces of manifolds equipped with surgery} is devoted to recalling a construction from \cite{P 17} that gives an alternative model for the cobordism category $\Cob_{\theta, d+1}^{\mf, k}$. 
We remark that these first two sections are mostly comprised of recollections from \cite{P 17}. 
The definition in this section is the central construction of this paper and so we repeat it here even though it also is defined in \cite{P 17}.
In Section \ref{section: Spaces of Manifold Equipped with PSC Metrics and Surgery Data} we show how to use $\Cob_{\theta, d+1}^{\mf, k}$ to probe the space of $\psc$ metrics on a compact manifold. 
It is here that we define the map $\rho$ from the statement of Theorem B. 
Here we also reformulate Theorem B (see Theorem \ref{theorem: factorization of index difference}). 
In Section \ref{section: fibre sequences and stabilization} we prove the even-dimensional case of Theorem B. 
In Section \ref{section: Manifolds with boundary equipped with surgery data} we construct the map $\tau: \Omega B\Cob_{\theta, 2n}^{\mf, k} \longrightarrow B\Cob_{\theta, 2n+1}^{\mf, k}$.
The final two sections, Sections \ref{section: a fibre transport map} and \ref{subsection: positive scalar curvature metrics on manifolds with boundary}, are devoted to proving commutativity of (\ref{equation: commutative square of cobordism categories}) and putting everything together. 

\subsection{Acknowledgments} \label{subsection: Acknowledgments}
The author was supported by an NSF Post-Doctoral Fellowship, DMS-1502657, at Stanford University.

\section{Preliminary Constructions} \label{section: preliminary constructions}
\subsection{Spaces of manifolds} \label{subsection: spaces of compact manifolds}

We begin by fixing some conventions regarding tangential structures that we will use throughout the paper. 
\begin{defn} \label{Convention: tangential structures}
A (stable) tangential structure is a fibration $\theta: B \longrightarrow  BO$. 
For any integer $d \in \Z_{\geq 0}$, we will denote by $\theta_{d}: B(d) \longrightarrow BO(d)$ the map obtained by forming the pullback of the diagram, $BO(d) \hookrightarrow BO \stackrel{\theta} \longleftarrow B$, where the first arrow is the standard inclusion of $BO(d)$ into $BO$. 
A \textit{$\theta$-structure} on a $d$-dimensional manifold $M$ is a bundle map (i.e. fibrewise linear isomorphism) $\hat{\ell}_{M}: TM \longrightarrow \theta_{d}^{*}\gamma^{d}$, where $\gamma^{d}$ denotes the canonical vector bundle on $BO(d)$.
For such a $\theta$-structure, we will denote by $\ell_{M}: M \longrightarrow B(d)$ the map underlying the bundle map $\hat{\ell}_{M}$. 
A \textit{$\theta$-manifold} is a smooth manifold equipped with a $\theta$-structure on its tangent bundle.
\end{defn}

We define certain spaces of $\theta$-manifolds. 
\begin{defn} \label{defn: space of closed manifolds}
Let $d \in \Z_{\geq 0}$. 
The space $\mathcal{M}_{\theta, d}$ consists of closed, $d$-dimensional submanifolds $P \subset \R^{\infty}$, equipped with a $\theta$-structure 
$\hat{\ell}_{P}$.
The space $\mathcal{M}_{\theta, d}$ is topologized in the standard way following \cite[Section 2]{GRW 09}.
Topologized in this way there is a homeomorphism,
$$
\mathcal{M}_{\theta, d} \cong \coprod_{[P]}\left(\Emb(P, \R^{\infty})\times\Bun(TP, \theta^{*}_{d}\gamma^{d})\right)/\Diff(P),
$$
where $P$ ranges over all diffeomorphism classes of $d$-dimensional closed manifolds that admit a $\theta$-structure. 
\end{defn}

We will also need to work with a space of compact manifolds with boundary.
\begin{defn} \label{defn: manifolds with free boundary}
Let $d$ be a non-negative integer. 
The space $\mathcal{M}^{\partial}_{\theta, d}$ consists of compact $d$-dimensional submanifolds $M \subset (-\infty, 0]\times\R^{\infty-1}$, equipped with $\theta$-structure $\ell_{M}$, subject to the following condition:
there exists $\varepsilon > 0$ such that, 
$
\partial M\times(-\varepsilon, 0] \; = \; M\cap\left((-\varepsilon, 0]\times\R^{\infty-1}\right). 
$
\end{defn}

Consider the restriction map, 
$r: \mathcal{M}^{\partial}_{\theta, d} \longrightarrow \mathcal{M}_{\theta, d-1},$ $M \mapsto \partial M.$
For any $P \in \mathcal{M}_{\theta, d-1}$, we denote by $\mathcal{M}_{\theta, d}(P)$ the fibre of the map $r$ over the element $P \in \mathcal{M}_{\theta, d-1}$. 
Clearly there is a homeomorphism, $\mathcal{M}_{\theta, d}(\emptyset) \cong \mathcal{M}_{\theta, d}$.
It is a well known result that this restriction map is a Serre-fibration and thus for any $P \in \mathcal{M}_{\theta, d-1}$ we have a fibre sequence,
$
\xymatrix{
\mathcal{M}_{\theta, d}(P) \ar[r] & \mathcal{M}^{\partial}_{\theta, d} \ar[r]^{r} & \mathcal{M}_{\theta, d-1}.
}
$

Next, we define a space of non-compact manifolds that are open in a single direction.
\begin{defn} \label{defn: space of long manifolds}
The space $\mathcal{D}_{\theta, d+1}$ consists of $(d+1)$-dimensional submanifolds,
$W \subset \R\times\R^{\infty},$ 
equipped with a $\theta$-structure $\hat{\ell}_{W}$, subject to the condition that the 
function, 
$W \hookrightarrow \R\times\R^{\infty} \stackrel{\text{proj.}} \longrightarrow \R,$ 
is a proper map. 
\end{defn}
The space $\mathcal{D}_{\theta, d+1}$ is topologized by the same method employed in \cite[Section 2]{GRW 09}.
We don't repeat their construction here. 
We will also need to work with an analogue of the above definition for manifolds with boundary.
\begin{defn} \label{defn: space of long manifolds with boundary}
Fix an integer $d \in \N$. 
The space $\mathcal{D}^{\partial}_{\theta, d+1}$ consists of $(d+1)$-dimensional compact manifolds $W \subset \R\times(-\infty, 0]\times\R^{\infty-1}$, equipped with $\theta$-structure $\hat{\ell}_{W}$, subject to the following conditions:
\begin{enumerate} \itemsep.2cm
\item[(i)] The height function, $W \hookrightarrow \R\times(-\infty, 0]\times\R^{\infty-1} \stackrel{\text{proj.}} \longrightarrow \R,$ is a proper function.
\item[(ii)] $\partial W \; = \; W\cap\left(\R\times\{0\}\times\R^{\infty-1}\right)$. 
Furthermore, there exists $\varepsilon > 0$ such that, 
$$
\partial W\times(-\varepsilon, 0] \; = \; W\cap\left(\R\times(-\varepsilon, 0]\times\R^{\infty-1}\right).
$$
\end{enumerate}
Consider the boundary restriction map,
$r: \mathcal{D}^{\partial}_{\theta, d+1} \longrightarrow \mathcal{D}_{\theta, d}$, $W \mapsto \partial W.$ Given $P \in \mathcal{M}_{\theta, d-1}$, the product $\R\times P \subset \R\times\R^{\infty-1}$ determines an element of $\mathcal{D}_{\theta, d}$.
We denote by $\mathcal{D}_{\theta, d+1}(P)$ the fibre of the map $r$ over this element $\R\times P \in \mathcal{D}_{\theta, d}$.
Clearly we have $\mathcal{D}_{\theta, d+1}(\emptyset) = \mathcal{D}_{\theta, d+1}$.
\end{defn}

\begin{remark} \label{remark: collar condition}
In the above definition, all elements $W \in \mathcal{D}^{\partial}_{\theta, d}$ are required to be embedded in $\R\times(-\infty, 0]\times\R^{\infty-1}$ in a way so that the boundary is equipped with a collar, see condition (ii). 
For certain constructions latter on it will be useful to have a version of the space $\mathcal{D}^{\partial}_{\theta, d}$ where this collar condition is relaxed. 
We denote by $\mathcal{D}^{\partial, \pitchfork}_{\theta, d}$ the space consisting of submanifolds $W \subset \R\times(-\infty, 0]\times\R^{\infty-1}$ that satisfy the same conditions as in Definition \ref{defn: space of long manifolds with boundary}, except now we relax condition (ii) and only require the intersection of $W$ with $\R\times\{0\}\times\R^{\infty-1}$ to be transverse, and not necessarily orthogonal with a collar (we still require $\partial W \; = \; W\cap(\R\times\{0\}\times\R^{\infty-1})$).
By employing the argument used in the proof of \cite[Lemma 3.4]{GRW 09}, it can be proven that the inclusion 
$\mathcal{D}^{\partial}_{\theta, d} \hookrightarrow \mathcal{D}^{\partial, \pitchfork}_{\theta, d}$ 
is a weak homotopy equivalence.
We define the space $\mathcal{M}^{\partial, \pitchfork}_{\theta, d}$ analogously by relaxing the collar condition in the same way. 
Again, it follows that the inclusion $\mathcal{M}^{\partial}_{\theta, d} \hookrightarrow \mathcal{M}^{\partial, \pitchfork}_{\theta, d}$ is a weak homotopy equivalence. 
\end{remark}

\subsection{Spaces of manifolds equipped with morse function} \label{subsection: spaces of manifolds equipped with morse functions}
We will ultimately need to work with spaces of manifolds equipped with a choice of Morse function. 
Before we define the next space, let us first fix some notation.
For any element $W \in \mathcal{D}_{\theta, d+1}$, we will let $h_{W}: W \longrightarrow \R$ denote the height function, 
$\xymatrix{
W \ar@{^{(}->}[r] & \R\times(-\infty, 0]\times\R^{\infty-1} \ar[r]^{\ \ \ \ \ \ \ \ \ \  \text{proj.}} & \R.
}$
We use the same notation for the height functions associated to elements of $\mathcal{D}^{\partial}_{\theta, d+1}$ and $\mathcal{D}_{\theta, d+1}(P)$.

\begin{defn} \label{defn: space of long manifolds morse}
The subspace $\mathcal{D}^{\mf}_{\theta, d+1} \subset \mathcal{D}_{\theta, d+1}$ consists of those $W$ for which the height function $h_{W}: W \longrightarrow \R$ is a Morse function. 
Let $k \leq d$ be an integer.
The subspace $\mathcal{D}^{\mf, k}_{\theta, d+1} \subset \mathcal{D}^{\mf}_{\theta, d+1}$ consists of those $W$ subject the following further conditions:
\begin{enumerate} \itemsep.2cm
\item[(a)] all critical points $c \in W$ of the height function $h_{W}$ satisfy $k < \text{index}(c) < d-k+1$;
\item[(b)] the map $\ell_{W}: W \longrightarrow B$ is $k$-connected.  
\end{enumerate}
\end{defn}
By the argument used in the proof of \cite[Theorem 3.9]{GRW 09}, it follows that for all $k$ there is a weak homotopy equivalence,
$
\mathcal{D}^{\mf, k}_{\theta, d+1} \simeq B\Cob_{\theta, d+1}^{\mf, k},
$
where $\Cob_{\theta, d+1}^{\mf, k}$ is the topological category discussed in the introduction. 
As before, we will work with analogues of the above definition for manifolds with boundary.
Recall the space $\mathcal{D}^{\partial, \pitchfork}_{\theta, d+1}$ from Remark \ref{remark: collar condition}. 
\begin{defn} \label{defn: space of long manifolds w boundary morse}
The subspace $\mathcal{D}^{\partial, \mf}_{\theta, d+1} \subset \mathcal{D}^{\partial, \pitchfork}_{\theta, d+1}$ consists of those $W$ for which the height function $h_{W}: W \longrightarrow \R$ is a Morse function. 
Let $P \in \mathcal{M}_{\theta, d-1}$. 
The space $\mathcal{D}^{\mf}_{\theta, d+1}(P)$ is defined to be the fibre of the restriction map, 
$
r: \mathcal{D}^{\partial, \mf}_{\theta, d+1} \longrightarrow \mathcal{D}^{\mf}_{\theta, d},$ $W \mapsto \partial W,$
over the element $\R\times P \in \mathcal{D}^{\mf}_{\theta, d}$.
Clearly we have $\mathcal{D}^{\mf}_{\theta, d+1}(\emptyset) = \mathcal{D}^{\mf}_{\theta, d+1}$.

Let $k \leq d/2$ be an integer.
Let $P \in \mathcal{M}_{\theta, d-1}$.
The subspace $\mathcal{D}^{\mf, k}_{\theta, d+1}(P) \subset \mathcal{D}^{\mf}_{\theta, d+1}(P)$ consists of those $W$ that satisfy the following conditions:
\begin{enumerate} \itemsep.2cm
\item[(a)] all critical points $c \in W$ of the height function $h_{W}$ satisfy $k < \text{index}(c) < d-k+1$;
\item[(b)] the map $\ell_{W}: W \longrightarrow B$ is $k$-connected.  
\end{enumerate}
Similarly, the subspace $\mathcal{D}^{\partial, \mf, k}_{\theta, d+1} \subset \mathcal{D}^{\partial, \mf}_{\theta, d+1}$ consists of those $W$ subject the following further conditions:
\begin{enumerate} \itemsep.2cm
\item[(a)] all critical points $c$ of the height functions $h_{W}$ and $h_{W}|_{\partial W}$ satisfy $k < \text{index}(c) < d-k+1$;
\item[(b)] the maps $\ell_{W}: W \longrightarrow B$ and $\ell_{W}|_{\partial W}: W \longrightarrow B$ are both $k$-connected.  
\end{enumerate}
\end{defn}

\begin{remark} \label{remark: morse functions on manifolds with boundary}
Let $W \in \mathcal{D}^{\partial, \mf}_{\theta, d+1}$.
In the above definition, we allow for the height function $h_{W}: W \longrightarrow \R$ to have critical points that occur on the boundary, $\partial W$, as well as on the interior. 
This is the reason for defining $\mathcal{D}^{\partial, \mf}_{\theta, d+1}$ to be a subspace of $\mathcal{D}^{\partial, \pitchfork}_{\theta, d+1}$ rather than a subspace of $\mathcal{D}^{\partial}_{\theta, d+1}$.
Indeed, if the function $h_{W}: W \longrightarrow \R$ has a critical point on the boundary $\partial W$, then it is impossible 
for the boundary of $W$ to be equipped with a collar as in condition (ii) of Definition \ref{defn: space of long manifolds with boundary} (however, it is still possible for $W$ to intersect $\R\times\{0\}\times\R^{\infty-1}$ transversally).
We refer the reader to \cite{R 16} for basic definitions and results regarding Morse functions on manifolds with boundary. 
We won't have to go too deep into the Morse theory for manifolds with boundary in this paper.
\end{remark}

\subsection{Thom spectra and the homotopy type of $\mathcal{D}^{\mf, k}_{\theta, d+1}$} \label{subsection: thom spectra}
We now recall the main theorem from \cite{P 17} which determines the homotopy type of the spaces $\mathcal{D}^{\mf, k}_{\theta, d+1}(P)$ for $k < d/2$ and $P \in \mathcal{M}_{\theta, d-1}$ (where $P$ is possibly the empty set).
We need to define some Thom spectra. 
\begin{defn} \label{defn: hW spectrum}
Fix $N, k \in \N$ with $k \leq d/2$. 
We let $G_{\theta, d+1}(\R^{d+1+N})$ denote the Grassmannian manifold consisting of $(d+1)$-dimensional vector subspaces $V \leq \R^{d+1+N}$, equipped with a $\theta$-orientation $\hat{\ell}_{V}$. 
$G_{\theta, d+1}(\R^{\infty})$ is defined to be the direct limit of the spaces $G_{\theta, d+1}(\R^{d+1+N})$ taken as $N \to \infty$. 
The space $G^{\mf}_{\theta, d+1}(\R^{d+1+N})$ consists of tuples $(V, l, \sigma)$ where:
\begin{itemize} \itemsep.2cm
\item  $V \in G_{\theta, d+1}(\R^{d+1+N})$;
\item $l: V \longrightarrow \R$ is a linear functional;
\item $\sigma: V\otimes V \longrightarrow \R$ is a bilinear form;
\end{itemize}
subject to the \textit{Morse condition}: if $l = 0$, then $\sigma$ is symmetric and non-degenerate. 
Let $k$ be an integer. 
The subspace $G^{\mf, k}_{\theta, d+1}(\R^{d+1+N}) \subset G^{\mf}_{\theta, d+1}(\R^{d+1+N})$ consists of those $(V, l, \sigma)$ that satisfy:
\begin{itemize}
\item if $l = 0$, then $\sigma$ is symmetric and non-degenerate with $k < \text{index}(\sigma) < d-k+1$. 
\end{itemize}
The space $G^{\mf, k}_{\theta, d+1}(\R^{\infty})$ is the direct limit of the spaces $G^{\mf, k}_{\theta, d+1}(\R^{d+1+N})$ taken as $N \to \infty$.
\end{defn}
Let $V_{d+1, N} \longrightarrow G_{\theta, d+1}(\R^{d+1+N})$ denote the canonical vector bundle and let $V^{\perp}_{d+1, N} \longrightarrow G_{\theta, d+1}(\R^{d+1+N})$ denote the orthogonal complement bundle.
The spectrum $\MT\theta_{d+1}$ has for its $(d+1+N)$-th space the Thom space $\Th(V^{\perp}_{d+1, N})$. 
The structure map
$
S^{1}\wedge\Th(V^{\perp}_{d+1, N}) \longrightarrow \Th(V^{\perp}_{d+1, N+1})
$
is the map induced by the bundle map 
$
\epsilon^{1}\oplus V^{\perp}_{d+1, N} \longrightarrow V^{\perp}_{d+1, N+1}
$
covering the standard inclusion $G_{\theta, d+1}(\R^{d+1+N}) \hookrightarrow G_{\theta, d+1}(\R^{d+1+N+1})$.
It follows that the spectrum $\MT\theta_{d+1}$ is homotopy equivalent to the Thom spectrum of the virtual bundle $-V_{d+1, \infty}$ over $G_{\theta, d+1}(\R^{\infty})$.

Analogous to the situation above, we let $U^{k}_{d+1, N} \longrightarrow G^{\mf, k}_{\theta, d+1}(\R^{d+1+N})$ denote the canonical bundle, i.e.\  the vector bundle obtained by pulling back $V_{d+1, N} \longrightarrow G_{\theta, d+1}(\R^{d+1+N})$ over the forgetful map $G^{\mf, k}_{\theta, d+1}(\R^{d+1+N}) \longrightarrow G_{\theta, d+1}(\R^{d+1+N})$. 
Similarly, we let $U^{k, \perp}_{d+1, N} \longrightarrow G^{\mf, k}_{\theta, d+1}(\R^{d+1+N})$ denote the orthogonal bundle. 
We define $\mb{hW}^{k}_{\theta, d+1}$ to be the spectrum with $(d+1+N)$-th space given by $\Th\left(U^{k, \perp}_{d+1, N}\right)$ and structure maps defined in the same way as with $\MT\theta_{d+1}$.
For the case $k = -1$, this is the main spectrum considered in \cite{MW 07}.

The proceeding proposition follows by combining the results of Galatius, Madsen, Tillmann, and Weiss \cite{GMTW 08}, with the work of Genauer \cite{G 12}.
\begin{proposition}  \label{proposition: boundary map long manifold fibre sequence}
For all $d \in \Z_{\geq 0}$ and choices of $P \in \mathcal{M}_{\theta, d-1}$,
there is a commutative diagram, 
\begin{equation} \label{equation: commutative diagram of scanning maps}
\xymatrix{
\mathcal{D}_{\theta, d+1}(P) \ar[r] \ar[d]_{\simeq} & \mathcal{D}^{\partial}_{\theta, d+1} \ar[r]^{r} \ar[d]_{\simeq} & \mathcal{D}_{\theta, d} \ar[d]_{\simeq} \\
\Omega^{\infty-1}\MT\theta_{d+1} \ar[r] &  \Omega^{\infty-1}\Sigma^{\infty}B_{+} \ar[r] & \Omega^{\infty-1}\MT\theta_{d},
}
\end{equation}
rows are fibre-sequences and the vertical maps are weak homotopy equivalences.
\end{proposition}
The theorem stated below is our main result from \cite{P 17}.
\begin{theorem} \label{theorem: homotopy type of space of long manifolds}
Let $k < d/2$ and let $P \in \mathcal{M}_{\theta, d-1}$ be null-bordant as a $\theta$-manifold.
Suppose that the space $B$ satisfies Wall's finiteness condition $F(k)$. 
Then there is a weak homotopy equivalence,
$$\mathcal{D}^{\mf, k}_{\theta, d+1}(P) \stackrel{\simeq} \longrightarrow \Omega^{\infty-1}\mb{hW}^{k}_{\theta, d+1}.$$
\end{theorem}
We remark that in the special case $k = -1$, the above weak homotopy equivalence was proven by Madsen and Weiss in \cite{MW 07}.

\subsection{$\KO$-orientations} \label{subsection: ko orientation}
Fix a tangential structure $\theta: B \longrightarrow BO$. 
We will assume for the rest of this section that the space $B$ is $2$-connected. 
This implies that the map $\theta$ factors (up to homotopy) through the map $B\Spin \longrightarrow BO$. 
For latter use it will be important to specify $\KO$-theory classes in the spectra $\mb{hW}_{\theta, d}$ and $\MT\theta_{d}$.
By our connectivity assumption on $B$ it follows that for each $N \in \N$, the bundle $V_{d, N}^{\perp} \longrightarrow G_{\theta, d}(\R^{d+N})$ has a $\Spin(d)$-structure, and thus the Thom space $\Th(U_{d, N}^{\perp})$ is equipped with a $\KO$-orientation, $\lambda_{V^{\perp}_{d, N}} \in \KO^{N}(\Th(U_{d, N}^{\perp}))$. 
These Thom classes fit together to form a spectrum $\KO$-theory class $\lambda_{-d} \in \KO^{-d}(\MT\theta_{d})$, or in other words they yield a spectrum map 
$
\lambda_{-d}: \MT\theta_{d} \longrightarrow \Sigma^{-d}\KO.
$
For any integer $i \in \Z$, we may apply the functor $\Omega^{\infty+i}(\--)$ to this map. 
We the resulting infinite loop map by,
$$
\mathcal{A}^{i}_{d} = \Omega^{\infty+i}\lambda_{-d}: \Omega^{\infty+i}\MT\theta_{d} \longrightarrow \Omega^{\infty+d+i}\KO.
$$
We may apply the exact same construction to the Thom spaces of the bundles 
$$U^{k, \perp}_{d+1, N} \longrightarrow G^{\mf, k}_{\theta, d+1}(\R^{d+1+N})$$ 
to obtain a spectrum map $\hat{\lambda}_{-d}: \mb{hW}_{\theta, d}^{k} \longrightarrow \Sigma^{-d}\KO$, and infinite loop maps,
$$
\widehat{\mathcal{A}}^{i}_{d} = \Omega^{\infty+i}\hat{\lambda}_{-d}: \Omega^{\infty+i}\mb{hW}_{\theta, d}^{k} \longrightarrow \Omega^{\infty+d+i}\KO
$$
for all integer $i$.
We now describe how these $\KO$-theory classes $\mathcal{A}^{i}_{d}$ relate to each other for different values of $d$. 
The inclusion, 
$G_{\theta, d}(\R^{d+N}) \hookrightarrow G_{\theta, d+1}(\R^{d+1+N}),$ $V \mapsto \R\times V,$
is covered by a bundle map $V^{\perp}_{d, N} \longrightarrow V^{\perp}_{d+1, N}$.
It thus induces a map of spectra
$\Sigma^{-1}\MT\theta_{d} \longrightarrow \MT\theta_{d+1}$, and for each $i \in \Z$ an infinite loopmap,
$
\iota^{i}_{d}: \Omega^{\infty+i}\MT\theta_{d} \longrightarrow \Omega^{\infty+i-1}\MT\theta_{d+1}. 
$
By how the map $\mathcal{A}^{i}_{d}$ was defined, it follows that 
\begin{equation} \label{equation: compatibility of k-theory orientations}
\mathcal{A}^{i-1}_{d+1}\circ\iota^{i}_{d} \; = \; \mathcal{A}^{i}_{d}
\end{equation}
for all $i$ and $d$.
Now, consider the map 
$
G_{\theta, d}(\R^{d+N}) \longrightarrow G^{\mf, k}_{\theta, d+1}(\R\times\R^{d+N}),
$
defined by sending $V \in G_{\theta, d}(\R^{\infty})$ to the element $(\R\times V, l_{0}, \sigma_{0})$, where $\R\times V \subset \R\times\R^{d+N}$ is the product vector space, $l_{0}: \R\times V \longrightarrow \R$ is the projection onto the first coordinate, and $\sigma_{0}$ is the trivial bilinear form. 
This map is covered by a bundle map $V^{\perp}_{d, N} \longrightarrow U^{\perp}_{d+1, N}$, and thus induces a map of spectra $\Sigma^{-1}\MT\theta_{d} \longrightarrow \mb{hW}^{k}_{\theta, d+1}$, and for each $i$ an infinite loop map,
$
\Omega^{\infty+i}\MT\theta_{d} \longrightarrow \Omega^{\infty+i-1}\mb{hW}^{k}_{\theta, d+1}.
$
It is easily verified that the map $\iota^{i}_{d}$ agrees with the composite,
$$
\xymatrix{
\Omega^{\infty+i}\MT\theta_{d} \ar[r] & \Omega^{\infty+i-1}\mb{hW}^{k}_{\theta, d+1} \ar[r] & \Omega^{\infty+i-1}\MT\theta_{d+1}. 
}
$$
The following proposition follows directly from the definitions of the maps described above.
\begin{proposition} \label{proposition: commutativity of k-theory orientations}
For all $k$, $i$, and $d$, the following diagram is commutative,
$$
\xymatrix{
\Omega^{\infty+i}\MT\theta_{d} \ar[r] \ar[dr]_{\mathcal{A}^{i}_{d} \ \ \ \ } & \Omega^{\infty+i-1}\mb{hW}^{k}_{\theta, d+1} \ar[r] \ar[d]^{\widehat{\mathcal{A}}^{i-1}_{d+1}} & \Omega^{\infty+i-1}\MT\theta_{d+1} \ar[dl]^{\mathcal{A}^{i-1}_{d+1} \ \ \ \ } \\
& \Omega^{\infty+d+i}\KO.
}
$$
\end{proposition}

\section{Spaces of Manifolds Equipped with Surgery Data} \label{section: Spaces of manifolds equipped with surgery}
\subsection{Spaces of manifolds equipped with surgery data} \label{subsection: spaces of manifolds equipped with surgery data}
In this section we recall the alternative model for the cobordism category $\Cob_{\theta, d+1}^{\mf, k}$ from \cite[Section 4]{P 17} (first conceived of in \cite{MW 07}).
Though the spaces in this section have already been defined in that paper, we will need to invoke the details of the construction many times latter on and so for the sake of completeness, we find it necessary to give the definitions again here. 
Plus, we also use slightly different notation than in \cite{P 17}.
For what follows, fix once and for all an infinite set $\Omega$. 
\begin{defn} \label{defn: surgery category}
Let $d \in \Z_{\geq 0}$. 
An object of the category $\mathcal{K}_{d}$ is a finite subset $\mb{t} \subset \Omega$ equipped with a map, 
$\delta: \mb{t} \longrightarrow \{0, 1, 2, \dots, d+1\}.$
A morphism from $\mb{s}$ to $\mb{t}$ is a pair $(j, \varepsilon)$ where $j$ is an injective map over $\{0, \dots, d+1\}$ from $\mb{s}$ to $\mb{t}$, and $\varepsilon$ is a function $\mb{t}\setminus j(\mb{s}) \longrightarrow \{-1, +1\}$. 
The composition of two morphisms $(j_{1}, \varepsilon_{1}): \mb{s} \longrightarrow \mb{t}$ and $(j_{2}, \varepsilon_{2}): \mb{t} \longrightarrow \mb{u}$ is $(j_{2}j_{1}, \varepsilon_{3})$, where $\varepsilon_{3}$ agrees with $\varepsilon_{2}$ outside $j_{2}(\mb{t})$ and $\varepsilon_{1}\circ j^{-1}_{2}$ on $j_{2}(\mb{t}\setminus j_{1}(\mb{s}))$. 

We will need to work some subcategories of $\mathcal{K}_{d}$.
Let $k \leq d/2$ be an integer. 
The subcategory $\mathcal{K}^{k}_{d} \subset \mathcal{K}_{d}$ is the full subcategory on those objects $\mb{t}$ for which $k < \delta(t) < d-k+1$ for all $t \in \mb{t}$. 
\end{defn}

We will ultimately define space-valued contravariant functors on the categories $\mathcal{K}^{k}_{d}$. 
To get to our main definition (Definition \ref{defn: colimit decomp of W}) we will need to fix some notation.  
We define the space $G^{\mf}_{\theta, d+1}(\R^{\infty})_{\loc}$ to consist of pairs $(V, \sigma)$ where:
\begin{enumerate} \itemsep.2cm
\item[(i)] $V \leq \R^{\infty}$ is a $(d+1)$-dimensional vector subspace equipped with a $\theta$-orientation $\hat{\ell}_{V}$; 
\item[(ii)] $\sigma: V\otimes V \longrightarrow \R$ is a non-degenerate, symmetric bilinear form. 
\end{enumerate}
Let $(V, \sigma) \in G^{\mf}_{\theta, d+1}(\R^{\infty})_{\loc}$. 
We denote by $V^{\pm} \leq V$ the positive and negative eigenspaces associated to $\sigma$. 
We denote by $D(V^{\pm})$ and $S(V^{\pm})$ the unit spheres and unit disks associated to $V^{\pm}$, defined using the inner product induced by the ambient space $\R^{\infty}$.
Let $\mb{t} \in \mathcal{K}_{d}$. 
We will need to work with the mapping space 
$$
(G^{\mf}_{\theta, d+1}(\R^{\infty})_{\loc})^{\mb{t}} \; = \; \Maps(\mb{t}, G^{\mf}_{\theta, d+1}(\R^{\infty})_{\loc}).
$$
Elements of $(G^{\mf}_{\theta, d+1}(\R^{\infty})_{\loc})^{\mb{t}}$ will be denoted by $(V, \sigma)$ where $(V, \sigma)(t) = (V(t), \sigma(t))$ for $t \in \mb{t}$.
Let $(V, \sigma) \in (G^{\mf}_{\theta, d+1}(\R^{\infty})_{\loc})^{\mb{t}}$.
We denote 
\begin{equation} \label{equation: sphere-disk bundle over t}
D(V^{-})\times_{\mb{t}}D(V^{+}) \; = \; \coprod_{t \in \mb{t}}D(V^{-}(t))\times D(V^{+}(t)).
\end{equation}
The spaces $S(V^{-})\times_{\mb{t}}D(V^{+})$, $D(V^{-})\times_{\mb{t}}S(V^{+})$, and $S(V^{-})\times_{\mb{t}}S(V^{+})$ are defined similarly.
The $\theta$-orientation $\hat{\ell}_{V}$ on $V$ induces a $\theta$-structure on $D(V^{-})\times_{\mb{t}}D(V^{+})$ which we denote by $\hat{\ell}_{V, \mb{t}}$. 

Fix $P \in \mathcal{M}_{\theta, d-1}$. 
For an integer $k$, let $\mathcal{M}_{\theta, d}^{k}(P) \subset  \mathcal{M}_{\theta, d}(P)$ be the subspace consisting of those $M$ for which the map $\ell_{M}: M \longrightarrow B$ is $k$-connected. 
It easy to see that $\mathcal{M}_{\theta, d}^{k}(P)$ is a collection of path-components of $\mathcal{M}_{\theta, d}(P)$.
With the above notation and terminology in place, we are now ready to present the main definition. 
\begin{defn} \label{defn: colimit decomp of W}
Let $k \leq d/2$ be an integer. 
Fix $P \in \mathcal{M}_{\theta, d-1}$ and let $\mb{t} \in \mathcal{K}^{k}_{d}$. 
The space $\mathcal{W}^{k}_{\theta, d}(P; \mb{t})$ consists of tuples $(M, (V, \sigma), e)$  where:
\begin{enumerate} \itemsep.3cm
\item[(i)] $M$ is an element of $\mathcal{M}^{k}_{\theta, d}(P)$.
\item[(ii)] The element $(V, \sigma) \in (G^{\mf}_{\theta, d+1}(\R^{\infty})_{\loc})^{\mb{t}}$ satisfies
$\text{index}(\sigma(t)) = \delta(t)$ for all $t \in \mb{t}$. 
\item[(iii)] $e: D(V^{-})\times_{\mb{t}}D(V^{+}) \longrightarrow (-\infty, 0)\times\R^{\infty-1}$ is an embedding that satisfies,
$$
e^{-1}(M) \; = \; S(V^{-})\times_{\mb{t}}D(V^{+}).
$$
\item[(iv)] The $\theta$-structure $\hat{\ell}_{V, \mb{t}}$ on $D(V^{-})\times_{\mb{t}}D(V^{+})$ agrees with $\hat{\ell}_{M}$ when restricted to $M$. 
\end{enumerate}
In the case that $P = \emptyset$, we will denote, $\mathcal{W}^{k}_{\theta, d}(\mb{t}) := \mathcal{W}^{k}_{\theta, d}(\emptyset; \mb{t})$.
\end{defn}

We need to review how to make $\mb{t} \mapsto \mathcal{W}^{k}_{\theta, d}(P; \mb{t})$ into a contravariant space-valued functor on $\mathcal{K}^{k}_{d}$. 
Let $(j, \varepsilon): \mb{s} \longrightarrow \mb{t}$ be a morphism in $\mathcal{K}^{k}_{d}$. 
If $j$ is bijective, there is an obvious identification $\mathcal{W}^{k}_{\theta, d}(P; \mb{t}) \cong \mathcal{W}^{k}_{\theta, d}(P; \mb{s})$ and this bijection is the induced map. 
To define $(j, \varepsilon)_{*}$, we may assume that $j$ is an inclusion $\mb{s} \hookrightarrow \mb{t}$. 
We may then reduce to the case where $\mb{t}\setminus \mb{s}$ has exactly one element, $a$. 
This case has two subcases: $\varepsilon(a) = +1$ and $\varepsilon(a) = -1$. 

\begin{defn}[$\varepsilon(a) = +1$]  \label{defn: pullback +1}
Let $(j, \varepsilon): \mb{s} \longrightarrow \mb{t}$ be as above and suppose that $\varepsilon(a) = +1$. 
The induced map 
$(j, \varepsilon)^{*}: \mathcal{W}^{k}_{\theta, d}(P; \mb{t}) \longrightarrow \mathcal{W}^{k}_{\theta, d}(P; \mb{s})$
is defined by sending $(M, (V, \sigma), e)$ to $(M, (V', \sigma'), e')$ where $e'$ is obtained by restricting $e$ to the subspace, 
$D(V^{-})\times_{\mb{s}}D(V^{+})  \subset D(V^{-})\times_{\mb{t}}D(V^{+}),$
and $(V', \sigma')$ is obtained by restricting $(V, \sigma)$ to $\mb{s}$.
\end{defn}

\begin{defn}[$\varepsilon(a) = -1$] \label{defn: pullback -1} 
Let $(j, \varepsilon): \mb{s} \longrightarrow \mb{t}$ be as above and suppose that $\varepsilon(a) = -1$. 
The induced map 
$
\mathcal{W}^{k}_{\theta, d}(P; \mb{t}) \longrightarrow \mathcal{W}^{k}_{\theta, d}(P; \mb{s})
$
is defined as follows:
Let $(M, (V, \sigma), e)$ be an element of $\mathcal{W}^{k}_{\theta, d}(P; \mb{t})$.
Map this to the element $(\widetilde{M}, (\widetilde{V}, \widetilde{\sigma}), \widetilde{e})$ in $\mathcal{W}^{k}_{\theta, d}(P; \mb{s})$ where: 
\begin{enumerate} \itemsep.2cm
\item[(i)] $\widetilde{M}$ is the element of $\mathcal{M}^{k}_{\theta, d}(P)$ given by, 
$$
\widetilde{M} \; = \; (M\setminus e(S(V^{-})\times_{a}D(V^{+}))\bigcup e(D(V^{-})\times_{a}S(V^{+})),
$$
where $S(V^{-})\times_{a}D(V^{+}) \subset S(V^{-})\times_{\mb{t}}D(V^{+})$ is the component of $S(V^{-})\times_{\mb{t}}D(V^{+})$ that corresponds to $a \in \mb{t}$.
\item[(ii)] $\widetilde{e}$ is obtained from $e$ by restriction to $D(V^{-})\times_{\mb{s}}D(V^{+})$. 
\item[(iii)] $(\widetilde{V}, \widetilde{\sigma})$ is obtained from $(V, \sigma)$ by restriction to $\mb{s}$. 
\end{enumerate}
The definitions above make the assignment $\mb{t} \mapsto \mathcal{W}^{k}_{\theta, d}(P; \mb{t})$ into a contravariant functor on $\mathcal{K}^{k}_{d}$.
We will be interested in the homotopy colimit of the spaces $\mathcal{W}^{k}_{\theta, d}(P; \mb{t})$ taken over $\mb{t} \in \mathcal{K}_{d}$. 
We denote
$$
\mathcal{W}^{k}_{\theta, d}(P) \; := \; \hocolim_{\mb{t} \in \mathcal{K}^{k}_{d}}\mathcal{W}^{k}_{\theta, d}(P; \mb{t}).
$$
In the case that $P$ is the empty set we will denote, $\mathcal{W}^{k}_{\theta, d} := \mathcal{W}^{k}_{\theta, d}(\emptyset)$.
\end{defn}
In the appendix of \cite{P 17} we proved that for all $k$ and $P$ there is a zig-zag of weak homotopy equivalences, 
\begin{equation} \label{equation: W to L to D zig zag}
\xymatrix{
\mathcal{W}^{k}_{\theta, d}(P) & \mathcal{L}^{k}_{\theta, d}(P) \ar[l]_{ \ \ \ \simeq} \ar[r]^{\simeq \ \ \ } & \mathcal{D}^{\mf, k}_{\theta, d+1}(P),
}
\end{equation}
and thus a weak homotopy equivalence, 
$
\mathcal{W}^{k}_{\theta, d}(P) \simeq B\Cob_{\theta, d+1}^{\mf, k}.
$
For this reason we consider the transport category $\mathcal{K}^{k}_{d}\wr\mathcal{W}^{k}_{\theta, d}(P, \--)$ to be an alternative model for the cobordism category $\Cob_{\theta, d+1}^{\mf, k}$. 

There is a particular sub-functor of $\mathcal{W}^{k}_{\theta, d}(P; \--)$ that we will also need to work with. 
\begin{defn} \label{defn: W-conn}
Let $k \in \Z_{\geq -1}$. 
Let $P \in \mathcal{M}_{\theta, d-1}$. 
For $\mb{t} \in \mathcal{K}_{d}^{k}$, the subspace 
$$\mathcal{W}^{k, c}_{\theta, d}(P; \mb{t}) \subset \mathcal{W}^{k}_{\theta, d}(P; \mb{t})$$
consists of those $(M, (V, \sigma), e)$ for which the map, 
$
\ell_{M}|_{M\setminus \Image(e)}: M\setminus \Image(e)\; \longrightarrow \; B,
$
is $(k+1)$-connected. 
It is easily verified that the correspondence $\mb{t} \mapsto \mathcal{W}^{k, c}_{\theta, d}(P; \mb{t})$ defines a sub-functor of $\mb{t} \mapsto \mathcal{W}^{k}_{\theta, d}(P; \mb{t})$, for all choices of $P$.
We define,
$
\mathcal{W}^{k, c}_{\theta, d}(P) \; := \; \displaystyle{\hocolim_{\mb{t} \in \mathcal{K}^{k}_{d}}}\mathcal{W}^{k, c}_{\theta, d}(P; \mb{t}).
$
\end{defn}

The following theorem is a restatement of \cite[Theorem 4.1]{P 17}.
\begin{theorem} \label{theorem: parametrized surgery}
Let $k < d/2$. 
Suppose that $\theta: B \longrightarrow BO$ is chosen so that the space $B$ satisfies Wall's finiteness condition $F(k)$.
For all $P \in \mathcal{M}_{\theta, d-1}$, the inclusion $\mathcal{W}^{k, c}_{\theta, d}(P) \hookrightarrow \mathcal{W}^{k}_{\theta, d}(P)$ is a weak homotopy equivalence. 
\end{theorem}

\subsection{The space $\mathcal{W}^{k}_{\theta, d, \loc}$}
For some arguments latter on we will also need to use the ``local version'' of the functor $\mathcal{W}^{k}_{\theta, d}(\--)$ from \cite{P 17}.
We review the definition below. 
For $k \in \Z_{\geq 0}$, let $\mathcal{K}^{\{k, d-k+1\}}_{d} \subset \mathcal{K}_{d}$ denote the subcategory consisting of those $\mb{t}$ for which $\delta(\mb{t}) \in \{k, d-k+1\}$. 

\begin{defn} \label{defn: W-loc}
Let $k \leq d/2$ be an integer. 
Let $\mb{t} \in \mathcal{K}^{\{k, d-k+1\}}_{d}$. 
The space $\mathcal{W}^{\{k, d-k+1\}}_{\theta, d, \loc}(\mb{t})$ consists of tuples $((V, \sigma), e)$  where:
\begin{enumerate} \itemsep.3cm
\item[(i)] $(V, \sigma) \in (G^{\mf}_{\theta, d+1}(\R^{\infty})_{\loc})^{\mb{t}}$ is an element that satisfies $\text{index}(\sigma(t)) = \delta(t)$ for all $t \in \mb{t}$; 
\item[(ii)] $e: D(V^{-})\times_{\mb{t}}D(V^{+}) \longrightarrow \R^{\infty}$ is an embedding.
\end{enumerate}
The assignment $\mb{t} \mapsto \mathcal{W}^{\{k, d-k+1\}}_{\theta, d, \loc}(\mb{t})$ defines a functor on $\mathcal{K}^{\{k, d-k+1\}}_{d}$. 
\end{defn}
Consider the map,
\begin{equation} \label{equation: localization transformation}
\mb{L}_{\mb{t}}: \mathcal{W}^{k-1}_{\theta, d}(P, \mb{t}) \longrightarrow \mathcal{W}^{\{k, d-k+1\}}_{\theta, d, \loc}(\mb{t}), \quad
(M, (V, \sigma), e) \mapsto ((V, \sigma)|_{\delta^{-1}(\{k, d-k+1\})}, \; e),
\end{equation}
where $(V, \sigma)|_{\{k, d-k+1\}}$ is the restriction of $(V, \sigma)$ to the subset $\delta^{-1}(\{k, d-k+1\}) \subset \mb{t}$. 
By precomposing the functor $\mb{t} \mapsto \mathcal{W}^{\{k, d-k+1\}}_{\theta, d, \loc}(\mb{t})$ with the projection $\mathcal{K}^{k-1}_{d} \longrightarrow \mathcal{K}^{\{k, d-k+1\}}_{d}$, we may consider $\mathcal{W}^{\{k, d-k+1\}}_{\theta, d, \loc}(\--)$ to be a functor on $\mathcal{K}^{k-1}_{d}$. 
In this way (\ref{equation: localization transformation}) may be considered to be a natural transformation. 
The following theorem is proven in \cite{P 17}.
Observe that the empty element $\emptyset \in \mathcal{W}^{\{k, d-k+1\}}_{\theta, d, \loc}(\emptyset)$ determines a natural basepoint in the homotopy colimit, $\mathcal{W}^{\{k, d-k+1\}}_{\theta, d, \loc}$.
\begin{theorem} \label{theorem: localization map homotopy equivalence}
Let $k < d/2$. 
For all $P \in \mathcal{M}_{\theta, d-1}$, the map induced by (\ref{equation: localization transformation}), 
$$
\mb{L}: \mathcal{W}^{k-1, c}_{\theta, d}(P) \longrightarrow \mathcal{W}^{\{k, d-k+1\}}_{\theta, d, \loc},
$$
is a quasi-fibration. 
The fibre over $\emptyset \in \mathcal{W}^{\{k, d-k+1\}}_{\theta, d, \loc}$ is homeomorphic to the space $\mathcal{W}^{k}_{\theta, d}(P)$.
\end{theorem}

\section{Spaces of Manifolds Equipped with PSC Metrics} \label{section: Spaces of Manifold Equipped with PSC Metrics and Surgery Data}
\subsection{Spaces of metrics} \label{defn: preliminary definitions psc}
We begin by giving some preliminary constructions involving positive scalar curvature metrics. 
\begin{defn} \label{defn: space of metrics of positive scalar curvature}
Let $M$ be a compact manifold equipped with a collar, 
$h: \partial M\times[0, \infty) \hookrightarrow M.$ 
We denote by $\mathcal{R}^{+}(M)$ the space of all $\psc$ metrics $g$ on $M$, with the property that $h^{*}g = g_{0} + dt^{2}$ for some $\psc$ metric $g_{0}$ on $\partial M$.
Let $\phi$ be a metric on $\partial M$ with positive scalar curvature.
We define $\mathcal{R}^{+}(M)_{\phi} \subset \mathcal{R}^{+}(M)$ to be the subspace consisting of those $g$ with the property that $h^{*}g = \phi + dt^{2}$. 
\end{defn}

Let $W$ be a $d$-dimensional compact manifold with boundary $\partial W = M$. 
There is a restriction map 
$
r: \mathcal{R}^{+}(W) \longrightarrow \mathcal{R}^{+}(M),$
$g \mapsto g|_{M}.$
The following proposition is proven in \cite{Ch 06} and will be very important for certain arguments later on in the paper. 
\begin{proposition} \label{proposition: chernysh restriction theorem}
The restriction map $r: \mathcal{R}^{+}(W) \longrightarrow \mathcal{R}^{+}(M)$ is a quasi-fibration. 
The fibre over $h \in \mathcal{R}^{+}(M)$ is homeomorphic to the space of metrics 
$\mathcal{R}^{+}(W)_{h}$.
\end{proposition}

Fix a tangential structure $\theta: B \longrightarrow BO$. 
We will also need to work with spaces of $\theta$-manifolds equipped with a $\psc$-metric. 
\begin{defn} \label{defn: space of manifolds equipped with metric}
Let $d \in \Z_{\geq 0}$. 
The space $\mathcal{M}^{\psc}_{\theta, d}$ consists of pairs $(M, g)$ where $M \in \mathcal{M}_{\theta, d}$ and $g \in \mathcal{R}^{+}(M)$.
Now, fix an element $(P, g_{P}) \in \mathcal{M}^{\psc}_{\theta, d-1}$. 
The space $\mathcal{M}^{\psc}_{\theta, d}(P, g_{P})$ consists of pairs $(M, g)$ where $M \in \mathcal{M}_{\theta, d}(P)$ and $g \in \mathcal{R}^{+}(M)_{g_{P}}$. 
Following Definition \ref{defn: space of closed manifolds}, the space $\mathcal{M}^{\psc}_{\theta, d}$ is topologized as the quotient space,
\begin{equation} \label{equation: topology of moduli spaces}
\mathcal{M}^{\psc}_{\theta, d} \cong \coprod_{[P]}\left(\Emb(P, \R^{\infty})\times\Bun(TP, \theta^{*}_{d}\gamma^{d})\times\mathcal{R}^{+}(P)\right)/\Diff(P).
\end{equation}
The spaces $\mathcal{M}^{\psc}_{\theta, d}(P, g_{P})$ are topologized similarly. 
\end{defn}

\subsection{The parametrized Gromov-Lawson construction} \label{subsection: Walsh-Chernysh theorem}
Our constructions will draw heavily on results of Chernysh \cite{Ch 04} and  Walsh \cite{Wa 13}, based on a parametrized version of the Gromov-Lawson construction from \cite{GL 80}. 
We review their results here. 
Let $g_{S^{k-1}}$ denote the standard \textit{round metric} on the $(k-1)$-dimensional sphere $S^{k-1}$.
Denote by $g^{d-k}_{\tor}$ the \textit{torpedo metric} on the $(d-k)$-dimensional disk $D^{d-k}$ (see \cite{Wa 13} for its definition).
Consider the product metric $g_{S^{k-1}}+ g^{d-k+1}_{\tor}$ on the product $S^{k-1}\times D^{d-k+1}$. 
Both $g_{S^{k-1}}$ and $g^{d-k+1}_{\tor}$ have positive scalar curvature and thus $g_{S^{k-1}}+ g^{d-k+1}_{\tor}$ has positive scalar curvature as well.
Let $M$ be a compact $d$-dimensional manifold and let $g_{\partial}$ be a $\psc$ metric on the boundary $\partial M$. 
Choose an embedding 
$\phi: S^{k-1}\times D^{d-k+1} \longrightarrow \Int(M).$ 
The subspace 
\begin{equation} \label{equation: standard metrics}
\mathcal{R}^{+}(M, \phi)_{g_{\partial}} \subset \mathcal{R}^{+}(M)_{g_{\partial}} 
\end{equation} 
is defined to consist of all $\psc$ metrics $g \in \mathcal{R}^{+}(M)_{g_{\partial}}$ that satisfy $\phi^{*}g = g_{S^{k-1}}+ g^{d-k+1}_{\tor}$.
The following result is due to Walsh \cite{Wa 13} and Chernysh \cite{Ch 04}.
\begin{theorem} \label{theorem: walsh chernysh}
If $d-k+1 \geq 3$, then the inclusion 
$
\mathcal{R}^{+}(M; \phi)_{g_{\partial}} \hookrightarrow \mathcal{R}^{+}(M)_{g_{\partial}}
$
is a weak homotopy equivalence. 
\end{theorem}
The construction from \cite{Wa 13} used in the proof of Theorem \ref{theorem: walsh chernysh} is a local construction; it involves a deformation through $\psc$-metrics that is supported entirely on a tubular neighborhood of the embedded sphere (in particular see \cite[Theorem 3.10]{Wa 11}).
For this reason, there is no problem extending Theorem \ref{theorem: walsh chernysh} to work for a collection of embeddings  
$$\phi_{i}: S^{k_{i}-1}\times D^{d-k_{i}+1} \longrightarrow \Int(M), \quad i = 1, \dots, p$$ 
with $\phi_{i}(S^{k_{i}-1}\times D^{d-k_{i}+1})\cap\phi_{j}(S^{k_{j}-1}\times D^{d-k_{j}+1}) = \emptyset$ for all $i \neq j$.
For such embeddings, let
$$\mathcal{R}^{+}(M, \phi_{1}, \dots, \phi_{p})_{g_{\partial}} \subset \mathcal{R}^{+}(M)_{g_{\partial}}$$ 
be the subspace consisting of those $\psc$ metrics $g \in \mathcal{R}^{+}(M)_{g_{\partial}}$ such that $\phi^{*}_{i}g = g_{S^{k_{i}-1}}+ g^{d-k_{i}+1}_{\tor}$ for all $i = 1, \dots, p$. 
By the exact same technique used in the proof of Theorem \ref{theorem: walsh chernysh} we obtain:
\begin{Addendum} \label{Addendum: extension of Walsh}
Let $\phi_{1}, \dots, \phi_{p}$ be embeddings as above. 
Let $d-i+ 1 \geq 3$ for all $i = 1, \dots, p$. 
Then for any subset $\{i_{1}, \dots, i_{l}\} \subset \{1 \dots, p\}$ the inclusion 
$$
\mathcal{R}^{+}(M, \phi_{1}, \dots, \phi_{p})_{g_{\partial}} \hookrightarrow \mathcal{R}^{+}(M, \phi_{i_{1}}, \dots, \phi_{i_{l}})_{g_{\partial}}$$
is a weak homotopy equivalence. 
\end{Addendum}

We now study how the space $\mathcal{R}^{+}(M)_{g_{\partial}}$ changes upon changing $M$ by a surgery. 
Let 
$$\phi: S^{k-1}\times D^{d-k+1} \longrightarrow \Int(M)$$ 
be an embedding as in the statement of Theorem \ref{theorem: walsh chernysh}.
Let $\widetilde{M}$ denote the manifold obtained by performing surgery with respect to the embedding $\phi$, i.e.
$
\widetilde{M} = M\setminus\phi(S^{k-1}\times D^{d-k+1})\bigcup D^{k}\times S^{d-k}.
$
There is a map 
\begin{equation} \label{equation: surgery map}
S_{\phi}: \mathcal{R}^{+}(M; \phi)_{g_{\partial}} \longrightarrow \mathcal{R}^{+}(\widetilde{M})_{g_{\partial}}
\end{equation}
defined by setting $S_{\phi}(g)$ equal to the metric $g$ on the complement $M\setminus\phi(S^{k-1}\times D^{d-k+1})$, and equal to $g^{k}_{\tor} + g_{S^{d-k}}$ on $D^{k}\times S^{d-k}$.
\begin{theorem} \label{theorem: perform surgery map}
If $k \geq 3$, then the map
$
S_{\phi}: \mathcal{R}^{+}(M; \phi)_{g_{\partial}} \longrightarrow \mathcal{R}^{+}(\widetilde{M})_{g_{\partial}}
$
is a weak homotopy equivalence. 
\end{theorem}
\begin{proof}
By how the map $S(\phi)$ was defined it actually factors as a composite 
$$
\xymatrix{
\mathcal{R}^{+}(M; \phi)_{g_{\partial}} \ar[rr]^{\bar{S}(\phi)} && \mathcal{R}^{+}(\widetilde{M}; \bar{\phi})_{g_{\partial}} \ar@{^{(}->}[r] & \mathcal{R}^{+}(\widetilde{M})_{g_{\partial}}
}
$$
where $\bar{\phi}: D^{k}\times S^{d-k} \longrightarrow \widetilde{M}$ is the embedding dual to $\phi$, and the second map is the inclusion. 
The first map is a homeomorphism as explained in \cite{Wa 13} and the second map is a weak homotopy equivalence by Theorem \ref{theorem: walsh chernysh}. 
\end{proof}

\begin{remark} \label{remark: multiple disjoint surgeries}
We remark that since the construction used by Chernysh and Walsh is local, as with the situation in Addendum \ref{Addendum: extension of Walsh}, a similar result to Theorem \ref{theorem: perform surgery map} can be obtained for performing surgery on a collection of disjoint embeddings simultaneously. 
\end{remark}

\subsection{The functor $\mb{t} \mapsto \mathcal{W}^{\psc, k}_{\theta, d}(P; \mb{t})$} \label{subsection: W psc functor}
We will now use the theorems described above to construct the functor $\mb{t} \mapsto \mathcal{W}^{\psc, k}_{\theta, d}(P; \mb{t})$ that was discussed at the beginning of the section.
We need to first fix some notation.
\begin{defn} \label{defn: standard metrics}
Let $\mb{t} \in \mathcal{K}^{k}_{d}$ and $(V, \sigma) \in G^{\mf}_{\theta, d}(\R^{\infty})_{\loc}$.
We let $g^{\std, -}_{\mb{t}}$ denote the metric on the manifold $S(V^{-})\times_{\mb{t}}D(V^{+})$ with the following property:
\begin{itemize} \itemsep.2cm
\item for each $i \in \mb{t}$ the restriction of $g^{\std, -}_{\mb{t}}$ to 
$S(V^{-}(i))\times D(V^{+}(i)) = S^{\delta(i)-1}\times D^{d-\delta(i) + 1}$ 
agrees with the metric, 
$g_{S^{\delta(i)-1}} + g^{d-\delta(i)+1}_{\tor}.$
\end{itemize}
Similarly, we let $g^{\std, +}_{\mb{t}}$ denote the metric on $D(V^{-}(i))\times_{\mb{t}}S(V^{+}(i))$
with the following property:
\begin{itemize} \itemsep.2cm
\item for each $i \in \mb{t}$ the restriction of $g^{\std, +}_{\mb{t}}$ to $D(V^{-}(i))\times S(V^{+}(i))$ agrees with, 
$g^{\delta(i)}_{\tor} + g_{S^{d-\delta(i)}}.$
\end{itemize}
\end{defn}

For the following definition, fix an element $(P, g_{P}) \in \mathcal{M}^{\psc}_{\theta, d-1}$.
\begin{defn} \label{defn: metrics standard on multiple surgeries}
Let $(M, (V, \sigma), e) \in \mathcal{W}^{k}_{\theta, d}(P, \mb{t})$. 
Let 
$
e_{-}: S(V^{-})\times_{\mb{t}}D(V^{+}) \longrightarrow M
$
denote the embedding obtained by restricting $e$ to $S(V^{-})\times_{\mb{t}}D(V^{+})$.
The subspace 
$$\mathcal{R}^{+}(M; \; \mb{t}, e)_{g_{P}} \subset \mathcal{R}^{+}(M)_{g_{P}}$$
is defined to consist of all metrics $g \in \mathcal{R}^{+}(M)_{g_{P}}$ such that
$
e_{-}^{*}g \; = \; g^{\std, -}_{\mb{t}},
$
where $g^{\std, -}_{\mb{t}}$ is defined in Definition \ref{defn: standard metrics}.
We define the space 
$\mathcal{W}^{\psc, k}_{\theta, d}(P, g_{P}; \mb{t})$
to consist of all tuples 
$((M, (V, \sigma), e), g)$ 
such that 
$(M, (V, \sigma), e) \in \mathcal{W}^{k}_{\theta, d}(P, \mb{t})$ and $g \in \mathcal{R}^{+}(M; \; \mb{t}, e)_{g_{P}}.$ 
Similarly, the subspace 
$$\mathcal{W}^{\psc, k, c}_{\theta, d}(P, g_{P}; \mb{t}) \subset \mathcal{W}^{\psc, k}_{\theta, d}(P, g_{P}; \mb{t})$$
consists of those $((M, (V, \sigma), e), g)$ for which $(M, (V, \sigma), e)$ is contained in the subspace $\mathcal{W}^{k, c}_{\theta, d}(P; \mb{t})$.
\end{defn}

We now make the assignment $\mb{t} \mapsto \mathcal{W}^{\psc, k}_{\theta, d}(P, g_{P}; \mb{t}) $ into a contravariant functor on $\mathcal{K}^{k}_{d}$. 
We need to describe where to send morphisms of $\mathcal{K}_{d}^{k}$.
Let $(j, \varepsilon): \mb{s} \longrightarrow \mb{t}$ be a morphism in $\mathcal{K}_{d}^{k}$.
Let $\mb{u}$ denote the complement $\mb{t}\setminus j(\mb{s})$ and let $\mb{u}_{\pm}$ denote the subset of $\mb{u}$ consisting of all points $i$ with $\varepsilon(i) = \pm1$.
Let $M_{e, (j, \varepsilon)}$ denote the surgered manifold, 
\begin{equation}
M_{e, (j, \varepsilon)} \; = \; M\setminus e(S(V^{-})\times_{\mb{u_{-}}}D(V^{+}))\bigcup e(D(V^{-})\times_{\mb{u_{-}}}S(V^{+})).
\end{equation}
We let $e_{(j, \varepsilon)}$ denote the embedding obtained by restricting $e$ to $S(V^{-})\times_{\mb{s}}D(V^{+})$.

\begin{defn} \label{defn: functoriality for psc functor}
With $(M, (V, \sigma), e) \in \mathcal{W}^{k}_{\theta, d}(P, \mb{t})$ and $(j, \varepsilon): \mb{s} \longrightarrow \mb{t}$ as above, the map 
$$
S_{(j, \varepsilon)}: \mathcal{R}^{+}(M; \mb{t}, e)_{g_{P}} \; \longrightarrow \; \mathcal{R}^{+}(M_{e, (j, \varepsilon)}; \; \mb{s}, e_{(j, \varepsilon)})_{g_{P}}
$$
is defined by sending a metric $g \in \mathcal{R}^{+}(M; \mb{t}, e)_{g_{P}}$, to a new metric $g'$ determined by the following properties:
\begin{enumerate} \itemsep.2cm
\item[(i)] $g'$ is equal to $g$ on the complement
$M\setminus e(S(V^{-})\times_{\mb{u_{-}}}D(V^{+}));$
\item[(ii)] on $e(D(V^{-})\times_{\mb{u_{-}}}S(V^{+}))$ the metric $g'$ agrees with $g^{\std, +}_{\mb{u}}$, where $g^{\std, +}_{\mb{u}}$ is the metric defined in Definition \ref{defn: standard metrics}. 
\end{enumerate}
Given any morphism $(j, \varepsilon): \mb{s} \longrightarrow \mb{t}$ in $\mathcal{K}_{d}^{k}$, the induced map, 
$$
(j, \varepsilon)^{*}: \mathcal{W}^{\psc, k}_{\theta, d}(P, g_{P}; \mb{t}) \longrightarrow \mathcal{W}^{\psc, k}_{\theta, d}(P, g_{P}; \mb{s}),
$$
is defined by setting,
$
(j, \varepsilon)^{*}((M, (V, \sigma), e), \; g) =  \left((M_{e, (j,\varepsilon)}, (V, \sigma)|_{\mb{s}}, e_{(j,\varepsilon)}), \; S_{(j, \varepsilon)}(g)\right).
$
For any $k$, the above construction makes the assignment $\mb{t} \mapsto \mathcal{W}^{\psc, k}_{\theta, d}(P, g_{P}; \mb{t})$ into a contravariant functor on $\mathcal{K}^{k}_{d}$. 
\end{defn}

We may take the homotopy colimit of this functor over $\mathcal{K}^{k}_{d}$. 
We define
\begin{equation} \label{equation: homotopy colimit psc}
\mathcal{W}^{\psc, k}_{\theta, d}(P, g_{P}) := \hocolim_{\mb{t} \in \mathcal{K}_{d}^{k}}\mathcal{W}^{\psc, k}_{\theta, d}(P, g_{P}; \mb{t}).
\end{equation}

\subsection{A fibre sequence} \label{subsection: a fibre sequence}
We now analyze the natural transformation, 
\begin{equation} \label{equation: forgetful natural trans}
F_{\mb{t}}: \mathcal{W}^{\psc, k}_{\theta, d}(P, g_{P}; \mb{t}) \longrightarrow \mathcal{W}^{k}_{\theta, d}(P, \mb{t}),
\end{equation}
defined by the formula, 
$((M, (V, \sigma), e), g) \mapsto (M, (V, \sigma), e).$
Taking the homotopy colimit of these maps over $\mathcal{K}^{k}_{d}$ induces the map 
$F: \mathcal{W}^{\psc, k-1}_{\theta, d}(P, g_{P}) \longrightarrow \mathcal{W}^{k-1}_{\theta, d}(P)$.
The main result of this section is the following theorem.
\begin{theorem} \label{theorem: fibresequence forget the metric}
Let $k, d-k+1 \geq 3$.
Then the map 
$
F: \mathcal{W}^{\psc, k-1}_{\theta, d}(P, g_{P}) \longrightarrow \mathcal{W}^{k-1}_{\theta, d}(P)
$
is a quasi-fibration.
The fibre over $M := (M, (\emptyset, \emptyset), \emptyset) \in \mathcal{W}^{k-1}_{\theta, d}(P, \emptyset)$ is given by the space $\mathcal{R}^{+}(M)_{g_{P}}$. 
\end{theorem}

The proof of the above theorem will take a few steps. 
We will need to use the following proposition about maps between homotopy colimits which is a restatement of \cite[Corollary B.3]{MW 07}.
\begin{proposition} \label{proposition: homotopy colimit fibre sequence}
Let $\mathcal{C}$ be a small category and let $u: \mathcal{G}_{1} \longrightarrow \mathcal{G}_{2}$ be a natural transformation between functors 
from $\mathcal{C}$ to the category of topological spaces. 
Suppose that for each morphism $f: a \longrightarrow b$ in $\mathcal{C}$, the map 
$$
f_{*}: \hofibre\left(u_{a}: \mathcal{G}_{1}(a) \rightarrow \mathcal{G}_{2}(a)\right) \; \longrightarrow \; \hofibre\left(u_{b}: \mathcal{G}_{1}(b) \rightarrow \mathcal{G}_{2}(b)\right)
$$
is a weak homotopy equivalence. 
Then for any object $c \in  \Ob\mathcal{C}$ the inclusion 
$$
\hofibre\left(u_{c}: \mathcal{G}_{1}(c) \rightarrow \mathcal{G}_{2}(c)\right) \; \hookrightarrow \; \hofibre\left(\hocolim\mathcal{G}_{1} \rightarrow \hocolim\mathcal{G}_{2}\right)
$$
is a weak homotopy equivalence. 
\end{proposition}
To prove Theorem \ref{theorem: fibresequence forget the metric} we apply the above proposition to the natural transformation (\ref{equation: forgetful natural trans}).
In our first step to doing this we observe that for all  
$\mb{t} \in \mathcal{K}^{k}_{d}$, the forgetful map 
$$
F_{\mb{t}}: \mathcal{W}^{\psc, k}_{\theta, d}(P, g_{P}; \mb{t}) \longrightarrow \mathcal{W}^{k}_{\theta, d}(P, \mb{t})
$$
is a Serre-fibration.
This follows directly from how the spaces $\mathcal{M}^{\psc}_{\theta, d}(P, g_{P})$ were topologized in (\ref{equation: topology of moduli spaces}). 
Now, for any element $(M, (V, \sigma), e) \in \mathcal{W}^{k}_{P, \mb{t}}$ the fibre $F_{\mb{t}}^{-1}(M, (V, \sigma), e)$ is given by the space, 
$\mathcal{R}^{+}(M; \mb{t}, e)_{g_{P}}$.
So, for all $\mb{t} \in \mathcal{K}^{k}_{d}$ we have the weak homotopy equivalence
$$
\textstyle{\hofibre_{x}}(F_{\mb{t}}) \simeq \mathcal{R}^{+}(M; \mb{t}, e)_{g_{P}}, 
$$
where $x := (M, (V, \sigma), e)$.
To apply Proposition \ref{proposition: homotopy colimit fibre sequence} to our situation, we need to show that any morphism $\mb{s} \longrightarrow \mb{t}$ induces a weak homotopy equivalence between these homotopy fibres. 
This is addressed in the following proposition which follows from Theorems \ref{theorem: walsh chernysh} and \ref{theorem: perform surgery map} (see also Addendum \ref{Addendum: extension of Walsh} and Remark \ref{remark: multiple disjoint surgeries}).
\begin{proposition} \label{proposition: homotopy equivalence on fibres}
Let $k, d-k+1 \geq 3$. 
Then for any morphism $(j, \varepsilon): \mb{s} \longrightarrow \mb{t}$ in $\mathcal{K}^{k-1}_{d}$, and any $(M, (V, \sigma), e) \in \mathcal{W}^{k-1}_{\theta, d}(P, \mb{t})$, 
the map 
$S_{(j, \varepsilon)}: \mathcal{R}^{+}(M; \mb{t}, e_{-})_{g_{P}} \; \longrightarrow \; \mathcal{R}^{+}(M_{e, (j, \varepsilon)}; \; \mb{s}, (e_{(j, \varepsilon)})_{-})_{g_{P}}$
is a weak homotopy equivalence. 
\end{proposition}

We can now finish the proof of Theorem \ref{theorem: fibresequence forget the metric}. 
\begin{proof}[Proof of Theorem \ref{theorem: fibresequence forget the metric}]
The theorem follows by applying Proposition \ref{proposition: homotopy colimit fibre sequence}. 
By our observation above, for each $\mb{t} \in \mathcal{F}_{d}^{k}$ the homotopy fibre of 
$
F_{\mb{t}}: \mathcal{W}^{\psc, k-1}_{\theta, d}(P, g_{P}; \mb{t}) \longrightarrow \mathcal{W}^{k-1}_{\theta, d}(P, \mb{t})
$
over the element
$(M, (V, \sigma), e) \in \mathcal{W}^{k-1}_{P, \mb{t}}$ 
is given by the space 
$\mathcal{R}^{+}(M; \mb{t}, e)_{g_{P}}$.
Proposition \ref{proposition: homotopy equivalence on fibres} then implies that the maps between these homotopy fibres induced by the morphisms of $\mathcal{K}^{k-1}_{d}$ are all weak homotopy equivalences. 
\end{proof}

\begin{remark}
The statement of Theorem \ref{theorem: fibresequence forget the metric} holds for the spaces $\mathcal{W}^{\psc, k-1, c}_{\theta, d}(P, g_{P})$ as well, namely the map $\mathcal{W}^{\psc, k-1, c}_{\theta, d}(P, g_{P}) \longrightarrow \mathcal{W}^{k-1, c}_{\theta, d}(P)$ is a quasi-fibration as well with the same fibre as before. 
This is proved using the exactly the same argument that was used above.
\end{remark}

\subsection{The fibre transport} \label{subsection: the fibre transport}
For this section we assume that $\theta: B \longrightarrow BO$ is chosen so that $B$ is $2$-connected. 
In this way the map $\theta$ factors up to homotopy through the map $B\Spin \longrightarrow BO$.
Let $k$ and $d$ be chosen so that $k, d-k+1 \geq 3$.
Fix $(P, g_{P}) \in \mathcal{M}^{\psc}_{\theta, d-1}$. 
Fix an element $M \in \mathcal{M}^{k-1}_{\theta, d}(P)$.
Theorem \ref{theorem: fibresequence forget the metric} implies that there is a homotopy fibre sequence, 
$$
\xymatrix{
\mathcal{R}^{+}(M)_{g_{P}} \ar[r] & \mathcal{W}^{\psc, k-1}_{\theta, d}(P, g_{P}) \ar[r] & \mathcal{W}^{k-1}_{\theta, d}(P).
}
$$
Choosing a basepoint $g_{M} \in \mathcal{R}^{+}(M)_{g_{P}}$ determines a fibre transport map, 
\begin{equation} \label{equation: fibre transport map 1}
j_{d}: \Omega\mathcal{W}^{k-1}_{\theta, d}(P) \longrightarrow \mathcal{R}^{+}(M)_{g_{P}}.
\end{equation}
Recall from Section \ref{subsection: ko orientation} the $\KO$-orientation $\widehat{\mathcal{A}}^{-1}_{d}: \Omega^{\infty-1}\mb{hW}^{k-1}_{\theta, d} \longrightarrow \Omega^{\infty+d}\KO$. 
We define 
\begin{equation}
\bar{\mathcal{A}}_{d+1}: \mathcal{W}^{k-1}_{\theta, d}(P) \longrightarrow \Omega^{\infty+d}\KO
\end{equation} 
to be the map obtained by precomposing $\widehat{\mathcal{A}}^{-1}_{d+1}$ with the weak homotopy equivalence 
$$\mathcal{W}^{k-1}_{\theta, d}(P) \simeq \Omega^{\infty-1}\mb{hW}^{k-1}_{\theta, d+1}.$$
The main technical result for us to prove in this paper is the following theorem, consisting of two parts.
\begin{theorem} \label{theorem: factorization of index difference}
Let $k$ and $d$ be chosen so that $k, d-k+1 \geq 3$.
Fix $(P, g_{P}) \in \mathcal{M}^{\psc}_{\theta, d-1}$ and $M \in \mathcal{M}^{k-1}_{\theta, d}(P)$.  
\begin{enumerate}
\item[(a)] Suppose that $d = 2n \geq 6$. 
Then the composite
$$
\xymatrix{
\Omega\mathcal{W}^{k-1}_{\theta, 2n}(P) \ar[rr]^{j_{2n}} && \mathcal{R}^{+}(M)_{g_{P}} \ar[rr]^{\inddiff} && \Omega^{\infty+2n+1}\KO
}
$$
agrees with the map, $\Omega\bar{\mathcal{A}}_{2n}$.
\item[(b)] Suppose that $d = 2n-1 \geq 5$. 
Then the composite 
$$
\xymatrix{
\Omega^{2}\mathcal{W}^{k-1}_{\theta, 2n-1}(P) \ar[rr]^{\Omega j_{2n-1}} && \Omega\mathcal{R}^{+}(M)_{g_{P}} \ar[rr]^{\Omega\inddiff} && \Omega^{\infty+2n+1}\KO
}
$$
agrees with the map, $\Omega^{2}\bar{\mathcal{A}}_{2n-1}$.
\end{enumerate}
\end{theorem}
Combining the above theorem with the weak homotopy equivalence, $\mathcal{W}^{k-1}_{\theta, d}(P) \simeq \Omega^{\infty-1}\mb{hW}^{k-1}_{\theta, d+1}$ implies Theorem B from the introduction.
We prove part (a) in the following section.
We then prove part (b) as a consequence of part (a) over the course of the remaining sections of the paper. 
\begin{remark} \label{remark: only need to verify for k = n-1}
To prove Theorem \ref{theorem: factorization of index difference}, part (a), it will suffice to verify that the composite 
$$
\xymatrix{
\Omega\mathcal{W}^{n-1, c}_{\theta, d}(P) \ar[rr]^{j_{2n}} && \mathcal{R}^{+}(M)_{g_{P}} \ar[rr]^{\inddiff} && \Omega^{\infty+2n+1}\KO
}
$$ 
agrees with $\Omega\bar{\mathcal{A}}_{2n}$.
This is because for all choices of $k$ (with $k, d-k+1 \geq 3$) the diagram 
$$
\xymatrix{
\Omega\mathcal{W}^{n-1, c}_{\theta, 2n}(P)  \ar[dr] \ar[rr] && \Omega\mathcal{W}^{k-1}_{\theta, 2n}(P)  \ar[dl] \\
& \mathcal{R}^{+}(M)_{g_{P}} & 
}
$$
is commutative, and clearly the $\KO$-class $\bar{\mathcal{A}}_{2n} \in \KO^{-2n}(\mathcal{W}^{n-1, c}_{\theta, 2n}(P))$ is pulled back from $\mathcal{W}^{k-1}_{\theta, 2n}(P)$ via the inclusion, $\mathcal{W}^{n-1, c}_{\theta, 2n}(P) \hookrightarrow \mathcal{W}^{k-1}_{\theta, 2n}(P)$ (which is a weak homotopy equivalence).
The analogous statement also holds true for part (b) of Theorem \ref{theorem: factorization of index difference} which is the odd dimensional case. 
\end{remark}

\section{Fibre Sequences and Stabilization} \label{section: fibre sequences and stabilization}
In this section we prove Theorem \ref{theorem: factorization of index difference}, part (a). 
We restrict our attention to the dimensional conditions that are relevant to this theorem.
Let us make the following convention. 
\begin{convention} \label{Convention: dimensional convention}
Throughout this section we set $d = 2n$ with $n \geq 3$. 
Furthermore, we will assume that $\theta: B \longrightarrow BO$ has been chosen so that:
\begin{itemize} \itemsep.2cm
\item[(i)] $B$ is $2$-connected, and
\item[(ii)] $B$ satisfies Wall's finiteness condition $F(n)$, see \cite{W 65}.
\end{itemize}
The first condition implies that any $\theta$-manifold has a canonical $\Spin$-structure induced by its $\theta$-structure. 
Condition (ii) implies that for all $k < n$ and $P \in \mathcal{M}_{\theta, 2n-1}$, there is a homotopy equivalence 
$\mathcal{W}_{\theta, 2n}^{k}(P) \simeq \Omega^{\infty-1}\mb{hW}^{k}_{\theta, 2n+1}$
as a consequence of the theorems in \cite{P 17}.
\end{convention}

In this section we will also need to use the cobordism category defined in \cite{GMTW 08} by Galatus, Madsen, Tillmann, and Weiss.
The topological category $\Cob_{\theta, 2n}$ has $\mathcal{M}_{\theta, 2n-1}$ for its space of objects. 
For $P, Q \in \Ob\Cob_{\theta, 2n}$, a morphism $P \rightsquigarrow Q$ is given by a pair $(t, W)$ where $t \in (0, \infty)$, and $W \subset [0, t]\times\R^{\infty}$ is a $2n$-dimensional compact manifold equipped with $\theta$-structure $\hat{\ell}_{W}$, such that,
$\partial W = W\cap(\{0, t\}\times\R^{\infty}) = P\sqcup Q,$ 
as a $\theta$-manifold. 
Furthermore, the manifold $W$ is required to be embedded in $[0, t]\times\R^{\infty}$ with a collar. 
Composition in this category is given in the usual way by concatenation of cobordisms. 
For $k \in \Z_{\geq -1}$, the subcategory $\Cob^{k}_{\theta, 2n} \subset \Cob_{\theta, 2n}$ has the same objects. 
Its morphisms are given by those $(t, W): P \rightsquigarrow Q$ for which the pair $(W, P)$ is $k$-connected. 
The main subcategory of interest for our constructions is $\Cob^{n-1}_{\theta, 2n}$.
As described in \cite[Section 7]{P 17} we may consider the correspondence $P \mapsto \mathcal{W}_{\theta, 2n}^{n-1}(P)$ to be a functor on the cobordism category $\Cob^{n-1}_{\theta, 2n}$.

\subsection{Stabilization for spaces of manifolds} \label{subsection: stabilization} 
We will apply a stabilization process to the functor $\mathcal{W}_{\theta, 2n}^{n-1, c}(\--, \--)$.
This subsection is mainly a recollection of \cite[Section 12]{P 17}.
\begin{Construction} \label{Construction: stabilize by s-n times s-n}
Let $P \in \mathcal{M}_{\theta, 2n-1}$. 
We construct a morphism
$
H_{n, n}(P): P \rightsquigarrow P
$
of $\Cob^{n-1}_{\theta, 2n}$ as follows.
Choose an embedding $j: S^{n}\times S^{n} \hookrightarrow (0, 1)\times\R^{\infty-1}$ with image disjoint from $(0, 1)\times P$. 
Fix a $\theta$-structure $\hat{\ell}_{n, n}$ on $S^{n}\times S^{n}$ with the property that $\ell_{n, n}: S^{n}\times S^{n} \longrightarrow B$ is null-homotopic. 
We then let $H_{n, n}(P) \subset [0, 1]\times\R^{\infty-1}$ be the cobordism obtained by forming the connected sum of $[0, 1]\times P$ with $j(S^{n}\times S^{n})$ along some embedded arc connecting the two submanifolds.
We equip $H_{n, n}(P)$ with a $\theta$-structure that agrees with $\hat{\ell}_{n, n}$ on the connect-summand of $S^{n}\times S^{n}$, and with $\hat{\ell}_{[0, 1]\times P}$ away from the connect-summand.
This self-cobordism $H_{n, n}(P): P \rightsquigarrow P$ equipped with its $\theta$-structure yields a morphism in the cobordism category $\Cob^{n-1}_{\theta, 2n}$. 
\end{Construction} 

\begin{defn} \label{defn: stable moduli space}
Fix $P \in \mathcal{M}_{\theta, 2n-1}$.
The morphism $H_{n, n}(P): P \rightsquigarrow P$ induces a map 
$$
\mb{S}: \mathcal{M}^{n}_{\theta, 2n}(P) \longrightarrow \mathcal{M}^{n}_{\theta, 2n}(P),\quad
M \mapsto M\cup H_{n, n}(P). 
$$
This map may be iterated to form a direct system. 
We define
\begin{equation} \label{equation: stable moduli space direct limit}
\mathcal{M}^{n}_{\theta, 2n}(P)^{\stb} := \hocolim\left(\mathcal{M}^{n}_{\theta, 2n}(P) \stackrel{\mb{S}} \longrightarrow \mathcal{M}^{n}_{\theta, 2n}(P) \stackrel{\mb{S}} \longrightarrow \cdots\right).
\end{equation}
\end{defn}

We may form similar direct systems with the spaces $\mathcal{W}^{n-1, c}_{\theta, 2n}(P; \mb{t})$.
The cobordism $H_{n, n}(P)$ induces a similar map, 
$
\mb{S}: \mathcal{W}^{n-1, c}_{\theta, 2n}(P; \mb{t}) \; \longrightarrow \; \mathcal{W}^{n-1, c}_{\theta, 2n}(P; \mb{t}).
$
This map may be iterated and we define,
\begin{equation} \label{equation: stabilization by s-n by s-n}
\mathcal{W}^{n-1, c}_{\theta, 2n}(P; \mb{t})^{\stb} := \hocolim\left[\mathcal{W}^{n-1, c}_{\theta, 2n}(P; \mb{t}) \rightarrow \mathcal{W}^{n-1, c}_{\theta, 2n}(P; \mb{t}) \rightarrow \cdots \right].
\end{equation}
We define,
$\mathcal{W}^{n-1, c}_{\theta, 2n}(P)^{\stb} \; := \; \displaystyle{\hocolim_{\mb{t} \in \mathcal{K}^{n-1}_{2n}}}\mathcal{W}^{n-1, c}_{\theta, 2n}(P, \mb{t})^{\stb}.$
The following proposition is the same as \cite[Proposition 12.2]{P 17}.
\begin{proposition} \label{proposition: stabiliization equivalence}
For all $P \in \mathcal{M}_{\theta, 2n-1}$,
the map 
$
\mb{S}: \mathcal{W}^{n-1, c}_{\theta, 2n}(P; \mb{t}) \; \longrightarrow \; \mathcal{W}^{n-1, c}_{\theta, 2n}(P; \mb{t})
$
is a weak homotopy equivalence, and thus the
natural inclusion $\mathcal{W}^{n-1, c}_{\theta, 2n}(P) \hookrightarrow \mathcal{W}^{n-1, c}_{\theta, 2n}(P)^{\stb}$ is a weak homotopy equivalence as well.
\end{proposition} 

We now consider the the localization map, 
$
\mb{L}: \mathcal{W}^{n-1, c}_{\theta, 2n}(P)^{\stb}  \; \longrightarrow \; \mathcal{W}^{\{n, n+1\}}_{\theta, 2n, \loc},
$
which is the map induced by (\ref{equation: localization transformation}).
We proceed to analyze the homotopy fibre of this map. 
We need a definition. 
\begin{defn} \label{defn: acyclic homology fibration}
A map $f: X \longrightarrow Y$ between topological spaces is said to be an \textit{acyclic homology fibration} if the for all $y \in Y$, the inclusion, 
$f^{-1}(y) \hookrightarrow \textstyle{\hofibre_{y}}(f)$,
is an acyclic map. 
\end{defn}

The following theorem is a restatement of \cite[Corollary 12.9]{P 17}.
The key ingredient of its proof in \cite{P 17} is the homological stability theorem of Galatius and Randal-Williams from \cite{GRW 16}.
\begin{theorem} \label{theorem: homology fibration}
The map,
$
\mb{L}: \mathcal{W}^{n-1, c}_{\theta, 2n}(P)^{\stb}  \; \longrightarrow \; \mathcal{W}^{\{n, n+1\}}_{\theta, 2n, \loc},
$
is an acyclic homology fibration.
The fibre over the empty element $\emptyset  \in \mathcal{W}^{\{n, n+1\}}_{\theta, 2n, \loc}$ is given by the stable moduli space $\mathcal{M}^{n}_{\theta, 2n}(P)^{\stb}$.
\end{theorem}

\begin{remark} \label{remark: equivalence of homotopy fibre to MT}
Let us denote,
$\mb{F}_{\theta, 2n} := \hofibre(\mathcal{W}^{n-1, c}_{\theta, 2n}(P)^{\stb}  \; \rightarrow \; \mathcal{W}^{\{n, n+1\}}_{\theta, 2n, \loc}).$
By the results in \cite{P 17} the map from the statement of Theorem \ref{theorem: homology fibration} fits into a commutative diagram
$$
\xymatrix{
\mathcal{W}^{n-1, c}_{\theta, 2n}(P)^{\stb} \ar[d]_{\simeq}  \ar[r] & \mathcal{W}^{\{n, n+1\}}_{\theta, 2n, \loc} \ar[d]_{\simeq} \\
\Omega^{\infty-1}\mb{hW}^{n-1}_{\theta, 2n+1} \ar[r] & \Omega^{\infty-1}\Sigma^{\infty}G_{\theta, 2n+1}^{\mf}(\R^{\infty})^{\{n, n+1\}}_{\loc, +},
}
$$
and furthermore the homotopy fibre of the bottom row is given by the space $\Omega^{\infty}\MT\theta_{2n}$.
This commutative diagram then induces the weak homotopy equivalence,
\begin{equation} \label{equation: homotopy fibre to MT}
\mb{F}_{\theta, 2n} \stackrel{\simeq} \longrightarrow \Omega^{\infty}\MT\theta_{2n}.
\end{equation}
Now, let 
$\rho: \mathcal{M}^{n}_{\theta, 2n}(P)^{\stb} \longrightarrow \Omega^{\infty}\MT\theta_{2n}$
be the \textit{scanning map}, which in \cite{GRW 16} is proven to be an acyclic map.
Tracing through the constructions, it follows that the diagram below is commutative,
$$
\xymatrix{
\mathcal{M}^{n}_{\theta, 2n}(P)^{\stb} \ar[dr] \ar[rr]^{\rho} && \Omega^{\infty}\MT\theta_{2n} \\
& \mb{F}_{\theta, 2n}. \ar[ur]^{\simeq} &
}
$$ 
The descending diagonal map is the inclusion of the fibre of $\mb{L}$ into its homotopy fibre, and the ascending-diagonal map is the weak homotopy equivalence from (\ref{equation: homotopy fibre to MT}). 
\end{remark}

\subsection{Stabilization with positive scalar curvature metrics}
Let $(P, g_{P}) \in \mathcal{M}_{\theta, 2n-1}^{\psc}$.
We proceed to define a similar stabilization for the spaces $\mathcal{W}^{\psc, n-1, c}_{\theta, 2n}(P, g_{P})$.
Recall the cobordism $H_{n, n}(P): P \rightsquigarrow P$. 
Our main construction to follow will require the use of a preliminary proposition, which follows as a result of \cite[Theorem 2.3.4]{BERW 16}.
Recall from Convention \ref{Convention: dimensional convention} that $n \geq 3$ and that the space $B$ is $2$-connected. 
\begin{proposition} \label{proposition: stability for metrics}
Let $(P, g_{P}) \in \mathcal{M}_{\theta, 2n-1}^{\psc}$ and suppose that $P$ is simply connected. 
There exists an element $\bar{g} \in \mathcal{R}^{+}(H_{n, n}(P))_{g_{P}, g_{P}}$ such that for all $M \in \mathcal{M}_{\theta, 2n}(P)$ the map 
$$
\mathcal{R}^{+}(M)_{g_{P}} \longrightarrow \mathcal{R}^{+}(M\cup H_{n,n}(P))_{g_{P}}, \quad 
g \mapsto g\cup\bar{g},
$$
is a weak homotopy equivalence. 
\end{proposition}

Moving forward it will be useful to have some terminology available for the $\psc$ metrics on $H_{n, n}(P)$ that satisfy the condition from the statement of Proposition \ref{proposition: stability for metrics}.  
\begin{defn} \label{defn: stable metrics}
Fix $(P, g_{P}) \in \mathcal{M}_{\theta, 2n-1}^{\psc}$. 
An element $\bar{g} \in \mathcal{R}^{+}(H_{n, n}(P))_{g_{P}, g_{P}}$ is said to be \textit{stable} if for all $M \in \mathcal{M}_{\theta, 2n}(P)$ the map, 
$$\mathcal{R}^{+}(M)_{g_{P}} \longrightarrow \mathcal{R}^{+}(M\cup H_{n,n}(P))_{g_{P}}, \quad
g \mapsto g\cup\bar{g},$$
is a weak homotopy equivalence. 
\end{defn}

\begin{Construction} \label{Construction: stabilization for psc}
Fix an element $(P, g_{P}) \in \mathcal{M}_{\theta, 2n-1}^{\psc}$ with $P$ simply-connected. 
With $P$ chosen to be simply connected, Theorem \ref{proposition: stability for metrics} holds with respect to $(P, g_{P})$.
Consider the cobordism $H_{n, n}(P)$ from Construction \ref{Construction: stabilize by s-n times s-n}. 
Fix once and for all a $\psc$ metric $\widehat{g} \in \mathcal{R}^{+}(H_{n,n}(P))_{g_{P}, g_{P}}$ that is stable in the sense of Definition \ref{defn: stable metrics}. 
Concatenation with $(H_{n, n}(P), \widehat{g})$ yields a map,
$$
\mb{S}_{\widehat{g}}: \mathcal{M}_{\theta, 2n}^{\psc, n}(P, g_{P}) \longrightarrow \mathcal{M}_{\theta, 2n}^{\psc, n}(P, g_{P}), \quad (W, h) \mapsto (W\cup H_{n, n}(P), \; h\cup \widehat{g}).
$$
Iterating this map forms a direct system and we define,
\begin{equation} \label{equation: stable moduli of psc metrics}
\mathcal{M}^{\psc, n}_{\theta, 2n}(P, g_{P})^{\stb} := \hocolim\left(\mathcal{M}_{\theta, 2n}^{\psc, n}(P, g_{P}) \stackrel{\mb{S}_{\widehat{g}}} \longrightarrow \mathcal{M}_{\theta, 2n}^{\psc, n}(P, g_{P}) \stackrel{\mb{S}_{\widehat{g}}} \longrightarrow \cdots\right).
\end{equation}

For each $\mb{t} \in \mathcal{K}^{n-1}_{2n}$ we have a similar map, 
$
\mb{S}_{\widehat{g}}: \mathcal{W}^{\psc, n-1, c}_{\theta, 2n}(P, g_{P}; \mb{t}) \; \longrightarrow \; \mathcal{W}^{\psc, n-1, c}_{\theta, 2n}(P, g_{P}; \mb{t}).
$
This map may be iterated and we define,
\begin{equation} \label{equation: stabilization by s-n by s-n}
\mathcal{W}^{\psc, n-1, c}_{\theta, 2n}(P, g_{P}; \mb{t})^{\stb} := \hocolim\left[\mathcal{W}^{\psc, n-1, c}_{\theta, 2n}(P, g_{P},; \mb{t}) \rightarrow \mathcal{W}^{\psc, n-1, c}_{\theta, 2n}(P, g_{P}; \mb{t}) \rightarrow \cdots \right].
\end{equation}

Similarly, for each $M \in \mathcal{M}_{\theta, 2n}^{\psc}(P, g_{P})$, the metric $\widehat{g} \in \mathcal{R}^{+}(H_{n, n}(P))_{g_{P}, g_{P}}$ yields a map 
$$
\mathcal{R}^{+}(M)_{g_{P}} \longrightarrow \mathcal{R}^{+}(M\cup H_{n, n}(P))_{g_{P}}, \quad
g \mapsto g\cup\widehat{g}.
$$
This map can be iterated as well and we define
\begin{equation} \label{equation: stabilization map for space of metrics}
\mathcal{R}^{+}(M)^{\stb}_{g_{P}} \; = \; \hocolim\left(\mathcal{R}^{+}(M)_{g_{P}} \rightarrow \mathcal{R}^{+}(M)_{g_{P}} \rightarrow \mathcal{R}^{+}(M)_{g_{P}} \rightarrow \cdots\right).
\end{equation}
\end{Construction}
Let $(P, g_{P}) \in \mathcal{M}_{\theta, d-1}^{\psc}$ be the element that was used in Construction \ref{Construction: stabilization for psc}. 
This element was chosen so that $P$ was simply connected. 
We will keep this element fixed throughout the rest of the section. 
The following proposition is similar to Theorem \ref{theorem: fibresequence forget the metric}.
\begin{proposition} \label{proposition: quasifibration of stable moduli}
The forgetful map, 
$F^{\stb}: \mathcal{W}^{\psc, n-1, c}_{\theta, 2n}(P, g_{P})^{\stb} \longrightarrow \mathcal{W}^{n-1, c}_{\theta, 2n}(P)^{\stb},$
is a quasi-fibration.
Any element $M \in \mathcal{M}^{n}_{\theta, 2n}(P)$ determines a point in $\mathcal{W}^{n-1, c}_{\theta, 2n}(P)^{\stb}$. 
The homotopy fibre of $F^{\stb}$ over such $M$ is given by the space $\mathcal{R}^{+}(M)_{g_{P}}^{\stb}$.
\end{proposition}
\begin{proof}
For $\mb{t} \in \mathcal{K}_{2n}^{n-1}$, consider the map, 
$
F^{\stb}_{\mb{t}}: \mathcal{W}^{\psc, n-1, c}_{\theta, 2n}(P, g_{P}; \mb{t})^{\stb} \longrightarrow \mathcal{W}^{n-1, c}_{\theta, 2n}(P, \mb{t})^{\stb}.
$
We claim that the homotopy fibre of this map over $(M, (V, \sigma), e) \in \mathcal{W}^{n-1, c}_{\theta, 2n}(P, \mb{t})$ (considered as an element of $\mathcal{W}^{n-1, c}_{\theta, 2n}(P, \mb{t})^{\stb}$) is given by the space $\mathcal{R}^{+}(M; \mb{t}, e)^{\stb}_{g_{P}}$.
Here, $\mathcal{R}^{+}(M; \mb{t}, e)_{g_{P}}$ is from Definition \ref{defn: metrics standard on multiple surgeries}, and 
$$\mathcal{R}^{+}(M; \mb{t}, e)^{\stb}_{g_{P}} \; = \; \hocolim\left(\mathcal{R}^{+}(M; \mb{t}, e)^{\stb}_{g_{P}} \rightarrow \mathcal{R}^{+}(M\cup H_{n, n}(P); \mb{t}, e)^{\stb}_{g_{P}} \rightarrow \cdots \right).$$
To establish the above claim we first observe that the forgetful map
$$
F_{\mb{t}}: \mathcal{W}^{\psc, n-1, c}_{\theta, 2n}(P, g_{P}; \mb{t}) \longrightarrow \mathcal{W}^{n-1, c}_{\theta, 2n}(P, \mb{t})
$$
is a Serre-fibration and that the fibre over any such $(M, (V, \sigma), e)$ is given by $\mathcal{R}^{+}(M; \mb{t}, e)_{g_{P}}$.
Fix an element $(M, (V, \sigma), e) \in  \mathcal{W}^{n-1, c}_{\theta, 2n}(P, \mb{t})$, and let us denote $x :=  (M, (V, \sigma), e)$.
The map $\mb{S}_{\widehat{g}}$ induces a map of fibre sequences 
$$
\xymatrix{
\mathcal{R}^{+}(M; \mb{t}, e)^{\stb}_{g_{P}} \ar[d] \ar[rr]^{h\mapsto h\cup\widehat{g}_{P} \ \ \ } && \mathcal{R}^{+}(M\cup H_{n, n}(P); \mb{t}, e)^{\stb}_{g_{P}} \ar[d] \\
\mathcal{W}^{\psc, n-1, c}_{\theta, 2n}(P, g_{P}; \mb{t}) \ar[d]^{F_{\mb{t}}} \ar[rr]^{\mb{S}_{\widehat{g}}} && \mathcal{W}^{\psc, n-1, c}_{\theta, 2n}(P, g_{P}; \mb{t}) \ar[d]^{F_{\mb{t}}} \\
\mathcal{W}^{n-1, c}_{\theta, 2n}(P; \mb{t}) \ar[rr]^{\mb{S}} && \mathcal{W}^{n-1, c}_{\theta, 2n}(P; \mb{t}).
}
$$
Since the metric $\widehat{g}$ was chosen to be stable, it follows that the top-horizontal map in the above diagram is a weak homotopy equivalence.  
By an application of Proposition \ref{proposition: homotopy colimit fibre sequence} it follows that the natural inclusion 
$$
\mathcal{R}^{+}(M; \mb{t}, e)_{g_{P}} \; \hookrightarrow \; \textstyle{\hofibre_{x}}(F^{\stb}_{\mb{t}}),
$$
is a weak homotopy equivalence, where $x$ denotes the element $(M, (V, \sigma), e)$.
Combining this weak homotopy equivalence with the weak homotopy equivalence,
$\mathcal{R}^{+}(M; \mb{t}, e)_{g_{P}} \; \simeq \; \mathcal{R}^{+}(M; \mb{t}, e)^{\stb}_{g_{P}},$
proves our claim that,
$$
\textstyle{\hofibre_{x}}(F^{\stb}_{\mb{t}}) \simeq \mathcal{R}^{+}(M; \mb{t}, e)^{\stb}_{g_{P}}.
$$
By Proposition \ref{proposition: homotopy equivalence on fibres}, every morphism $(j, \varepsilon): \mb{t} \longrightarrow \mb{s}$ induces a weak homotopy equivalence 
$$
\textstyle{\hofibre_{(j, \varepsilon)^{*}x}}(F^{\stb}_{\mb{s}}) \; \stackrel{\simeq} \longrightarrow \; \textstyle{\hofibre_{x}}(F^{\stb}_{\mb{t}}),
$$
and thus it follows from Proposition \ref{proposition: homotopy colimit fibre sequence} that the induced map of homotopy colimits (over $\mathcal{K}^{n-1}_{2n}$), 
$$
F^{\stb}: \mathcal{W}^{\psc, n-1, c}_{\theta, 2n}(P, g_{P})^{\stb} \longrightarrow \mathcal{W}^{n-1, c}_{\theta, 2n}(P)^{\stb},
$$
is a quasi-fibration.
As observed above it follows that the fibre over 
$M \in \mathcal{W}^{n-1, c}_{\theta, 2n}(P, \emptyset)$
is given by the space $\mathcal{R}^{+}(M)^{\stb}_{g_{P}}$.
This concludes the proof of the proposition.
\end{proof}

Let 
$
\pi^{\stb}: \mathcal{M}^{\psc, n}_{\theta, 2n}(P, g_{P})^{\stb} \longrightarrow \mathcal{M}^{n}_{\theta, 2n}(P)^{\stb}
$
be the map obtained by forming the homotopy colimit of the maps,
\begin{equation} \label{equation: direct limit of the projections pi}
\xymatrix{
\mathcal{M}^{\psc, n}_{\theta, 2n}(P, g_{P}) \ar[d]^{\pi} \ar[r] & \mathcal{M}^{\psc, n}_{\theta, 2n}(P, g_{P}) \ar[r] \ar[d]^{\pi} & \mathcal{M}^{\psc, n}_{\theta, 2n}(P, g_{P}) \ar[r] \ar[d]^{\pi} & \cdots \\
\mathcal{M}^{n}_{\theta, 2n}(P) \ar[r] & \mathcal{M}^{n}_{\theta, 2n}(P) \ar[r] & \mathcal{M}^{n}_{\theta, 2n}(P) \ar[r] & \cdots
}
\end{equation}
where $\pi$ is the bundle projection (alias forgetful map), and the horizontal maps are the stabilization maps induced by cobordism $H_{n, n}(P)$ and the metric $\widehat{g}$.

\begin{corollary} \label{corollary: homotopy cartesian 1}
The commutative diagram, 
\begin{equation} \label{equation: homotopy cartesian 1}
\xymatrix{
\mathcal{M}^{\psc, n}_{\theta, 2n}(P, g_{P})^{\stb}  \ar[r] \ar[d]^{\pi^{\stb}} & \mathcal{W}_{\theta, 2n}^{\psc, n-1, c}(P, g_{P})^{\stb} \ar[d]^{F^{\stb}} \\
\mathcal{M}^{n}_{\theta, 2n}(P)^{\stb} \ar[r] & \mathcal{W}_{\theta, 2n}^{n-1, c}(P)^{\stb},
}
\end{equation}
is homotopy Cartesian. 
\end{corollary}
\begin{proof}
Since $\widehat{g}$ is stable, by the same argument used in the proof go Proposition \ref{proposition: quasifibration of stable moduli} it follows that 
$\pi^{\stb}: \mathcal{M}^{\psc, n}_{\theta, 2n}(P, g_{P})^{\stb} \longrightarrow \mathcal{M}^{n}_{\theta, 2n}(P)^{\stb}$
is a quasi-fibration with fibre (over $M$) given by $\mathcal{R}^{+}(M)^{\stb}_{g_{P}}$. 
By Proposition \ref{proposition: quasifibration of stable moduli}, the inclusion $\mathcal{R}^{+}(M)^{\stb}_{g_{P}} \hookrightarrow \hofibre(F^{\stb})$ is a weak homotopy equivalence. 
It follows that the map $\hofibre(\pi^{\stb}) \longrightarrow \hofibre(F^{\stb})$ induced by the commutative diagram (\ref{equation: homotopy cartesian 1}) is a weak homotopy equivalence. 
This implies that the diagram is homotopy Cartesian and thus proves the corollary.
\end{proof}

\begin{corollary} \label{proposition: inclusion into colimit equivalence}
The inclusion $\mathcal{W}^{\psc, n-1, c}_{\theta, 2n}(P, g_{P}) \hookrightarrow \mathcal{W}^{\psc, n-1, c}_{\theta, 2n}(P, g_{P})^{\stb}$ is a weak homotopy equivalence.
\end{corollary}
\begin{proof}
Consider the commutative diagram,
$$
\xymatrix{
\mathcal{R}^{+}(M)_{g_{P}} \ar[r]^{\simeq} \ar[d] & \mathcal{R}^{+}(M)^{\stb}_{g_{P}} \ar[d] \\
\mathcal{W}^{\psc, n-1, c}_{\theta, 2n}(P, g_{P}) \ar[r] \ar[d] & \mathcal{W}^{\psc, n-1, c}_{\theta, 2n}(P, g_{P})^{\stb} \ar[d] \\
\mathcal{W}^{n-1, c}_{\theta, 2n}(P) \ar[r]^{\simeq}  & \mathcal{W}^{n-1, c}_{\theta, 2n}(P)^{\stb}.
}
$$ 
The horizontal maps are inclusions and the vertical columns are homotopy fibre sequences. 
The bottom-horizontal map is a weak homotopy equivalence by Proposition \ref{proposition: stabiliization equivalence}. 
Since the metric $\widehat{g}$ was chosen in Construction \ref{Construction: stabilization for psc} to be stable it follows that the top-horizontal map is a weak homotopy equivalence as well. 
The proof of the proposition then follows from the fact that both columns are homotopy fibre sequences. 
\end{proof}

As in the previous section we have a localization map, 
$$
\mb{L}^{\psc}: \mathcal{W}^{\psc, n-1, c}_{\theta, 2n}(P, g_{P})^{\stb}  \; \longrightarrow \; \mathcal{W}^{\{n, n+1\}}_{\theta, 2n, \loc}
$$
defined in the same way as (\ref{equation: localization transformation}). 
Let us denote,
\begin{equation} \label{equation: homotopy fibres of maps}
\begin{aligned}
\mb{F}^{\psc}_{\theta, 2n} &:= \textstyle{\hofibre}\left(\mathcal{W}_{\theta, 2n}^{\psc, n-1, c}(P, g_{P})^{\stb} \stackrel{\mb{L}^{\psc}} \longrightarrow \mathcal{W}^{\{n, n+1\}}_{\theta, 2n, \loc}\right), \\
\mb{F}_{\theta, 2n} &:= \textstyle{\hofibre}\left(\mathcal{W}_{\theta, 2n}^{n-1, c}(P)^{\stb} \stackrel{\mb{L}} \longrightarrow \mathcal{W}^{\{n, n+1\}}_{\theta, 2n, \loc}\right),
\end{aligned}
\end{equation}
where the homotopy fibres are taken over $\emptyset \in \mathcal{W}^{\{n, n+1\}}_{\theta, 2n, \loc}$
(since $\mathcal{W}^{\{n, n+1\}}_{\theta, 2n, \loc}$ is path-connected, it doesn't matter which point in $\mathcal{W}^{\{n, n+1\}}_{\theta, 2n, \loc}$ that we choose). 
Since the diagram 
$$
\xymatrix{
\mathcal{W}_{\theta, 2n}^{\psc, n-1, c}(P, g_{P})^{\stb} \ar[dr]^{\mb{L}^{\psc}}  \ar[rr]^{F^{\stb}} && \mathcal{W}_{\theta, 2n}^{n-1, c}(P)^{\stb} \ar[dl]_{\mb{L}} \\
& \mathcal{W}^{\{n, n+1\}}_{\theta, 2n, \loc} &
}
$$
is commutative, the forgetful map $F^{\stb}$ induces a map between the homotopy fibres which we denote
$$\widetilde{F}^{\stb}: \mb{F}^{\psc}_{\theta, 2n} \longrightarrow \mb{F}_{\theta, 2n}.$$
Recall from Theorem \ref{theorem: homology fibration} that the inclusion 
$\mathcal{M}^{n}_{\theta, 2n}(P)^{\stb} \longrightarrow \mb{F}_{\theta, 2n}$ is an acyclic map. 
Since the inclusion map $\mathcal{M}^{\psc, n}_{\theta, 2n}(P, g_{P})^{\stb} \hookrightarrow \mathcal{W}_{\theta, 2n}^{\psc, n-1, c}(P, g_{P})^{\stb}$ factors through $(\mb{L}^{\psc})^{-1}(\emptyset)$, this inclusion induces an embedding into the homotopy fibre, 
$\mathcal{M}^{\psc, n}_{\theta, 2n}(P, g_{P})^{\stb}  \hookrightarrow \mb{F}^{\psc}_{\theta, 2n}.$
The following theorem is the main result of this section.
\begin{theorem} \label{theorem: homology fibration psc}
The following diagram 
\begin{equation} \label{equation: cartesian psc square}
\xymatrix{
\mathcal{M}^{\psc, n}_{\theta, 2n}(P, g_{P})^{\stb} \ar[r] \ar[d]^{\pi^{\stb}}  & \mb{F}^{\psc}_{\theta, 2n} \ar[d]^{\widetilde{F}^{\stb}} \\
\mathcal{M}^{n}_{\theta, 2n}(P)^{\stb} \ar[r] & \mb{F}_{\theta, 2n}
}
\end{equation}
is homotopy cartesian and both horizontal maps are acyclic. 
\end{theorem}
\begin{proof}
Consider the map of fibre sequences,
\begin{equation} \label{equation: diagram of fibre sequences}
\xymatrix{
\mb{F}^{\psc}_{\theta, 2n} \ar[r] \ar[d]^{\widetilde{F}^{\stb}} & \mathcal{W}_{\theta, 2n}^{\psc, n-1, c}(P, g_{P})^{\stb} \ar[rr]^{\mb{L}^{\psc}} \ar[d]^{F^{\stb}} &&  \mathcal{W}^{\{n, n+1\}}_{\theta, 2n, \loc} \ar[d]^{=} \\
\mb{F}_{\theta, 2n} \ar[r] & \mathcal{W}_{\theta, 2n}^{n-1, c}(P)^{\stb} \ar[rr]^{\mb{L}} && \mathcal{W}^{\{n, n+1\}}_{\theta, 2n, \loc}.
}
\end{equation}
Since the rightmost vertical map is the identity, it follows that the left-square in (\ref{equation: diagram of fibre sequences}) is a homotopy Cartesian square. 
We now observe that the inclusion maps,
$$
\mathcal{M}^{\psc, n}_{\theta, 2n}(P, g_{P})^{\stb} \hookrightarrow \mathcal{W}_{\theta, 2n}^{\psc, n-1, c}(P, g_{P})^{\stb} \quad \text{and} \quad
\mathcal{M}^{n}_{\theta, 2n}(P)^{\stb} \hookrightarrow \mathcal{W}_{\theta, 2n}^{n-1, c}(P)^{\stb},
$$
factor through the fibres $(\mb{L}^{\psc})^{-1}(\emptyset)$ and $\mb{L}^{-1}(\emptyset)$ respectively. 
From this we obtain the commutative diagram
\begin{equation} \label{equation: double commutative square} 
\xymatrix{
\mathcal{M}^{\psc, n}_{\theta, 2n}(P, g_{P})^{\stb} \ar[d]^{\pi^{\stb}} \ar[r] & \mb{F}^{\psc}_{\theta, 2n} \ar[r] \ar[d]^{\widetilde{F}^{\stb}} &  \mathcal{W}_{\theta, 2n}^{\psc, n-1, c}(P, g_{P})^{\stb} \ar[d]^{F^{\stb}} \\
 \mathcal{M}^{n}_{\theta, 2n}(P)^{\stb} \ar[r] & \mb{F}_{\theta, 2n} \ar[r] & \mathcal{M}^{n}_{\theta, 2n}(P)^{\stb}.
}
\end{equation}
The horizontal maps in this diagram induce maps between the homotopy fibres of $\pi^{\stb}$, $\widetilde{F}^{\stb}$, and $F^{\stb}$.
We denote these induced maps by,
\begin{equation}
\xymatrix{
\hofibre(\pi^{\stb}) \ar[r]^{a} & \hofibre(\widetilde{F}^{\stb}) \ar[r]^{b} & \hofibre(F^{\stb}).
}
\end{equation}
By Corollary \ref{corollary: homotopy cartesian 1}, the composite $b\circ a$ is a homotopy equivalence (this is because Corollary \ref{corollary: homotopy cartesian 1} implies that the outer square of (\ref{equation: double commutative square}) is homotopy cartesian). 
By our initial observation that the right-square of (\ref{equation: diagram of fibre sequences}) was homotopy-Cartesian, it follows that the map $b$ is a weak homotopy equivalence. 
By the two-out-of-three property it follows that the map $a$ is a weak homotopy equivalence as well. 
This proves that (\ref{equation: cartesian psc square}) is homotopy cartesian.
By Theorem \ref{theorem: homology fibration} the bottom horizontal arrow of (\ref{equation: cartesian psc square}) is an acyclic map. 
Since the square is homotopy Cartesian it follows form \cite[Proposition 2.2]{HH 79} that the top-horizontal map, 
$
\mathcal{M}^{\psc, n}_{\theta, 2n}(P, g_{P})^{\stb} \longrightarrow \mb{F}^{\psc}_{\theta, 2n},
$
is acyclic as well. 
This concludes the proof of the theorem.
\end{proof}

\subsection{Proof of Theorem \ref{theorem: factorization of index difference}, part (a)} \label{subsection: the even dimensional result}
In this section we prove Theorem \ref{theorem: factorization of index difference}, part (a), which is the the even-dimensional case. 
For the statement of the result we will work with the element $(P, g_{P}) \in \mathcal{M}^{\psc}_{\theta, 2n-1}$ that was fixed back in Construction \ref{Construction: stabilization for psc}. 
To give the proof of Theorem \ref{theorem: factorization of index difference}, we need to briefly review some $K$-theoretic results from \cite{BERW 16}.

\begin{Construction} \label{Construction: relaive k-theory construction} 
let $f: X \longrightarrow Y$ be a map between two topological spaces. 
Let $m \in \Z$.
Consider the relative $\KO^{m}$-group $\KO^{m}(f)$. 
Following \cite{BERW 16}, we denote:
\begin{itemize} \itemsep.2cm
\item $\bas: \KO^{m}(f) \longrightarrow \KO^{m}(Y)$ is the map given by restriction;
\item $\trg: \KO^{m}(f) \longrightarrow \KO^{m-1}(\hofibre(f))$ is the transgression map (see \cite[Section 3.5]{BERW 16}).
\end{itemize}
Let $\iota: \Omega Y \longrightarrow \hofibre(f)$ be the fibre transport map. 
In \cite[Lemma 3.5.1]{BERW 16}, it is proven that for all classes $x \in \KO^{m}(f)$, the following equation holds:
$$
\Omega\bas(x) = \iota^{*}\trg(x) \in \textstyle{\KO^{m-1}}(\Omega Y).
$$
\end{Construction}

Following \cite{BERW 16} we apply Construction \ref{Construction: relaive k-theory construction} to the bundle projection
\begin{equation} \label{equation: bundle projection unstable}
\pi: \mathcal{M}^{\psc, n}_{\theta, 2n}(P, g_{P}) \longrightarrow \mathcal{M}^{n}_{\theta, 2n}(P).
\end{equation}
Choose once and for all an element $M \in \mathcal{M}^{n}_{\theta, 2n}(P)$. 
We will keep this element fixed for the rest of this section. 
The fibre of $\pi$ over $M$ is given by the space $\mathcal{R}^{+}(M)_{g_{P}}$.
Recall the index difference class, $\inddiff \in \KO^{2n+1}(\mathcal{R}^{+}(M)_{g_{P}})$ (see \cite{E 16} or \cite{BERW 16} for the definition).
Let, 
$$\rho: \mathcal{M}^{n}_{\theta, 2n}(P) \longrightarrow \Omega^{\infty}\MT\theta_{2n}$$
denote the \textit{scanning map} (alias parametrized Pontryagin-Thom map) from \cite{GRW 14}.
The proposition below follows from the construction in \cite[Section 3.8.4]{BERW 16}.
\begin{proposition} \label{proposition: relative index difference}
There exists a class $\beta \in \KO^{2n}(\pi)$ that satisfies the following conditions:
\begin{enumerate} \itemsep.2cm
\item[(a)] $\bas(\beta) = \rho^{*}(\mathcal{A}_{2n})$,
\item[(b)] $\trg(\beta) = \inddiff$. 
\end{enumerate}
\end{proposition}

We will use a stable version of the class $\beta \in \KO^{2n}(\pi)$ from the above proposition.
Recall the map, 
$
\pi^{\stb}: \mathcal{M}^{\psc, n}_{\theta, 2n}(P, g_{P})^{\stb} \longrightarrow \mathcal{M}^{n}_{\theta, 2n}(P)^{\stb}
$
defined by taking the direct limit of the direct system (\ref{equation: direct limit of the projections pi}). 
By what was proven about this map in the previous section it follows that the 
natural map,
$$\mathcal{R}^{+}(M)_{g_{P}}^{\stb} \stackrel{\simeq} \longrightarrow \textstyle{\hofibre_{M}}(\pi^{\stb}),$$
is a weak homotopy equivalence. 
The scanning maps $\rho$ induce a \textit{stable scanning map},
$$
\rho^{\stb}: \mathcal{M}^{n}_{\theta, 2n}(P)^{\stb} \longrightarrow \Omega^{\infty}\MT\theta_{2n},
$$
which is an acyclic map by \cite{GRW 16}.
The following proposition is proven using \cite[Proposition 3.8.6]{BERW 16}, and it is essentially the same as \cite[Proposition 4.2.6]{BERW 16}. 
\begin{proposition} \label{proposition: stable beta class}
There exists a class $\beta^{\stb} \in \KO^{2n}(\pi^{\stb})$ with the following properties:
\begin{enumerate} \itemsep.2cm
\item[(ii)] $\bas(\beta^{\stb}) \; = \; (\rho^{\stb})^{*}(\mathcal{A}_{2n})$;
\item[(iii)] $i^{*}\trg(\beta^{\stb}) \; = \; \inddiff$, where $i: \mathcal{R}^{+}(M)_{g_{P}} \hookrightarrow \mathcal{R}^{+}(M)^{\stb}_{g_{P}}$ is the inclusion, which is a weak homotopy equivalence. 
\end{enumerate}
\end{proposition}

In view of the above proposition, we will denote,
$
\textstyle{\inddiff^{\stb}} := \trg(\beta^{\stb}),
$
which represents an element of the group $\KO^{2n+1}(\mathcal{R}^{+}(M)^{\stb}_{g_{P}})$.
Recall from Remark \ref{remark: equivalence of homotopy fibre to MT}, the commutative diagram 
\begin{equation} \label{equation: commutative diagram with scanning}
\xymatrix{
\mathcal{M}^{n}_{\theta, 2n}(P)^{\stb} \ar[rr]^{\rho^{\stb}} \ar[dr] && \Omega^{\infty}\MT\theta_{2n} \\
& \mb{F}_{\theta, 2n}, \ar[ur]^{\simeq} &
}
\end{equation}
where the descending-diagonal map is the inclusion (which is acyclic), and the ascending diagonal map is the weak homotopy equivalence from Remark \ref{remark: equivalence of homotopy fibre to MT}. 
Let us now define a map 
\begin{equation}
\mathcal{A}'_{2n}: \mb{F}_{\theta, 2n} \longrightarrow \Omega^{\infty+2n}\KO
\end{equation}
by precomposing $\mathcal{A}_{2n}$ with the weak homotopy equivalence 
$\mb{F}_{\theta, 2n} \stackrel{\simeq} \longrightarrow \Omega^{\infty}\MT\theta_{2n}.$
We are now ready to prove Theorem \ref{theorem: factorization of index difference}, part (a).
\begin{proof}[Proof of Theorem \ref{theorem: factorization of index difference}, part (a)]
To prove the theorem it will suffice to show that 
$$
\xymatrix{
\Omega\mathcal{W}^{n-1, c}_{\theta, 2n}(P) \ar[r]^{j_{2n}} & \mathcal{R}^{+}(M)_{g_{P}} \ar[r]^{\inddiff \ \ \ } & \Omega^{\infty+2n+1}\KO
}
$$
agrees with the map $\Omega\bar{\mathcal{A}}_{2n+1}$, where $j_{2n}$ is the fibre transport map from Section \ref{subsection: the fibre transport} (see Remark \ref{remark: only need to verify for k = n-1}).
Recall from Theorem \ref{theorem: homology fibration psc} the homotopy cartesian diagram 
\begin{equation} \label{equation: cartesian diagram with moduli space}
\xymatrix{
\mathcal{M}^{\psc, n}_{\theta, 2n}(P, g_{P})^{\stb} \ar[rr]^{h^{\psc}} \ar[d]^{\pi^{\stb}} &&  \mb{F}^{\psc}_{\theta, 2n} \ar[d]^{\widetilde{F}^{\stb}} \\
\mathcal{M}^{n}_{\theta, 2n}(P)^{\stb} \ar[rr]^{h} && \mb{F}_{\theta, 2n},
}
\end{equation}
where the horizontal maps $h$ and $h^{\psc}$ are acyclic maps. 
The homotopy fibre of $\pi^{\stb}$ (and hence of $\widetilde{F}^{\stb}$) is given by the space $\mathcal{R}^{+}(M)^{\stb}_{g_{P}}$.
It is a well known fact that if $X \longrightarrow Y$ is an acyclic map, then for any generalized cohomology theory $E^{*}$, the induced map $E^{*}(Y) \longrightarrow E^{*}(X)$ is an isomorphism (this fact can be deduced from \cite[Proposition 3.1]{HH 79}).
Since both horizontal maps $h$ and $h^{\psc}$ in the diagram (\ref{equation: cartesian diagram with moduli space}) are acyclic, it follows that the induced map 
$$
\bar{h}^{*}: \textstyle{\KO^{-2n}}(F^{\stb}) \longrightarrow \textstyle{\KO^{-2n}}(\pi^{\stb})
$$
is an isomorphism. 
So, it follows that there exists a unique class $\beta' \in \KO^{-2n}(F^{\stb})$ such that 
$\bar{h}^{*}(\beta') = \beta^{\stb}.$
By naturality of Construction \ref{Construction: relaive k-theory construction} it follows that
$\trg(\beta') = \textstyle{\inddiff^{\stb}},$
as elements in the group $\KO^{2n+1}(\mathcal{R}^{+}(M)^{\stb}_{g_{P}})$.
Furthermore, we have,
$$h^{*}\bas(\beta') = \bas(\textstyle{\beta^{\stb}}) = \rho^{*}(\mathcal{A}_{2n}).$$ 
Combining the above equation with commutativity of (\ref{equation: commutative diagram with scanning}), it follows that, 
$$
\bas(\beta') = \mathcal{A}'_{2n} \in \textstyle{\KO^{-2n}}(\mb{F}_{\theta, 2n}), 
$$
and thus we have
$$
\iota^{*}(\textstyle{\inddiff^{\stb}}) \; = \; \Omega\mathcal{A}'_{2n},
$$
where $\iota: \Omega\mb{F}_{\theta, 2n} \longrightarrow \mathcal{R}^{+}(M)^{\stb}_{g_{P}}$ is the fibre transport map.
It follows from this that the top row of the commutative diagram,
$$
\xymatrix{
 \Omega\mb{F}_{\theta, 2n} \ar[r]^{\iota} & \mathcal{R}^{+}(M)^{\stb}_{g_{P}} \ar[rr]^{\inddiff^{\stb}}  && \Omega^{\infty +2n +1}\KO, \\
 & & \mathcal{R}^{+}(M)_{g_{P}} \ar[ul]_{\simeq} \ar[ur]^{\inddiff} & 
}
$$
agrees with $\Omega\mathcal{A}'_{2n}$.
The theorem then follows by observing that the map  $\iota: \Omega\mb{F}_{\theta, 2n} \longrightarrow \mathcal{R}^{+}(M)^{\stb}_{g_{P}}$ factors through the fibre transport map 
$
\Omega\mathcal{W}^{n-1, c}_{\theta, 2n}(P)^{\stb} \longrightarrow \mathcal{R}^{+}(M)^{\stb}_{g_{P}}.
$
\end{proof}

\section{Spaces of Manifolds With Free Boundary} \label{section: Manifolds with boundary equipped with surgery data}
We now begin our work on proving the odd-dimensional case of the main theorem stated in the introduction. 
To do this we will need to construct a version of Definition \ref{defn: colimit decomp of W} for spaces of manifolds with free-boundary.  
\subsection{Homotopy colimit decompositions} \label{subsection: spaces of manifolds with boundary}
In this section we construct a functor analogous to Definition \ref{defn: colimit decomp of W} for manifolds with free boundary.
We will need to first define a Grassmannian manifold of ``half vector spaces''.
\begin{defn} \label{defn: grassmannian of half-planes}
The space $G^{\partial, \mf}_{d+1}(\R^{\infty})_{\loc}$ consists of tuples $(V, \sigma, \phi)$ where:
\begin{enumerate}
\item[(i)] $V \leq \R^{\infty}$ is a $(d+1)$-dimensional vector subspace equipped with a $\theta$-structure $\hat{\ell}$;
\item[(ii)] $\sigma: V\otimes V \longrightarrow \R$ is a non-degenerate, symmetric bilinear form;
\item[(iii)] $\phi: V \longrightarrow \R$ is a non-zero linear functional with the property that 
$V^{-} \leq \Ker(\phi)$, where $V^{-} \leq V$ is the negative eigenspace associated to $\sigma$.
\end{enumerate}
For $(V, \sigma, \phi) \in G^{\partial, \mf}_{d+1}(\R^{\infty})_{\loc}$ we let $\bar{V}$ denote $\Ker(\phi: V \rightarrow \R)$ and we denote by $\bar{\sigma}: \bar{V}\otimes\bar{V} \longrightarrow \R$ the bilinear form obtained by restricting $\sigma$.
Condition (iii) implies that $\bar{\sigma}$ is non-degenerate and that $\text{index}(\sigma) = \text{index}(\bar{\sigma})$.
\end{defn}

We will use the space $G^{\partial, \mf}_{d+1}(\R^{\infty})_{\loc}$ to parametrize surgeries on the boundary of manifolds in way similar to what was done in Section \ref{subsection: spaces of manifolds equipped with surgery data}. 
Before doing this we need to cover a preliminary construction. 
\begin{Construction} \label{Construction: half disk with collars}
Let $(V, \sigma, \phi) \in G^{\partial, \mf}_{d+1}(\R^{\infty})_{\loc}$.
As before we may consider the spaces $D(V^{\pm})$ and $S(V^{\pm})$. 
We will need to use the functional $\phi$ to specify some subspaces. 
We define,
\begin{equation} \label{equation: half-disk}
D_{+}(V^{+}) := D(V^{+})\cap\phi^{-1}([0, \infty)). 
\end{equation}
This space is a half-disk of dimension $d+1 - \text{index}(\sigma)$.
This is a manifold with corners, and the boundary $\partial D_{+}(V^{+})$ decomposes as the union 
$$\partial D_{+}(V^{+}) = \partial_{0}D_{+}(V^{+})\cup\partial_{1}D_{+}(V^{+}),$$ 
where:
$$
\begin{aligned}
\partial_{0}D_{+}(V^{+}) \; &= \; D_{+}(V^{+})\cap\phi^{-1}(0), \\
\partial_{1}D_{+}(V^{+}) \; &= \; S(V^{+})\cap\phi^{-1}([0, \infty)).
\end{aligned}
$$
It will be important to have a version of the half-disk $D_{+}(V^{+})$ that is equipped with an internal collar determined by the functional $\phi$. 
Let $v_{\phi} \in V^{+}$ be the vector dual to the functional $\phi$ with respect to the inner product $\langle \--, \--\rangle$ coming from the ambient space. 
Fix once and for all a smooth non-decreasing function $\rho: [0, 1] \longrightarrow [0, 1]$ that satisfies $\rho(t) = 0$ for all $t < 1/8$ and $\rho(t) = t$ for all $t > 1/4$. 
We then define,
\begin{equation}
D^{c}_{+}(V^{+}) \; = \; \{ v \in V^{+} \; | \; |\pi_{\phi}(v)|^{2} + \rho(\phi(v))^{2} \leq 1 \; \},
\end{equation}
where $\pi_{\phi}: V^{+} \longrightarrow \bar{V}^{+}$ is the orthogonal projection away from $v_{\phi}$, using the orthogonal direct sum decomposition, $V^{+} = \bar{V}^{+}\oplus\langle v_{\phi} \rangle$. 
Clearly, $D^{c}_{+}(V^{+})$ is diffeomorphic to $D_{+}(V^{+})$ as a manifold with corners. 
The key property of this space that we will latter use is:
$$
D^{c}_{+}(V^{+})\cap\phi^{-1}([0, \varepsilon)) \; = \; \partial_{0}D^{c}_{+}(V^{+})\times[0, \varepsilon) \; = \; D(\bar{V}^{+})\times[0, \varepsilon),
$$
for all $0 < \varepsilon < 1/4$.
As before we have 
$$
\partial D^{c}_{+}(V^{+}) \; = \; \partial_{0}D^{c}_{+}(V^{+})\cup\partial_{1}D^{c}_{+}(V^{+}),
$$
where $ \partial_{0}D^{c}_{+}(V^{+})$ and $\partial_{1}D^{c}_{+}(V^{+})$ are defined in the same way as above. 
\end{Construction}

We will need to fix some further notation.
\begin{Notation} \label{Notation: half space surgery}
Let $\mb{t} \in \mathcal{K}_{d}$. 
We will need to consider elements of the mapping space,
$$
(G^{\partial, \mf}_{d+1}(\R^{\infty})_{\loc})^{\mb{t}} = \Maps(\mb{t}, G^{\partial, \mf}_{d+1}(\R^{\infty})_{\loc}).
$$
Given $(V, \sigma, \phi) \in (G^{\partial}_{d+1}(\R^{\infty})_{\loc})^{\mb{t}}$ we define, 
$$
D(V^{-})\times_{\mb{t}}D^{c}_{+}(V^{+}) \; = \; \coprod_{i \in \mb{t}}D(V^{-}(i))\times D^{c}_{+}(V^{+}(i)).
$$
The spaces $S(V^{-})\times_{\mb{t}}D^{c}_{+}(V^{+})$, $D(V^{-})\times_{\mb{t}}\partial_{i}D^{c}_{+}(V^{+})$, etc... are defined similarly. 
\end{Notation}

For an integer $k \in \Z_{\geq -1}$, let $\widehat{\mathcal{K}}^{k}_{d}$ denote the product category $\mathcal{K}^{k}_{d}\times\mathcal{K}^{k}_{d-1}$. 
Objects of $\widehat{\mathcal{K}}^{k}_{d}$ will be denoted by pairs $(\mb{t}_{0}, \mb{t}_{1})$ where $\mb{t}_{0} \in \mathcal{K}^{k}_{d}$ and $\mb{t}_{1} \in \mathcal{K}^{k}_{d-1}$.

Recall from Definition \ref{Convention: tangential structures} the space $\mathcal{M}^{\partial}_{\theta, d}$.
For an integer $k$, we define $\mathcal{M}^{\partial, k}_{\theta, d} \subset \mathcal{M}^{\partial}_{\theta, d}$ to be the subspace consisting of those $M$ for which both maps, 
$$\ell_{M}: M \longrightarrow B \quad \text{and} \quad \ell_{M}|_{\partial M}: \partial M \longrightarrow B,$$ are $k$-connected.
The main definition of this section is given below.
\begin{defn} \label{defn: colimit decomp of W w boundary}
Fix an integer $k \in \Z_{\geq -1}$. 
Let $\mb{t} = (\mb{t}_{0}, \mb{t}_{1}) \in \widehat{\mathcal{K}}^{k}_{d}$. 
The space $\mathcal{W}^{\partial, k}_{\theta, d}(\mb{t})$ consists of tuples 
$\left(M, (V_{0}, \sigma_{0}, \phi_{0}), (V_{1}, \sigma_{1}), e\right)$  
where:
\begin{enumerate} \itemsep.3cm
\item[(i)] $M$ is an element of the space $\mathcal{M}^{\partial, k}_{\theta, d}$;
\item[(ii)] 
$(V_{0}, \sigma_{0}, \phi_{0}) \in (G_{d+1}^{\partial, \mf}(\R^{\infty})_{\loc})^{\mb{t}_{0}}$ and $(V_{1}, \sigma_{1}) \in \left(G^{\mf}_{d+1}(\R^{\infty})_{\loc}\right)^{\mb{t}_{1}}$
are elements with the property that,
$$\delta(t_{0}) = \text{index}(\sigma_{0}) \quad  \text{and} \quad 
\delta(t_{1}) = \text{index}(\sigma_{1})$$
for all $(t_{0}, t_{1}) \in (\mb{t}_{0}, \mb{t}_{1}) = \mb{t}$.
\item[(ii)] $e = (e_{0}, e_{1})$ is a pair of embeddings
$$\begin{aligned}
e_{0}: D(V_{0}^{-})\times_{\mb{t}_{0}}D^{c}_{+}(V_{0}^{+}) \; &\longrightarrow \; (-\infty, 0]\times\R^{\infty-1},\\
e_{1}:  D(V_{1}^{-})\times_{\mb{t}_{1}}D(V^{+}_{1}) &\longrightarrow (-\infty, 0)\times\R^{\infty-1},
\end{aligned}$$
subject to the following further conditions:
\begin{enumerate} \itemsep.2cm
\item[(a)] The embedding $e_{0}$ satisfies 
$$
e^{-1}_{0}((-\infty, 0)\times\{0\}\times\R^{\infty-1}) = D(V_{0}^{-})\times_{\mb{t}_{0}}\partial_{0}D^{c}_{+}(V_{0}^{+}),
$$
and furthermore it respects collars. 
We also require: 
$$\begin{aligned}
e^{-1}_{0}(M) &= S(V^{-}_{0})\times_{\mb{t}_{0}}D^{c}_{+}(V_{0}^{+}), \\
e^{-1}_{0}(\partial M) &= S(V^{-}_{0})\times_{\mb{t}_{0}}\partial_{0}D^{c}_{+}(V_{0}^{+}).
\end{aligned}$$
\item[(b)] The embedding $e_{1}$ satisfies,
$e^{-1}_{1}(M) = S(V^{-}_{1})\times_{\mb{t}_{1}}D(V^{+}_{1}).$
\end{enumerate}
\end{enumerate}
For elements of $\mathcal{W}^{\partial, k}_{\theta, d}(\mb{t})$, we will usually compress the notation and denote, 
$$(M, (V, \sigma), e) := \left(M, (V_{0}, \sigma_{0}, \phi_{0}), (V_{1}, \sigma_{1}), e\right).$$
\end{defn}
It remains to be described how to make the correspondence $\mb{t} \mapsto \mathcal{W}^{\partial, k}_{\theta, d}(\mb{t})$ into a functor on $\widehat{\mathcal{K}}^{k}_{d}$. 
Let $(j, \varepsilon): \mb{s} \longrightarrow \mb{t}$ be a morphism in $\widehat{\mathcal{K}}^{k}_{d}$. 
Since $\widehat{\mathcal{K}}^{k}_{d}$ is the product category $\mathcal{K}^{k}_{d}\times\mathcal{K}^{k}_{d-1}$, it will suffice to describe the induced map $(j, \varepsilon)^{*}: \mathcal{W}^{\partial, k}_{\theta, d}(\mb{t}) \longrightarrow \mathcal{W}^{\partial, k}_{\theta, d}(\mb{s})$ in just two cases: 
\begin{enumerate} \itemsep.2cm
\item[(1)] $j_{0}: \mb{s}_{0} \longrightarrow \mb{t}_{0}$ is the identity, and $j_{1}: \mb{s}_{1} \longrightarrow \mb{t}_{1}$ is an inclusion whose complement contains a single element;
\item[(2)] $j_{1}: \mb{s}_{1} \longrightarrow \mb{t}_{1}$ is the identity, and $j_{0}: \mb{s}_{0} \longrightarrow \mb{t}_{0}$ is an inclusion whose complement contains a single element.
\end{enumerate} 
In case (1), the induced map $(j, \varepsilon)^{*}$ is defined in exactly the same way as was done in Section \ref{subsection: spaces of manifolds equipped with surgery data}, following Definitions \ref{defn: pullback +1} and \ref{defn: pullback -1}. 
Let us focus on case (2). 
We let the single element of $\mb{t}_{0}\setminus j_{0}(\mb{s}_{0})$ be denoted by $a$. 
This case breaks down into two subcases: $\varepsilon_{0}(a) = +1$ and $\varepsilon_{0}(a) = -1$. 
The construction of $(j, \varepsilon)^{*}$ is described in the proceeding definitions below for each of these two subcases. 

\begin{defn}[$\varepsilon_{0}(a) = +1$]  \label{defn: pullback +1}
Let $\varepsilon_{0}(a) = +1$. 
We describe the induced map 
$
(j, \varepsilon)^{*}: \mathcal{W}^{\partial, k}_{\theta, d}(\mb{t}) \longrightarrow \mathcal{W}^{\partial, k}_{\theta, d}(\mb{s}).
$
Let $(M, (V, \sigma), e)$ be an element of $\mathcal{W}^{\partial, k}_{\theta, d}(\mb{t})$.
Map this to an element of $\mathcal{W}^{\partial, k}_{\theta, d}(\mb{s})$ by restricting $e_{0}$ the subspace,
$
D(V^{-})\times_{\mb{s}_{0}}D^{c}_{+}(V^{+})  \subset D(V^{-})\times_{\mb{t}_{0}}D^{c}_{+}(V^{+}),
$
while keeping $((V_{1}, \sigma_{1}), e_{1})$ the same as before (by our assumption above, the map $j_{1}: \mb{s}_{1} \longrightarrow \mb{t}_{1}$ is the identity). 
\end{defn}

\begin{defn}[$\varepsilon_{0}(a) = -1$] \label{defn: pullback -1} 
Let $\varepsilon_{0}(a) = -1$. 
The induced map 
$
\mathcal{W}^{\partial, k}_{\theta, d}(\mb{t}) \longrightarrow \mathcal{W}^{\partial, k}_{\theta, d}(\mb{s})
$
is defined as follows:
Let $(M, (V, \sigma), e)$ be an element of $\mathcal{W}^{\partial, k}_{\theta, d}(\mb{t})$.
Map this to the element $(\widetilde{M}, (\widetilde{V}, \widetilde{\sigma}), \widetilde{e})$ in $\mathcal{W}^{\partial, k}_{\theta, d}(\mb{s})$ where: 
\begin{enumerate} \itemsep.2cm
\item[(i)] $\widetilde{M}$ is the element of $\mathcal{M}^{\partial, k}_{\theta, d}$ given by the union
$$
M\setminus e_{0}(S(V^{-})\times_{a}D^{c}_{+}(V^{+}))\bigcup e_{0}(D(V^{-})\times_{a}\partial_{1}D^{c}_{+}(V^{+})),
$$
where $D(V^{-})\times_{a}D^{c}_{+}(V^{+})$ is the component of $D(V^{-})\times_{\mb{t}_{0}}D^{c}_{+}(V^{+})$ over the point $a \in \mb{t}_{0}\setminus\mb{s}_{0}$.
\item[(ii)] $\widetilde{e}_{0}$ and $(\widetilde{V}_{0}, \widetilde{\sigma}_{0})$ are obtained from $e_{0}$ and $(V_{0}, \sigma_{0})$ by restriction to $\mb{s}_{0}$. 
\item[(iii)] $\widetilde{e}_{1}$ and $(\widetilde{V}_{1}, \widetilde{\sigma}_{1})$ are kept equal to $e_{1}$ and $(V_{1}, \sigma_{1})$ respectively.
\end{enumerate}
\end{defn}

The above definitions and above discussion make the assignment $\mb{t} \mapsto \mathcal{W}^{\partial, k}_{\theta, d}(\mb{t})$ into a contravariant functor on $\widehat{\mathcal{K}}^{k}_{d}$.
We define,
$\mathcal{W}^{\partial, k}_{\theta, d} := \displaystyle{\hocolim_{\mb{t} \in \widehat{\mathcal{K}}^{k}_{d}}}\mathcal{W}^{\partial, k}_{\theta, d}(\mb{t}).
$

\begin{remark} \label{remark: alternative model for cobordism morse category}
In analogy with Definition \ref{defn: colimit decomp of W}, and the weak homotopy equivalence $B\Cob_{\theta, d}^{\mf, k} \simeq \mathcal{W}^{k}_{\theta, d}$, the space $\mathcal{W}^{\partial, k}_{\theta, d}$ serves as our alternative model for the cobordism category $\Cob_{\theta, d}^{\partial, \mf, k}$ discussed in the introduction. 
As for the case with $\mathcal{W}^{k}_{\theta, d}$ and $\Cob_{\theta, d}^{\mf, k}$, the space $\mathcal{W}^{\partial, k}_{\theta, d}$ is much easier to work with than $\Cob_{\theta, d}^{\partial, \mf, k}$, and thus all of our arguments will be conducted using $\mathcal{W}^{\partial, k}_{\theta, d}$ as a stand-in.
The only advantage provided by the cobordism categories is that they have a simpler definition and thus are easier to use in a casual explanation of our results in the introduction. 
\end{remark}

\subsection{The boundary functor} \label{subsection: the boundary functor}
Fix $k \in \Z_{\geq -1}$. 
For each $\mb{t} \in \widehat{\mathcal{K}}^{k}_{d}$
there is a map
$$
b_{\mb{t}}: \mathcal{W}^{\partial, k}_{\theta, d}(\mb{t}) \longrightarrow \mathcal{W}^{k}_{\theta, d-1}(\mb{t}_{0}),
$$
defined by sending $(M, (V, \sigma), e) \in \mathcal{W}^{\partial, k}_{\theta, d}(\mb{t})$ to the element $(\partial M, (\bar{V}_{0}, \bar{\sigma}_{0}), \partial_{0}e_{0})$, where $(\bar{V}_{0}, \bar{\sigma}_{0})$ is obtained from $(V_{0}, \sigma_{0}, \phi_{0})$ by restricting to the kernel of the associated linear functional $\phi_{0}$ (recall Definition \ref{defn: grassmannian of half-planes}), and $\partial_{0}e_{0}$ is the restriction of $e_{0}$ to the subspace,
$$D(V^{-})\times_{t_{0}}\partial_{0}D^{c}_{+}(V^{+}) \subset D(V^{-})\times_{t_{0}}D^{c}_{+}(V^{+}).$$
By precomposing the functor $\mb{t} \mapsto \mathcal{W}^{k}_{\theta, d-1}(\mb{t})$ with the projection $\widehat{K}_{d}^{k} \mapsto \mathcal{K}_{d-1}^{k}$, $\mathcal{W}^{k}_{\theta, d-1}(\--)$ may be considered to be a functor on $\widehat{\mathcal{K}}^{k}_{d}$, and thus $b_{\mb{t}}$ can be considered to be a natural transformation of functors on the category $\widehat{\mathcal{K}}^{k}_{d}$.
The maps $b_{\mb{t}}$ induce a map between the homotopy colimits,
\begin{equation} \label{equation: induced maps of homotopy colimits}
b: \mathcal{W}^{\partial, k}_{\theta, d} \longrightarrow \mathcal{W}^{k}_{\theta, d-1}.
\end{equation}
The boundary map $b$ defined above is analogous to
$r: \mathcal{D}^{\partial, \mf, k}_{\theta, d+1} \longrightarrow \mathcal{D}^{\mf, k}_{\theta, d}$ from Definition \ref{defn: space of long manifolds w boundary morse}. 

\begin{theorem} \label{theorem: equivalence with long manifolds with boundary}
For all $k \in \Z_{\geq -1}$ there is a zig-zag of weak homotopy equivalences,
$$
\xymatrix{
\mathcal{W}^{\partial, k}_{\theta, d} & \mathcal{L}^{\partial, k}_{\theta, d} \ar[l]_{\simeq} \ar[r]^{\simeq} & \mathcal{D}^{\partial, \mf, k}_{\theta, d+1},
}
$$
such that the following diagram is homotopy commutative,
\begin{equation} \label{equation: W to D to L diagram}
\xymatrix{
\mathcal{W}^{\partial, k}_{\theta, d} \ar[d]^{b} &  \mathcal{L}^{\partial, k}_{\theta, d} \ar[l]_{\simeq} \ar[r]^{\simeq} & \mathcal{D}^{\partial, \mf, k}_{\theta, d+1} \ar[d]^{r} \\
 \mathcal{W}^{k}_{\theta, d-1} &  \mathcal{L}^{k}_{\theta, d-1} \ar[l]_{\simeq} \ar[r]^{\simeq} & \mathcal{D}^{\mf, k}_{\theta, d},
}
\end{equation}
where the bottom-horizontal arrows are the weak homotopy equivalences from (\ref{equation: W to L to D zig zag}).
\end{theorem}
Theorem \ref{theorem: equivalence with long manifolds with boundary} is analogous to \cite[Theorem 4.1]{P 17} and it is proven in the same way, by mimicking the constructions in \cite[Section 5]{MW 07}). 
To avoid too much redundancy we give the proof in Appendix \ref{section: long manifolds to W partial} omitting most of the details.
\begin{remark} \label{remark: necessity of above proposition}
For our purposes (namely proving Theorem B from the introduction), proving the full claim of Theorem \ref{theorem: equivalence with long manifolds with boundary} is not actually necessary. 
All that is needed is the existence of a zig-zag diagram, 
$\xymatrix{
\mathcal{W}^{\partial, k}_{\theta, d} & \mathcal{L}^{\partial, k}_{\theta, d} \ar[l]_{\simeq} \ar[r] & \mathcal{D}^{\partial, \mf, k}_{\theta, d+1},
}$
where the first arrow is a weak homotopy equivalence.
It is only necessary that the second map exists and that it is defined so that (\ref{equation: W to D to L diagram}) is homotopy commutative.
\end{remark}
We now move on to study the map $b: \mathcal{W}^{\partial, k}_{\theta, d} \longrightarrow \mathcal{W}^{k}_{\theta, d-1}$ from (\ref{equation: induced maps of homotopy colimits}). 
We will restrict our attention to a particular dimensional case. 
For the following theorem, we will let $d = 2n$ and $k = n-1$. 
Fix $P \in \mathcal{M}^{n-1}_{\theta, 2n-1}$. 
Since $\mathcal{M}^{n-1}_{\theta, 2n-1} = \mathcal{W}^{n-1}_{\theta, 2n-1}(\emptyset)$ (where $\emptyset$ is considered the empty object of $\mathcal{K}^{n-1}_{2n-1}$), we may consider $P$ to be an element of $\mathcal{W}^{n-1}_{\theta, 2n-1}$, and thus we may 
consider the homotopy fibre of the map 
$b: \mathcal{W}^{\partial, n-1}_{\theta, 2n} \longrightarrow \mathcal{W}^{n-1}_{\theta, 2n-1}$
over $P$.
The inclusion 
$\mathcal{W}^{n-1}_{\theta, 2n}(P) \hookrightarrow \mathcal{W}^{\partial, n-1}_{\theta, 2n}$
factors through the fibre $b^{-1}(P)$, and thus there is an induced map 
\begin{equation} \label{equation: induced map into homotopy fibre}
\mathcal{W}^{n-1}_{\theta, 2n}(P) \; \longrightarrow \; \textstyle{\hofibre_{P}}\left(b: \mathcal{W}^{\partial, n-1}_{\theta, 2n} \rightarrow \mathcal{W}^{n-1}_{\theta, 2n-1}\right).
\end{equation}
\begin{theorem} \label{theorem: boundary map is quasifibration d = 2n-1}
Let $P \in \mathcal{M}^{n-1}_{\theta, 2n-1}$.
The map (\ref{equation: induced map into homotopy fibre}),
$$
\mathcal{W}^{n-1}_{\theta, 2n}(P) \; \longrightarrow \; \textstyle{\hofibre_{P}}\left(b: \mathcal{W}^{\partial, n-1}_{\theta, 2n} \rightarrow \mathcal{W}^{n-1}_{\theta, 2n-1}\right),
$$
is a weak homotopy equivalence. 
Thus, the following sequence, 
$$
\xymatrix{
\mathcal{W}^{n-1}_{\theta, 2n}(P)  \ar[r] & \mathcal{W}^{\partial, n-1}_{\theta, 2n} \ar[r] & \mathcal{W}^{n-1}_{\theta, 2n-1},
}
$$
is a homotopy fibre sequence.
\end{theorem}

Proving the above theorem will take a bit of work and the proof is carried out over the rest of the section. 
To begin, recall that
$
\mathcal{W}^{\partial, n-1}_{\theta, 2n} \; = \; \displaystyle{\hocolim_{\mb{t} \in \widehat{\mathcal{K}}_{2n}^{n-1}}}\mathcal{W}^{\partial, n-1}_{\theta, 2n}(\mb{t}).
$
Since $\widehat{\mathcal{K}}_{2n}^{n-1}$ is a product category $\mathcal{K}^{n-1}_{2n}\times\mathcal{K}^{n-1}_{2n-1}$, the above homotopy colimit may be computed in stages. 
We have a homeomorphism,
$$
\mathcal{W}^{\partial, n-1}_{\theta, 2n} \; \cong \; \hocolim_{\mb{t}_{0} \in \mathcal{K}^{n-1}_{2n}}\hocolim_{\mb{t}_{1} \in \mathcal{K}^{n-1}_{2n-1}}\mathcal{W}^{\partial, n-1}_{\theta, 2n}(\mb{t}_{0}, \mb{t}_{1}).
$$
Furthermore, the order in which the homotopy colimits are constructed can be exchanged, 
$$
 \hocolim_{\mb{t}_{0} \in \mathcal{K}^{n-1}_{2n}}\hocolim_{\mb{t}_{1} \in \mathcal{K}^{n-1}_{2n-1}}\mathcal{W}^{\partial, n-1}_{\theta, 2n}(\mb{t}_{0}, \mb{t}_{1}) \; \cong \; \hocolim_{\mb{t}_{1} \in \mathcal{K}^{n-1}_{2n-1}}\hocolim_{\mb{t}_{0} \in \mathcal{K}^{n-1}_{2n}}\mathcal{W}^{\partial, n-1}_{\theta, 2n}(\mb{t}_{0}, \mb{t}_{1}).
$$
Now, let us fix $\mb{t}_{0} \in \mathcal{K}^{n-1}_{2n-1}$ and consider the map 
\begin{equation} \label{equation: restriction with limit over t-1}
\hocolim_{\mb{t}_{1} \in \mathcal{K}^{n-1}_{2n}}\mathcal{W}^{\partial, n-1}_{\theta, 2n}(\mb{t}_{0}, \mb{t}_{1}) \; \longrightarrow \; \mathcal{W}^{n-1}_{\theta, 2n-1}(\mb{t}_{0})
\end{equation}
induced from the boundary maps $b_{\mb{t}}$.
We have the following lemma:
\begin{lemma} \label{lemma: partial homotopy fibre sequence}
Let $\mb{t}_{0} \in \mathcal{K}^{n-1}_{2n-1}$ be as chosen above.
Fix an element, 
$(M, (V_{0}, \sigma_{0}), e_{0}) \in \mathcal{W}^{n-1}_{\theta, 2n-1}(\mb{t}_{0}).$
The homotopy fibre of (\ref{equation: restriction with limit over t-1}) over $(M, (V_{0}, \sigma_{0}), e_{0})$ is given by the space, 
$$
\mathcal{W}^{n-1}_{\theta, 2n}(M) = \displaystyle{\hocolim_{\mb{t}_{1} \in \mathcal{K}^{n-1}_{2n}}}\mathcal{W}^{n-1}_{\theta, 2n}(M; \mb{t}_{1}).
$$
\end{lemma}
\begin{proof}
Let us denote $x := (M, (V_{0}, \sigma_{0}), e_{0})$.
By \cite[Lemma 6.11]{MW 07}, the space
$$
\textstyle{\hofibre_{x}}\left(\displaystyle{\hocolim_{\mb{t}_{1} \in \mathcal{K}^{n-1}_{2n}}}\mathcal{W}^{\partial, n-1}_{\theta, 2n}(\mb{t}_{0}, \mb{t}_{1}) \longrightarrow \mathcal{W}^{n-1}_{\theta, 2n-1}(\mb{t}_{0})\right)
$$
is weak homotopy equivalent to the space, 
$$
\hocolim_{\mb{t}_{1} \in \mathcal{K}^{n-1}_{2n}}\left[\textstyle{\hofibre_{x}}\left(\mathcal{W}^{\partial, n-1}_{\theta, 2n}(\mb{t}_{0}, \mb{t}_{1}) \; \longrightarrow \; \mathcal{W}^{n-1}_{\theta, 2n-1}(\mb{t}_{0}) \right)\right].
$$ 
Fix $\mb{t}_{1} \in \mathcal{K}^{n-1}_{2n}$ and write $\mb{t} = (\mb{t}_{0}, \mb{t}_{1})$.
The map
$$
b_{\mb{t}}: \mathcal{W}^{\partial, n-1}_{\theta, 2n}(\mb{t}_{0}, \mb{t}_{1}) \; \longrightarrow \; \mathcal{W}^{n-1}_{\theta, 2n-1}(\mb{t}_{0})
$$
is a Serre-fibration. 
The fibre of $b_{\mb{t}}$ over the element $x = (M, (V_{0}, \sigma_{0}), e_{0})$ is a subspace of $\mathcal{W}_{\theta, 2n}^{n-1}(M; \mb{t}_{1})$; let us describe this fibre explicitly. 
It is the subspace of $\mathcal{W}_{\theta, 2n}^{n-1}(M; \mb{t}_{1})$ consisting of those $(N, (V, \sigma), e)$ for which $N \in \mathcal{M}_{\theta, 2n}^{n-1}(M)$ contains the submanifold, 
$$e_{0}(S(V^{-})\times_{\mb{t}_{0}}D^{c}_{+}(V^{+}) \subset (-\infty, 0]\times\R^{\infty-1}.$$
Since $e_{0}(S(V^{-})\times_{\mb{t}_{0}}\partial_{0}D^{c}_{+}(V^{+}))$ is contained in $M$, and 
$S(V^{-})\times_{\mb{t}_{0}}D^{c}_{+}(V^{+})$ is homotopy equivalent to $S(V^{-})\times_{\mb{t}_{0}}\partial_{0}D^{c}_{+}(V^{+})$, it follows that the inclusion 
$$
b_{\mb{t}}^{-1}(M, (V_{0}, \sigma_{0}), e_{0}) \hookrightarrow \mathcal{W}_{\theta, 2n}^{n-1}(M; \mb{t}_{1}),
$$
is a weak homotopy equivalence. 
From this we obtain the weak homotopy equivalence, 
$$\textstyle{\hofibre_{x}}\left(\mathcal{W}^{\partial, n-1}_{\theta, 2n}(\mb{t}_{0}, \mb{t}_{1}) \; \longrightarrow \; \mathcal{W}^{n-1}_{\theta, 2n-1}(\mb{t}_{0}) \right) \; \simeq \; \mathcal{W}_{\theta, 2n}^{n-1}(M; \mb{t}_{1}).$$
Combining this with what was observed in the first part of the proof we obtain the weak homotopy equivalence,
$$
\hocolim_{\mb{t}_{1} \in \mathcal{K}^{n-1}_{2n}}\mathcal{W}_{\theta, 2n}^{n-1}(M; \mb{t}_{1}) \; \simeq \; \textstyle{\hofibre_{x}}\left(\displaystyle{\hocolim_{\mb{t}_{1} \in \mathcal{K}^{n-1}_{2n}}}\mathcal{W}^{\partial, n-1}_{\theta, 2n}(\mb{t}_{0}, \mb{t}_{1}) \longrightarrow \mathcal{W}^{n-1}_{\theta, 2n-1}(\mb{t}_{0})\right),
$$
and this concludes the proof of the lemma.
\end{proof}

Let $(j, \varepsilon): \mb{s}_{0} \longrightarrow \mb{t}_{0}$ be a morphism in $\mathcal{K}^{n-1}_{2n-1}$. 
Choose 
$(M, (V, \sigma), e) \in \mathcal{W}^{n-1}_{\theta, 2n-1}(\mb{t}_{0}),$ 
and then let $(M', (V', \sigma'), e') \in \mathcal{W}^{n-1}_{\theta, 2n-1}(\mb{s}_{0})$ be the image of $(M, (V, \sigma), e)$ under the map induced by the morphism $(j, \varepsilon)$.
This morphism induces a map of fibre sequences,
\begin{equation} \label{equation: induced map of fibre sequences}
\xymatrix{
\displaystyle{\hocolim_{\mb{t}_{1} \in \mathcal{K}^{n-1}_{2n}}}\mathcal{W}^{n-1}_{\theta, 2n}(M; \mb{t}_{1}) \ar[rr]^{(j, \varepsilon)_{*}} \ar[d]  && \displaystyle{\hocolim_{\mb{t}_{1} \in \mathcal{K}^{n-1}_{2n}}}\mathcal{W}^{n-1}_{\theta, 2n}(M'; \mb{t}_{1})  \ar[d] \\
\displaystyle{\hocolim_{\mb{t}_{1} \in \mathcal{K}^{n-1}_{2n}}}\mathcal{W}^{\partial, n-1}_{\theta, 2n}(\mb{t}_{0}, \mb{t}_{1}) \ar[rr]^{(j, \varepsilon)_{*}} \ar[d] && \displaystyle{\hocolim_{\mb{t}_{1} \in \mathcal{K}^{n-1}_{2n}}}\mathcal{W}^{\partial, n-1}_{\theta, 2n}(\mb{s}_{0}, \mb{t}_{1}) \ar[d] \\
\mathcal{W}^{n-1}_{\theta, 2n-1}(\mb{t}_{0}) \ar[rr]^{(j, \varepsilon)_{*}} && \mathcal{W}^{n-1}_{\theta, 2n-1}(\mb{s}_{0}).
}
\end{equation}
We want to apply Proposition \ref{proposition: homotopy colimit fibre sequence} to this situation. 
Indeed, we make the following claim:
\begin{claim} \label{claim: map between fibres is a weak equivalence}
For any morphism $(j, \varepsilon): \mb{s}_{0} \longrightarrow \mb{t}_{0}$ in $\mathcal{K}^{n-1}_{2n-1}$, the top-vertical map in (\ref{equation: induced map of fibre sequences}) is a weak homotopy equivalence. 
\end{claim}
Let us describe how to finish the proof of Theorem \ref{theorem: boundary map is quasifibration d = 2n-1} assuming the validity of the above claim. 
After finishing the proof of the theorem we will address the proof of Claim \ref{claim: map between fibres is a weak equivalence}.
\begin{proof}[Proof of Theorem \ref{theorem: boundary map is quasifibration d = 2n-1}, assuming Claim \ref{claim: map between fibres is a weak equivalence}]
To finish the proof of Theorem \ref{theorem: boundary map is quasifibration d = 2n-1} we simply apply Proposition \ref{proposition: homotopy colimit fibre sequence}. 
This proposition implies that the homotopy fibre of
$$\hocolim_{\mb{t}_{0} \in \mathcal{K}^{n-1}_{2n-1}}\hocolim_{\mb{t}_{1} \in \mathcal{K}^{n-1}_{2n}}\mathcal{W}^{\partial, n-1}_{\theta, 2n}(\mb{t}_{0}, \mb{t}_{1}) \; \longrightarrow \; \hocolim_{\mb{t}_{0} \in \mathcal{K}^{n-1}_{2n-1}}\mathcal{W}^{n-1}_{\theta, 2n-1}(\mb{t}_{0})$$
over $(M, (V, \sigma), e)$ is given by the space,
$
\displaystyle{\hocolim_{\mb{t}_{1} \in \mathcal{K}^{n-1}_{2n}}}\mathcal{W}^{n-1}_{\theta, 2n}(M; \mb{t}_{1}).
$
Theorem \ref{theorem: boundary map is quasifibration d = 2n-1} then follows from the homeomorphism, 
$$
\hocolim_{\mb{t}_{0} \in \mathcal{K}^{n-1}_{2n-1}}\hocolim_{\mb{t}_{1} \in \mathcal{K}^{n-1}_{2n}}\mathcal{W}^{\partial, n-1}_{\theta, 2n}(\mb{t}_{0}, \mb{t}_{1}) \; \cong \; \hocolim_{\mb{t} \in \widehat{\mathcal{K}}^{n-1}_{2n}}\mathcal{W}^{\partial, n-1}_{\theta, 2n}(\mb{t}).
$$
\end{proof}

We now address Claim \ref{claim: map between fibres is a weak equivalence}. 
It will be useful for us to reinterpret the map from the statement of the claim as one that is induced by concatenation with a morphism in the cobordism category $\Cob_{\theta, 2n}^{n-1}$. 

For the next definition, choose once and for all a diffeomorphism $\psi: (-\infty, 1] \stackrel{\cong} \longrightarrow (-\infty, 0]$ that is the identity on $(-\infty, -1)$. 
Let $\Psi: (-\infty, 1]\times\R^{\infty-1} \longrightarrow (-\infty, 0]\times\R^{\infty-1}$ be the diffeomorphism given by setting $\Psi = \psi\times\Id_{\R^{\infty-1}}$.
\begin{defn} \label{defn: cobordism category action}
Fix $\mb{t} \in \mathcal{K}^{n-1}_{2n}$.
Let $W: P \rightsquigarrow Q$ be a morphism in $\Cob_{\theta, 2n}^{n-1}$. 
The map, 
$$
\mb{S}_{W, \mb{t}}: \mathcal{W}^{n-1}_{\theta, 2n}(P; \mb{t}) \rightsquigarrow \mathcal{W}^{n-1}_{\theta, 2n}(Q; \mb{t}),
$$
is defined by the formula,
$$
(M, (V, \sigma), e) \; \mapsto \; (\Psi(M\cup_{P}W), (V, \sigma), \Psi\circ e),
$$
where $M\cup_{P}W \in \mathcal{M}_{\theta, 2n}^{n-1}(Q)$ is the element obtained by forming the union of $M \subset (-\infty, 0]\times\R^{\infty-1}$ with $W \subset [0, 1]\times\R^{\infty-1}$ along the boundary $P$.
Taking the homotopy colimit over $\mathcal{K}^{n-1}_{2n}$ yields the map, 
$
\mb{S}_{W}: \mathcal{W}^{n-1}_{\theta, 2n}(P) \rightsquigarrow \mathcal{W}^{n-1}_{\theta, 2n}(Q).
$
\end{defn}

Let, 
$(j, \varepsilon)_{*}: \mathcal{W}^{n-1}_{\theta, 2n}(M) \longrightarrow \mathcal{W}^{n-1}_{\theta, 2n}(M'),$
be the top-horizontal map from (\ref{equation: induced map of fibre sequences}), which is the subject of Claim \ref{claim: map between fibres is a weak equivalence}; we have reverted back to the notation, 
$\mathcal{W}^{n-1}_{\theta, 2n}(M) = \displaystyle{\hocolim_{\mb{t}_{1} \in \mathcal{K}^{n-1}_{2n}}}\mathcal{W}^{n-1}_{\theta, 2n}(M; \mb{t}_{1}).$
It will suffice to prove the claim in the special case that the morphism $(j, \varepsilon): \mb{s}_{0} \longrightarrow \mb{t}_{0}$ is chosen such that $\mb{t}_{0}\setminus j(\mb{s}_{0})$ contains a single point, $a$.
Indeed, any morphism can always be written as the composite of such morphisms. 
The proposition below follows directly from the definition of the map $(j, \varepsilon)_{*}$ and can be proven by hand. 
Furthermore it is similar to what was done in \cite[Section 7]{P 17} and so we omit the proof. 
\begin{proposition} \label{proposition: reinterpretation of top horizontal map}
There exists a morphism $W: M \rightsquigarrow M'$ of $\Cob_{\theta, 2n}^{n-1}$ such that the maps
$$
\mb{S}_{W}, \; (j, \varepsilon)_{*}: \mathcal{W}^{n-1}_{\theta, 2n}(M) \longrightarrow \mathcal{W}^{n-1}_{\theta, 2n}(M')
$$
are homotopic. 
The cobordism $W$ can be taken to be an elementary cobordism of index $n$.
\end{proposition}

In view of the above proposition, Claim \ref{claim: map between fibres is a weak equivalence} translates to the following theorem.
\begin{theorem} \label{theorem: stability for n-handles}
Let $W: P \rightsquigarrow Q$ be a morphism in $\Cob_{\theta, 2n}^{n-1}$ and suppose that $W$ is an elementary cobordism of index $n$.
Then the map, 
$\mb{S}_{W}: \mathcal{W}^{n-1}_{\theta, 2n}(P) \longrightarrow \mathcal{W}^{n-1}_{\theta, 2n}(Q),$
is a weak homotopy equivalence. 
\end{theorem}
We give the proof of the above theorem in the following subsection. 
With the proof of Theorem \ref{theorem: stability for n-handles} established, the proof of Theorem \ref{theorem: boundary map is quasifibration d = 2n-1} will finally be complete.

\subsection{Proof of Theorem \ref{theorem: stability for n-handles}} \label{subsection: proof of remaining lemma}
The statement of Theorem \ref{theorem: stability for n-handles} is a slight extension of \cite[Theorem 7.7]{P 17}. 
In particular \cite[Theorem 7.7]{P 17} is the same statement as the above theorem but for cobordisms comprised of handles of index strictly higher than $n$. 
The proof of Theorem \ref{theorem: stability for n-handles} is similar to this theorem but with a few important modifications. 
In particular, we will have to work with a stabilized version of the space $\mathcal{W}^{n-1}_{\theta, 2n}(P)$ (see Remark \ref{remark: reason for stabilizing} below for an explanation).
Our first step is the following proposition.
\begin{proposition} \label{proposition: adding an extra component}
Fix $P \in \Ob\Cob_{\theta, 2n}^{n-1}$. 
Let $S \in \Ob\Cob_{\theta, 2n}^{n-1}$ be another object that is diffeomorphic to $S^{2n-1}$, such that $S\cap P = \emptyset$.
Let $V: P\sqcup S \rightsquigarrow P$ be a morphism in $\Cob_{\theta, 2n}^{n-1}$ that is obtained by filling in the sphere $S$ with a disk, i.e.\ $V$ is an elementary cobordism of index $2n$.
Then the map, 
$$
\mb{S}_{V}: \mathcal{W}^{n-1}_{\theta, 2n}(P\sqcup S) \longrightarrow \mathcal{W}^{n-1}_{\theta, 2n}(P),
$$
is a weak homotopy equivalence. 
\end{proposition}
\begin{proof}
Since $V$ is an elementary cobordism of index $2n$, the theorem follows directly from \cite[Theorem 7.7]{P 17}.
\end{proof}

We now fix an object $S \in \Ob\Cob_{\theta, 2n}^{n-1}$ once and for all that is diffeomorphic to a sphere. 
For any $P \in \Ob\Cob_{\theta, 2n}^{n-1}$ with $P\cap S = \emptyset$, and $\mb{t} \in \mathcal{K}_{2n}^{n-1}$, we define,
\begin{equation} \label{equation: space with extra boundary}
\mathcal{W}^{n-1}_{\theta, 2n}(P, \mb{t})_{S} := \mathcal{W}^{n-1}_{\theta, 2n}(P\sqcup S, \mb{t}).
\end{equation}
The subspace $\mathcal{W}^{n-1, c}_{\theta, 2n}(P, \mb{t})_{S} \subset \mathcal{W}^{n-1}_{\theta, 2n}(P, \mb{t})_{S}$ is defined as in Definition \ref{theorem: parametrized surgery}. 
The cobordism $H_{n, n}(S): S \rightsquigarrow S$ induces a map, 
$$
\mb{S}_{H_{n, n}(S)}: \mathcal{W}^{n-1, c}_{\theta, 2n}(P, \mb{t})_{S} \longrightarrow \mathcal{W}^{n-1, c}_{\theta, 2n}(P, \mb{t})_{S}. 
$$
Using this map, 
we define the space, $\mathcal{W}^{n-1, c}_{\theta, 2n}(P, \mb{t})_{S}^{\stb}$, to be the homotopy colimit of the direct system obtained by iterating this map (this construction is analogous to (\ref{equation: stabilization by s-n by s-n}) but with a fixed boundary component given by the object $S$).
As before, the spaces $\mathcal{W}^{n-1}_{\theta, 2n}(P)_{S}, \mathcal{W}^{n-1, c}_{\theta, 2n}(P)_{S}$, and $\mathcal{W}^{n-1, c}_{\theta, 2n}(P)_{S}^{\stb}$ denote the respective homotopy colimits of the above spaces taken over $\mathcal{K}_{2n}^{n-1}$.

By combining Propositions \ref{proposition: adding an extra component}, Theorem \ref{theorem: parametrized surgery}, and Proposition \ref{proposition: stabiliization equivalence}, all three arrows in the diagram below are weak homotopy equivalences,
\begin{equation} \label{equation: auxiliary homotopy equivalences}
\xymatrix{
\mathcal{W}^{n-1}_{\theta, 2n}(P) & \mathcal{W}^{n-1}_{\theta, 2n}(P)_{S} \ar[l]_{\simeq} \ar@{^{(}->}[r]^{\simeq} & \mathcal{W}^{n-1, c}_{\theta, 2n}(P)_{S} \ar@{^{(}->}[r]^{\simeq} & \mathcal{W}^{n-1, c}_{\theta, 2n}(P)^{\stb}_{S}.
}
\end{equation}
Furthermore these weak homotopy equivalences are functorial in the variable $P \in \Ob\Cob_{\theta, 2n}^{n-1}$.
In view of the these homotopy equivalences, to prove Theorem \ref{theorem: stability for n-handles} it will suffice to prove that the map, 
\begin{equation} \label{equation: stable stability}
\mb{S}_{W}: \mathcal{W}^{n-1, c}_{\theta, 2n}(P)^{\stb}_{S} \longrightarrow \mathcal{W}^{n-1, c}_{\theta, 2n}(Q)^{\stb}_{S},
\end{equation}
is a weak homotopy equivalence for any elementary cobordism $W: P \rightsquigarrow Q$ of index $n$.

As with \cite[Theorem 7.7]{P 17}, the key ingredient to proving that (\ref{equation: stable stability}) is a weak homotopy equivalence is to construct a semi-simplicial resolution of the space $\mathcal{W}^{n-1, c}_{\theta, 2n}(P)^{\stb}_{S}$.
The following definition is the same as \cite[Definition 10.1]{P 17} and should also be compared to the semi-simplicial resolution constructed in \cite{GRW 16}.
\begin{defn} \label{defn: semi-simplicial space 1}
Fix $P \in \Ob\Cob_{\theta, 2n}$ and $\mb{t} \in \mathcal{K}^{n-1}_{2n}$, together with the following data:
\begin{itemize} \itemsep.2cm
\item an embedding 
$\chi: S^{n-1}\times(1, \infty)\times\R^{n-1} \longrightarrow P;$
\item a $1$-parameter family $\hat{\ell}^{\std}_{t}$, $t \in (2, \infty)$, of $\theta$-structures on $D^{n}\times D^{n}$ such that 
$$
\hat{\ell}^{\std}_{t}|_{\partial D^{n}\times D^{n}} = \chi^{*}\hat{\ell}_{P}|_{\partial D^{n}\times(t\cdot e_{1} + D^{n})},
$$
where $e_{1} \in (1, \infty)\times\R^{n-1}$ is the basis vector corresponding to the first coordinate. 
\end{itemize}
The space $\mathcal{M}_{\theta, 2n}^{n-1, c}(P, \chi; \mb{t})_{0}$ consists of tuples, 
$
\left((M, (V, \sigma), e), (t, \phi, \hat{L})\right),
$
where:
\begin{itemize} \itemsep.2cm
\item $(M, (V, \sigma), e)$ is an element of $\mathcal{W}^{n-1, c}_{\theta, 2n}(P, \mb{t})_{S}$;
\item $t \in (1, \infty)$;
\item $\phi: (D^{n}\times D^{n}, \; \partial D^{n}\times D^{n}) \longrightarrow (M, P)$ is an embedding;
\item $\hat{L}$ is a path of $\theta$-structures on $D^{n}\times D^{n}$;
\end{itemize} 
subject to the following conditions:
\begin{enumerate} \itemsep.3cm
\item[(i)] the restriction of $\phi$ to $\partial D^{n}\times D^{n}$ satisfies the equation, 
$\phi(y, v) = \chi(\tfrac{y}{|y|}, v + t\cdot e_{1}),$ 
for all $(y, v) \in \partial D^{n}\times D^{n}$.
\item[(ii)] the image $\phi(D^{n}\times D^{n})$ is disjoint from the image of the embedding $e$;
\item[(iii)] the family of $\theta$-structures $\hat{L}$ satisfies: $\hat{L}(0) = \phi^{*}\hat{\ell}_{M}$, $\hat{L}(1) = \hat{\ell}^{\std}_{t}$, and
$\hat{L}(s)|_{\partial D^{n}\times D^{n}}$ is independent of $s \in [0,1]$. 
\end{enumerate}
For $p \in \Z_{\geq 0}$, $\mathcal{M}_{\theta, 2n}^{n-1, c}(P, \chi; \mb{t})_{p}$ is the space of tuples 
$$
\left((M, (V, \sigma), e), (t_{0}, \phi_{0}, \hat{L}_{0}), \dots, (t_{p}, \phi_{p}, \hat{L}_{p})\right),
$$
subject to the following conditions:
\begin{enumerate} \itemsep.2cm
\item[(a)] each $\left((M, (V, \sigma), e), (t_{i}, \phi_{i}, \hat{L}_{i})\right)$ is an element of $\mathcal{M}_{\theta, 2n}^{n-1, c}(P, \chi; \mb{t})_{0}$; 
\item[(b)] $t_{0} < \cdots < t_{p}$; 
\item[(c)] $\phi_{i}(D^{n}\times D^{n})\cap\phi_{j}(D^{n}\times D^{n}) \; = \; \emptyset$ for all $i \neq j$.
\end{enumerate}
 The assignment $[p] \mapsto \mathcal{M}_{\theta, 2n}^{n-1, c}(P, \chi; \mb{t})_{p}$ defines a semi-simplicial space. 
 The $i$th face map is defined by forgetting what is in the $i$-th coordinate.
 The forgetful maps,
 \begin{equation} \label{equation: forgetful augmentation map}
 \begin{aligned}
 \mathcal{M}_{\theta, 2n}^{n-1, c}(P, \chi; \mb{t})_{p} &\longrightarrow \mathcal{W}_{\theta, 2n}^{n-1, c}(P, \mb{t})_{S}, \\
 \left((M, (V, \sigma), e), (t_{0}, \phi_{0}, \hat{L}_{0}), \dots, (t_{p}, \phi_{p}, \hat{L}_{p})\right) &\mapsto (M, (V, \sigma), e),
 \end{aligned}
 \end{equation}
 yield an augmented semi-simplicial space, 
 $\mathcal{M}_{\theta, 2n}^{n-1, c}(P, \chi; \mb{t})_{\bullet} \; \longrightarrow \; \mathcal{M}_{\theta, 2n}^{n-1, c}(P, \chi; \mb{t})_{-1},$ where we have set,
  $\mathcal{M}_{\theta, 2n}^{n-1, c}(P, \chi; \mb{t})_{-1} := \mathcal{W}_{\theta, 2n}^{n-1, c}(P; \mb{t})_{S}$.
\end{defn}
 
 \begin{remark} \label{remark: reason for stabilizing}
Consider the map induced by the augmentation,
$$
 |\mathcal{M}_{\theta, 2n}^{n-1, c}(P, \chi; \mb{t})_{\bullet}| \; \longrightarrow \; \mathcal{M}_{\theta, 2n}^{n-1, c}(P, \chi; \mb{t})_{-1}.
$$
The key step in the proof of \cite[Theorem 7.7]{P 17} was proving that the analogous augmentation map is a weak homotopy equivalence. 
The proof used in \cite{P 17} (namely Proposition 10.4) will not work to prove that this map is a weak homotopy equivalence. 
This is because the embeddings $\phi$ from Definition \ref{defn: semi-simplicial space 1} are embeddings of index-$n$ handles while the dimension of the ambient space is $2n$, and thus 
the cores of such embedded handles may generically have non-trivial intersections. 
Instead of the simple argument used in \cite{P 17}, we will have to apply a stabilization process to $\mathcal{M}_{\theta, 2n}^{n-1, c}(P, \chi; \mb{t})_{\bullet}$, and then use the argument of Galatius and Randal-Williams from \cite[Theorem 4.12]{GRW 16}.
\end{remark}

Concatenation with the cobordism $H_{n, n}(S): S \rightsquigarrow S$ along $S$ induces a semi-simplicial map, 
$$\mb{S}_{H_{n, n}(S)}: \mathcal{M}_{\theta, 2n}^{n-1, c}(P, \chi; \mb{t})_{\bullet} \longrightarrow \mathcal{M}_{\theta, 2n}^{n-1, c}(P, \chi; \mb{t})_{\bullet}.$$
We define $\mathcal{M}_{\theta, 2n}^{n-1, c}(P, \chi; \mb{t})^{\stb}_{\bullet}$ to be the level-wise homotopy colimit of the direct system obtained by iterating this semi-simplicial map. 
The same forgetful maps (\ref{equation: forgetful augmentation map}) yield the augmented semi-simplicial space, 
 $$
 \mathcal{M}_{\theta, 2n}^{n-1, c}(P, \chi; \mb{t})^{\stb}_{\bullet} \longrightarrow \mathcal{M}_{\theta, 2n}^{n-1, c}(P, \chi; \mb{t})^{\stb}_{-1} \; = \; \mathcal{W}_{\theta, 2n}^{n-1, c}(P; \mb{t})^{\stb}_{S}.
$$
We have the following proposition.
\begin{proposition} \label{proposition: stable augmentation map}
Let $P \in \Ob\Cob_{\theta, 2n}^{n-1}$,  $\mb{t} \in \mathcal{K}_{2n}^{n-1}$, and 
$\chi$ be as in Definition \ref{defn: semi-simplicial space 1}. 
The augmentation map, 
$$
 |\mathcal{M}_{\theta, 2n}^{n-1, c}(P, \chi; \mb{t})^{\stb}_{\bullet}| \longrightarrow \mathcal{M}_{\theta, 2n}^{n-1, c}(P, \chi; \mb{t})^{\stb}_{-1} \; = \; \mathcal{W}_{\theta, 2n}^{n-1, c}(P; \mb{t})^{\stb}_{S},
$$
is a weak homotopy equivalence.
\end{proposition}
\begin{proof}
This proof of this proposition is the same as \cite[Theorem 4.12]{GRW 16}.
In the case that $\mb{t} = \emptyset$, then the augmented semi-simplicial space $ \mathcal{M}_{\theta, 2n}^{n-1, c}(P, \chi; \mb{t})^{\stb}_{\bullet} \longrightarrow \mathcal{M}_{\theta, 2n}^{n-1, c}(P, \chi; \mb{t})^{\stb}_{-1}$ is isomorphic to the one considered there. 
In the case of general $\mb{t}$, their proof goes through in the exact same way none-the-less.
\end{proof}

Now, the assignment $\mb{t} \mapsto \mathcal{M}_{\theta, 2n}^{n-1, c}(P, \chi; \mb{t})^{\stb}_{\bullet}$ is a contravariant functor on $\mathcal{K}_{2n}^{n-1}$ and we may take the homotopy colimit. 
As usual we denote this homotopy colimit by $\mathcal{M}_{\theta, 2n}^{n-1, c}(P, \chi)^{\stb}_{\bullet}$.
We obtain a new augmented semi-simplicial space, 
$
\mathcal{M}_{\theta, 2n}^{n-1, c}(P, \chi)^{\stb}_{\bullet} \longrightarrow \mathcal{M}_{\theta, 2n}^{n-1, c}(P, \chi)^{\stb}_{-1}.
$
We have the following corollary.
\begin{corollary} \label{corollary: stabilized augmented space}
The augmentation map 
$$
|\mathcal{M}_{\theta, 2n}^{n-1, c}(P, \chi)^{\stb}_{\bullet}| \longrightarrow \mathcal{M}_{\theta, 2n}^{n-1, c}(P, \chi)^{\stb}_{-1} = \mathcal{W}_{\theta, 2n}^{n-1, c}(P)^{\stb}_{S}
$$
is a weak homotopy equivalence.
\end{corollary}
\begin{proof}
By Proposition \ref{proposition: stable augmentation map} it follows that the map,
$$
\hocolim_{\mb{t} \in \mathcal{K}^{n-1}_{2n}}|\mathcal{M}_{\theta, 2n}^{n-1, c}(P, \chi; \mb{t})^{\stb}_{\bullet}| \longrightarrow \hocolim_{\mb{t} \in \mathcal{K}^{n-1}_{2n}}\mathcal{M}_{\theta, 2n}^{n-1, c}(P, \chi; \mb{t})^{\stb}_{-1}
$$
is a weak homotopy equivalence.
The result then follows by combining this with the homotopy equivalence,
$
\displaystyle{\hocolim_{\mb{t} \in \mathcal{K}^{n-1}_{2n}}}|\mathcal{M}_{\theta, 2n}^{n-1, c}(P, \chi; \mb{t})^{\stb}_{\bullet}| \; \simeq \; |\displaystyle{\hocolim_{\mb{t} \in \mathcal{K}^{n-1}_{2n}}}\mathcal{M}_{\theta, 2n}^{n-1, c}(P, \chi; \mb{t})^{\stb}_{\bullet}|.
$
\end{proof}

We now may use the above semi-simplicial resolution to prove Proposition \ref{theorem: stability for n-handles}.
The argument is essentially the same as the proof of \cite[Theorem 7.7]{P 17}.
We sketch the argument below. 
\begin{proof}[Proof of Theorem \ref{theorem: stability for n-handles}]
We will use the notation for surgeries and traces from \cite[Section 9.1]{P 17}.
For an object $P \in \Ob\Cob_{\theta, 2n}$, we let $\Surg_{n-1}(P)$ denote the set of $(n-1)$-surgeries in the object $P$. 
More precisely, an element $\sigma \in \Surg_{n-1}(P)$ is an embedding $\sigma: D^{n}\times D^{n} \longrightarrow \R^{\infty}$ with $\sigma^{-1}(P) = S^{n-1}\times D^{n}$, equipped with a $\theta$-structure $\hat{\ell}_{\sigma}$ on $D^{n}\times D^{n}$ such that, $\sigma^{*}\hat{\ell}_{P} = \hat{\ell}_{\sigma}|_{S^{n-1}\times D^{n}}$. 
Following \cite[Definition 9.1]{P 17} we may form the \textit{trace} of an element $\sigma \in \Surg_{n-1}(P)$, which is a morphism in $\Cob_{\theta, 2n}^{n-1}$,
$$
T(\sigma): P \rightsquigarrow P(\sigma),
$$
where $P(\sigma) \in \Ob\Cob_{\theta, 2n}^{n-1}$ is the result of the surgery on $\sigma$.
Given $\sigma \in \Surg_{n-1}(P)$ we denote by $\bar{\sigma} \in \Surg_{n-1}(P)$ the surgery \textit{dual} to $\sigma$. 
The underlying manifold of the trace $T(\bar{\sigma}): P(\sigma) \rightsquigarrow P$, is just the reflection of the trace $T(\sigma)$ along $\{1/2\}\times\R^{\infty} \subset [0, 1]\times\R^{\infty}$. 
We let 
$$\textstyle{\Surg^{\tr}_{n-1}}(P) \subset \Surg_{n-1}(P)$$ 
denote the subset consisting of those $\sigma$ for which the embedding $\sigma|_{S^{n-1}\times D^{n}}: S^{n-1}\times D^{n} \longrightarrow P$ factors through a disk in $P$. 
We refer to these surgeries in $\Surg^{\tr}_{n-1}(P)$ as \textit{trivial surgeries}.

Let $\sigma \in \Surg^{\tr}_{n-1}(P)$. 
It is easily checked that the composite cobordism, 
$$
T(\bar{\sigma})\circ T(\sigma): P \rightsquigarrow P,
$$
is isomorphic to the morphism $H_{n, n}(P): P \rightsquigarrow P$ from Construction \ref{Construction: stabilize by s-n times s-n}. 
It follows from Proposition \ref{proposition: stabiliization equivalence} that the induced map 
$$
\mb{S}_{T(\bar{\sigma})\circ T(\sigma)}: \mathcal{W}_{\theta, 2n}^{n-1, c}(P)_{S}^{\stb} \; \longrightarrow \; \mathcal{W}_{\theta, 2n}^{n-1, c}(P)_{S}^{\stb}, 
$$
is a weak homotopy equivalence whenever the surgery $\sigma$ is contained in the subset $\Surg^{\tr}_{n-1}(P)$.
We will need to consider the morphism $T(\bar{\sigma})\circ T(\sigma): P \rightsquigarrow P$ for an arbitrary surgery $\sigma \in \Surg_{n-1}(P)$, that isn't necessary trivial. 
In this case $T(\bar{\sigma})\circ T(\sigma)$ is not necessarily diffeomorphic to $H_{n,n}(P)$. 
Using Corollary \ref{corollary: stabilized augmented space}, we can prove the following lemma. 
\begin{lemma} \label{lemma: resolution of arbitrary surgery}
Let $\sigma \in \Surg_{n-1}(P)$. 
The induced map 
$$
\mb{S}_{T(\bar{\sigma})\circ T(\sigma)}: \mathcal{W}_{\theta, 2n}^{n-1, c}(P)_{S}^{\stb} \; \longrightarrow \; \mathcal{W}_{\theta, 2n}^{n-1, c}(P)_{S}^{\stb}
$$
is a weak homotopy equivalence.
\end{lemma}
The proof of this lemma uses Corollary \ref{corollary: stabilized augmented space} and is proven similarly to \cite[Theorem 9.9]{P 17}. 
We give the proof below.
Assuming Lemma \ref{lemma: resolution of arbitrary surgery} for now let us complete the proof of Theorem \ref{theorem: stability for n-handles}. 
Using the weak homotopy equivalences from (\ref{equation: auxiliary homotopy equivalences}) it will suffice to show that 
$$
\mb{S}_{T(\sigma)}: \mathcal{W}_{\theta, 2n}^{n-1, c}(P)_{S}^{\stb} \; \longrightarrow \; \mathcal{W}_{\theta, 2n}^{n-1, c}(P(\sigma))_{S}^{\stb}
$$
is a weak homotopy equivalence for any $\sigma \in \Surg_{n-1}(P)$.
Let such a $\sigma$ be chosen. 
We may consider the composite, 
$$T(\bar{\bar{\sigma}})\circ T(\bar{\sigma})\circ T(\sigma): P \rightsquigarrow P(\sigma).$$
By Lemma \ref{lemma: resolution of arbitrary surgery} the map 
$$\mb{S}_{T(\bar{\sigma})\circ T(\sigma)} = \mb{S}_{T(\bar{\sigma})}\circ\mb{S}_{T(\sigma)}$$ 
is a weak homotopy equivalence and thus $\mb{S}_{T(\bar{\sigma})}$ induces a a surjection on homotopy groups. 
Similarly, $\mb{S}_{T(\bar{\bar{\sigma}})}\circ\mb{S}_{T(\bar{\sigma})}$ is a weak homotopy equivalence and thus $\mb{S}_{T(\bar{\sigma})}$ induces an injection on homotopy groups. 
Combining this with our first observation it follows that $\mb{S}_{T(\bar{\sigma})}$ is a weak homotopy equivalence. 
Since $\mb{S}_{T(\bar{\sigma})}\circ\mb{S}_{T(\sigma)}$ is a weak homotopy equivalence it follows that $\mb{S}_{T(\sigma)}$ is a weak homotopy equivalence as well. 
This concludes the proof of Theorem \ref{theorem: stability for n-handles}, assuming Lemma \ref{lemma: resolution of arbitrary surgery}. 
\end{proof}

We now prove Lemma \ref{lemma: resolution of arbitrary surgery}. 
The argument is essentially the same as the argument carried out on page 62 in \cite{P 17}.
We provide an outline of the proof. 
Let $\sigma \in \Surg_{n-1}(P)$.
We consider the augmented semi-simplicial space 
$\mathcal{M}_{\theta, 2n}^{n-1, c}(P, \chi)^{\stb}_{\bullet} \longrightarrow \mathcal{M}_{\theta, 2n}^{n-1, c}(P, \chi)^{\stb}_{-1}$
for a very particular choice of embedding $\chi$. 
We choose the embedding $\chi: S^{n-1}\times(1, \infty)\times\R^{n-1} \longrightarrow P$ as in Definition \ref{defn: semi-simplicial space 1}, so that it satisfies the two following further conditions:
\begin{enumerate} \itemsep.2cm
\item[(i)] $\chi(S^{n-1}\times(1, \infty)\times\R^{n-1})\cap \sigma(S^{n-1}\times D^{n}) = \emptyset$; 
\item[(ii)] the restricted embedding $\chi|_{S^{n-1}\times\{2\}\times\{0\}}: S^{n-1} \longrightarrow P$ is isotopic to the embedding $\sigma|_{S^{n-1}\times\{0\}}: S^{n-1} \longrightarrow P$.
\end{enumerate}
With $\chi$ chosen in this way, it follows that $\chi$ determines an embedding in the surgered manifold $P(\sigma)$ as well, and thus it makes sense to form the semi-simplicial space $\mathcal{M}_{\theta, 2n}^{n-1, c}(P(\sigma), \chi)^{\stb}_{\bullet}$.
Using \cite[Construction 10.1]{P 17}, concatenation with the trace $T(\sigma): P \rightsquigarrow P(\sigma)$ induces a semi-simplicial map, 
$$
\mb{S}_{T(\bar{\sigma})\circ T(\sigma), \bullet}: \mathcal{M}_{\theta, 2n}^{n-1, c}(P, \chi)^{\stb}_{\bullet} \longrightarrow \mathcal{M}_{\theta, 2n}^{n-1, c}(P(\sigma), \chi)^{\stb}_{\bullet},
$$
that covers the map from the statement of Lemma \ref{lemma: resolution of arbitrary surgery}. 

For each $p \in \Z_{\geq 0}$, we need to identify the space of $p$-simplices of $\mathcal{M}_{\theta, 2n}^{n-1, c}(P, \chi)^{\stb}_{\bullet}$ with something more familiar. 
We do this by implementing \cite[Construction 10.2]{P 17}.
Choose an embedding,
$\widetilde{\chi}: D^{n}\times(1, \infty)\times\R^{n-1} \longrightarrow \R^{\infty}$
that satisfies:
\begin{itemize} \itemsep.2cm
\item $\widetilde{\chi}^{-1}(P) = \partial D^{n}\times(1, \infty)\times\R^{n-1}$,
\item $\widetilde{\chi}|_{\partial D^{n}\times(1, \infty)\times\R^{n-1}} \; = \; \chi$.
\end{itemize}
For $i \in \Z_{\geq 0}$, we define an embedding, 
$ \chi_{i}: D^{n}\times D^{n} \longrightarrow \R^{\infty},$
by the formula,
$$
\chi_{i}(x, \; y) \; = \; \widetilde{\chi}(x, \; 3(i+1)e_{0} + y).
$$
For each $i$, we define a $\theta$-structure 
$\hat{L}_{i}(\chi)$ on $D^{n}\times D^{n}$ 
by setting, 
$\hat{L}_{i}(\chi) = \chi_{i}^{*}\hat{\ell}^{\std}_{3i+1}.$
Equipped with the $\theta$-structure $\hat{L}_{i}(\chi)$, $\chi_{i}$ can be considered an element of $\Surg_{n-1}(P)$. 
For $p \in \Z_{\geq 0}$, let us denote by $\chi(p)$ the truncated sequence of surgery data $(\chi_{0}, \chi_{1}, \dots, \chi_{p})$.
We let $\bar{\chi}(p)$ denote the sequence of surgery data $(\bar{\chi}_{0}, \bar{\chi}_{1}, \dots, \bar{\chi}_{p})$, where each $\bar{\chi}_{i}$ is the surgery in $P(\chi(p))$ dual to $\chi_{i}$.
Since the embeddings $\bar{\chi}_{i}$ are all disjoint, we may form their simultaneous trace,
$T(\bar{\chi}(p)): P(\chi(p)) \rightsquigarrow P.$
The trace $T(\bar{\chi}(p))$ comes equipped with embeddings  
$$
\psi_{i}: (D^{n}\times D^{n}, S^{n-1}\times D^{n}) \; \longrightarrow \; (T(\bar{\chi}(p)), \; P)
$$
for $i = 0, \dots, p$, where 
$\psi_{i}(D^{n}\times D^{n})\cap\psi_{j}(D^{n}\times D^{n}) = \emptyset$ for $i \neq j$.
These embeddings are the handles attached to $P$ by the surgeries $\chi_{i}$ (see \cite[Construction 9.1, Definition 9.1]{P 17}).
Using this, 
we may define a map
$$
\mathcal{H}(\chi)_{p}: \mathcal{W}^{n-1, c}_{\theta, 2n}(P(\chi(p)))_{S}^{\stb} \; \longrightarrow  \; \mathcal{M}^{n-1, c}_{\theta, 2n}(P)^{\stb}_{p}
$$
by sending 
$x \in \mathcal{W}^{n-1, c}_{\theta, 2n}(P(\chi(p)))_{S}^{\stb}$ 
to the $p$-simplex given by 
$$
\left(\mb{S}_{T(\bar{\chi}(p))}(x); \; (1, \psi_{0}, \hat{L}_{0}(\chi)),\; \dots, \; (3i+1, \psi_{i}, \hat{L}_{i}(\chi)), \; \dots, \; (3p+1, \psi_{p}, \hat{L}_{p}(\chi_{p}))\right),
$$
where each $\hat{L}_{i}(\chi)$ is considered to be the constant one-parameter family of $\theta$-structures as defined above.
By the same argument given in \cite[Lemma 10.6]{P 17} it follows that this map $\mathcal{H}(\chi)_{p}$ is a weak homotopy equivalence for all $p \in \Z_{\geq 0}$.
With the above constructions in place, we can now give the proof of Lemma \ref{lemma: resolution of arbitrary surgery}. 
\begin{proof}[Proof of Lemma \ref{lemma: resolution of arbitrary surgery}]
Let $\sigma \in \Surg_{n-1}(P)$ and $\chi$ be chosen as above so that conditions (i) and (ii) are satisfied.
By condition (i), it follows that for each $p \in \Z_{\geq 0}$, $\sigma(S^{n-1}\times D^{n})$ is contained in $P(\chi(p))$, and thus $\sigma$ determines surgery data in $P(\chi(p))$.
For each $p \in \Z_{\geq 0}$,
we denote this induced surgery data by $\sigma_{p} \in \Surg_{n-1}(P(\chi(p)))$.
By condition (ii) it follows that this induced surgery data is contained in the subset $\Surg^{\tr}_{n-1}(P(\chi(p)))$. 
It follows that the composite cobordism $T(\bar{\sigma}_{p})\circ T(\sigma_{p}): P(\chi(p)) \rightsquigarrow P(\chi(p))$ is diffeomorphic to the cobordism $H_{n, n}(P(\chi(p)))$, and thus the self-map it induces on $\mathcal{W}_{\theta, 2n}^{n-1, c}(P(\chi(p)))^{\stb}_{S}$ is a weak homotopy equivalence.
By the commutative diagram, 
$$
\xymatrix{
\mathcal{W}_{\theta, 2n}^{n-1, c}(P(\chi(p)))^{\stb}_{S} \ar[d]^{\simeq}_{\mathcal{H}(\chi)_{p}} \ar[rrr]^{\mb{S}_{T(\bar{\sigma}_{p})\circ T(\sigma_{p})}}_{\simeq} &&& \mathcal{W}_{\theta, 2n}^{n-1, c}(P(\chi(p)))^{\stb}_{S} \ar[d]^{\simeq}_{\mathcal{H}(\chi)_{p}} \\
\mathcal{M}_{\theta, 2n}^{n-1, c}(P, \chi)^{\stb}_{p} \ar[rrr]^{\mb{S}_{T(\bar{\sigma})\circ T(\sigma), p}} &&& \mathcal{M}_{\theta, 2n}^{n-1, c}(P, \chi)^{\stb}_{p} 
}
$$
it follows that the bottom-horizontal map is a weak homotopy equivalence for all $p$, and thus it follows that the semi-simplicial map $\mb{S}_{T(\bar{\sigma})\circ T(\sigma), \bullet}$ is a level-wise weak homotopy equivalence.
The lemma then follows from the commutative diagram,
$$
\xymatrix{
|\mathcal{M}_{\theta, 2n}^{n-1, c}(P, \chi)^{\stb}_{\bullet}| \ar[rr]^{\simeq} \ar[d]_{\simeq} && |\mathcal{M}_{\theta, 2n}^{n-1, c}(P(\sigma), \chi)^{\stb}_{\bullet}| \ar[d]_{\simeq} \\
\mathcal{W}_{\theta, 2n}^{n-1, c}(P)^{\stb}_{S} \ar[rr]^{\mb{S}_{T(\bar{\sigma})\circ T(\sigma)}} && \mathcal{W}_{\theta, 2n}^{n-1, c}(P(\sigma))^{\stb}_{S}
}
$$
together with Corollary \ref{corollary: stabilized augmented space} which implies that the vertical maps are weak homotopy equivalences.
This completes the proof of Lemma \ref{lemma: resolution of arbitrary surgery} and thus finishes the proof of Theorem \ref{theorem: stability for n-handles}. 
\end{proof}

\section{The Fibre Transport} \label{section: a fibre transport map}
We now show how to use the results of the previous section, namely Theorem \ref{theorem: boundary map is quasifibration d = 2n-1}, to prove part (b) of Theorem \ref{theorem: factorization of index difference} (which is the odd-dimensional case). 
\subsection{The fibre transport} \label{subsection: the fibre transport boundary sequence} 
Fix a positive integer $n$.
Let $d = 2n-1$ and $k = n-1$.
These are the dimensional conditions assumed in the statement of Theorem \ref{theorem: boundary map is quasifibration d = 2n-1}. 
Fix an element $P \in \mathcal{M}_{\theta, 2n-1}$. 
By Theorem \ref{theorem: boundary map is quasifibration d = 2n-1}, we have a homotopy fibre sequence, 
$$\xymatrix{
\mathcal{W}^{n-1}_{\theta, 2n}(P) \ar[r] & \mathcal{W}^{\partial, n-1}_{\theta, 2n} \ar[r] & \mathcal{W}^{n-1}_{\theta, 2n-1}.
}$$
Choosing a basepoint in the fibre $\mathcal{W}^{n-1}_{\theta, 2n}(P)$ yields a fibre-transport map, 
\begin{equation} \label{equation: fibre transfer boundary}
\iota_{2n}: \Omega\mathcal{W}^{n-1}_{\theta, 2n-1} \longrightarrow \mathcal{W}^{n-1}_{\theta, 2n}(P).
\end{equation}
Now, let us choose an element $M \in \mathcal{M}^{\partial}_{\theta, 2n} \subset \mathcal{W}^{\partial, n-1}_{\theta, 2n}$. 
We denote $P := \partial M$.
Consider the restriction map,
\begin{equation} \label{equation: boundary face restriction map}
r: \mathcal{R}^{+}(M) \; \longrightarrow \; \mathcal{R}^{+}(P), \quad
g \mapsto g|_{P}.
\end{equation}
By Proposition \ref{proposition: chernysh restriction theorem}, this map is a quasi-fibration. 
The fibre over $g_{P} \in \mathcal{R}^{+}(P)$ is given by the space of metrics $\mathcal{R}^{+}(M)_{g_{P}}$.
We obtain the fibre sequence 
$$
\xymatrix{
\mathcal{R}^{+}(M)_{g_{P}} \ar[r] & \mathcal{R}^{+}(M) \ar[r]^{r} & \mathcal{R}^{+}(P),
}
$$
and hence a fibre transport map 
\begin{equation} \label{equation: fibre transfer space of metrics}
T: \Omega\mathcal{R}^{+}(P) \longrightarrow \mathcal{R}^{+}(M)_{g_{P}}.
\end{equation}
The following two theorems are the main ingredients in the proof of Theorem \ref{theorem: factorization of index difference}, part (b). 
\begin{theorem} \label{theorem: commutativity of restrictions and transfer}
The following diagram is homotopy commutative
$$
\xymatrix{
\Omega^{2}\mathcal{W}^{n-1}_{\theta, 2n-1} \ar[d]^{\Omega j_{2n-1}} \ar[rr]^{\Omega\iota_{2n}} && \Omega\mathcal{W}^{n-1}_{\theta, 2n}(P) \ar[d]^{j_{2n}} \\
\Omega\mathcal{R}^{+}(P) \ar[rr]^{T} && \mathcal{R}^{+}(M)_{g_{P}},
}
$$
where the vertical maps are the ones constructed in Section \ref{subsection: the fibre transport}, using the fibre-transport. 
\end{theorem}

\begin{theorem} \label{theorem: compatibility of K-theory classes}
The following diagram is homotopy commutative,
$$\xymatrix{
\Omega^{2}\mathcal{W}^{n-1}_{\theta, 2n-1} \ar[dr]_{\Omega^{2}\bar{\mathcal{A}}_{2n-1}} \ar[rr]^{\Omega\iota_{2n}} && \Omega\mathcal{W}^{n-1}_{\theta, 2n}(P) \ar[dl]^{\Omega\bar{\mathcal{A}}_{2n}} \\
& \Omega^{\infty+2n+1}\KO. &
}$$
\end{theorem}

\subsection{Proof of Theorem \ref{theorem: factorization of index difference}, part (b)} \label{subsection: main theorem part b}
We now show how to prove Theorem \ref{theorem: factorization of index difference}, part (b), assuming that Theorems \ref{theorem: commutativity of restrictions and transfer} and \ref{theorem: compatibility of K-theory classes} have been established.
Let $(P, g_{P}) \in \mathcal{M}^{\psc}_{\theta, 2n-1}$ and $M \in \mathcal{M}_{\theta, 2n}(P)$ be chosen as above. 
Consider the fibre transport map 
$T: \Omega\mathcal{R}^{+}(P) \longrightarrow \mathcal{R}^{+}(M)_{g_{P}} $
from (\ref{equation: fibre transfer space of metrics}). 
We will need to use the lemma below which follows from \cite[Theorem 3.6.1]{BERW 16}.
\begin{lemma} \label{lemma: commutativity fibre transport index difference}
The following diagram is homotopy commutative,
$$
\xymatrix{
\Omega\mathcal{R}^{+}(P) \ar[dr]_{-\Omega\inddiff} \ar[rr]^{T} && \mathcal{R}^{+}(M)_{g_{P}} \ar[dl]^{\inddiff} \\
& \Omega^{\infty + 2n + 1}\KO. & 
}
$$
\end{lemma}

\begin{proof}[Proof of Theorem \ref{theorem: factorization of index difference}, part (b), assuming Theorems \ref{theorem: commutativity of restrictions and transfer} and \ref{theorem: compatibility of K-theory classes}]
Consider the diagram 
$$
\xymatrix{
\Omega^{2}\mathcal{W}^{n-1}_{\theta, 2n-1} \ar[d]^{\Omega j_{2n-1}} \ar[rr]^{\Omega\iota_{2n}} && \Omega\mathcal{W}^{n-1}_{\theta, 2n}(P) \ar[d]^{j_{2n}} \\
\Omega\mathcal{R}^{+}(P) \ar[dr]_{-\Omega\inddiff} \ar[rr]^{T} && \mathcal{R}^{+}(M)_{g_{P}} \ar[dl]^{\inddiff} \\
& \Omega^{\infty + 2n + 1}\KO. &
}
$$
By Theorem \ref{theorem: commutativity of restrictions and transfer} the upper square is homotopy commutative, and by Lemma \ref{lemma: commutativity fibre transport index difference} the bottom triangle is homotopy commutative; by combining these two results the whole diagram is commutative. 
By Theorem \ref{theorem: factorization of index difference}, part (a), the composite $\inddiff\circ j_{2n}$ agrees with the map $\Omega\bar{\mathcal{A}}_{2n}$.
By Theorem \ref{theorem: compatibility of K-theory classes}, we have, 
$$\Omega\bar{\mathcal{A}}_{2n}\circ\Omega\iota_{2n} = \Omega(\bar{\mathcal{A}}_{2n}\circ\iota_{2n}) \; = \;  \Omega^{2}\bar{\mathcal{A}}_{2n-1},$$
and thus,
$$
\inddiff\circ j_{2n}\circ\Omega\iota_{2n} \; = \; \Omega^{2}\bar{\mathcal{A}}_{2n-1}.
$$
By commutativity of the above diagram it follows that, 
$$
-\Omega\inddiff\circ\Omega j_{2n-1} \; = \; \Omega^{2}\bar{\mathcal{A}}_{2n-1}.
$$
This concludes the proof of Theorem \ref{theorem: factorization of index difference}, part (b).
\end{proof}

\subsection{Proof of Theorem \ref{theorem: compatibility of K-theory classes}}  \label{subsection: proof of compatibility of KO orientations}
We now give the proof of Theorem \ref{theorem: compatibility of K-theory classes}.
Let $P \in \mathcal{M}_{\theta, d-1}$ and $k \in \Z_{\geq -1}$.
Consider the commutative diagram,
\begin{equation} \label{equation: big commutative diagram}
\xymatrix{
\mathcal{W}^{k}_{\theta, d}(P) \ar[d] & \mathcal{L}^{k}_{\theta, d}(P) \ar[l]_{\simeq} \ar[r]^{\simeq \ \ \ \ } & \mathcal{D}^{\mf, k}_{\theta, d+1}(P) \ar[d] \ar@{^{(}->}[r] & \mathcal{D}_{\theta, d+1}(P)  \ar[d]   \\
\mathcal{W}^{\partial, k}_{\theta, d} \ar[d] & \mathcal{L}^{\partial, k}_{\theta, d} \ar[l]_{\simeq} \ar[r]^{\simeq} & \mathcal{D}^{\partial, \mf, k}_{\theta, d+1} \ar[d] \ar@{^{(}->}[r] & \mathcal{D}^{\partial}_{\theta, d+1}  \ar[d] \\
\mathcal{W}^{k}_{\theta, d-1} & \mathcal{L}^{k}_{\theta, d-1} \ar[l]_{\simeq} \ar[r]^{\simeq} & \mathcal{D}^{\mf, k}_{\theta, d} \ar@{^{(}->}[r] & \mathcal{D}_{\theta, d}.
}
\end{equation}
Commutativity of the bottom-left square follows from Theorem \ref{theorem: equivalence with long manifolds with boundary}. 
Commutativity of the upper-left square follows from the construction carried out in Appendix \ref{section: long manifolds to W partial}, which is where we prove that the first two middle-horizontal arrows in the above diagram are weak homotopy equivalences. 
The right-most horizontal arrows are the inclusions, which in general are not weak homotopy equivalences. 
The right-most column is a homotopy fibre sequence for all integers $d$ and manifolds $P$, this follows from the results of \cite{G 12}.
By Theorem \ref{theorem: boundary map is quasifibration d = 2n-1}, the left-most column is a homotopy fibre sequence in the case that $d = 2n$ and $k = n-1$.
Specializing to this particular case ($d = 2n$ and $k = n-1$) the fibre-transports for the column fibre-sequences yield the commutative square,
\begin{equation} \label{equation: commutative square of fibre transports}
\xymatrix{
\Omega\mathcal{W}^{k}_{\theta, d-1} \ar[d] \ar[r] & \Omega\mathcal{D}_{\theta, d} \ar[d] \\
\mathcal{W}^{k}_{\theta, d}(P) \ar[r] & \mathcal{D}_{\theta, d+1}(P).
}
\end{equation}
Composing this commutative square with the following commutative square (which arrises as a result of Proposition \ref{proposition: boundary map long manifold fibre sequence}), 
$$
\xymatrix{
\Omega\mathcal{D}_{\theta, d} \ar[d] \ar[r]^{\simeq} & \Omega^{\infty}\MT\theta_{d} \ar[d]^{\iota_{d}} \\
\mathcal{D}_{\theta, d+1}(P) \ar[r]^{\simeq} & \Omega^{\infty-1}\MT\theta_{d+1},
}
$$
yields the commutative square,
\begin{equation} \label{equation: target commutative diagram}
\xymatrix{
\Omega\mathcal{W}^{k}_{\theta, d-1} \ar[d] \ar[r] & \Omega^{\infty}\MT\theta_{d} \ar[d]^{\iota_{d}} \\
\mathcal{W}^{k}_{\theta, d}(P) \ar[r] & \Omega^{\infty-1}\MT\theta_{d+1}.
}
\end{equation}
Using this commutative square, the proof of Theorem \ref{theorem: compatibility of K-theory classes} then follows from the equation, 
$$\Omega\mathcal{A}_{2n-1} = \iota_{2n}^{*}\mathcal{A}_{2n},$$ 
together with the fact that $\Omega\bar{\mathcal{A}}_{2n-1}$ and $\bar{\mathcal{A}}_{2n}$ are both defined as the pullbacks of $\Omega\mathcal{A}_{2n-1}$ and $\mathcal{A}_{2n}$ via the horizontal maps of (\ref{equation: target commutative diagram}) respectively. 
This concludes the proof of Theorem \ref{theorem: compatibility of K-theory classes}. 
 
\section{Positive scalar curvature metrics on manifolds with boundary} \label{subsection: positive scalar curvature metrics on manifolds with boundary}
In this section we carry out a construction similar to what was done in Section \ref{section: Spaces of Manifold Equipped with PSC Metrics and Surgery Data}, except now we use the space $\mathcal{W}^{\partial}_{\theta, d}$.
This is the main ingredient that will be used in our proof of Theorem \ref{theorem: commutativity of restrictions and transfer}. 
We give the proof of this theorem in Subsection \ref{subsection: proof of boundary map commutativity}.
\subsection{The parametrized Gromov-Lawson construction on manifolds with boundary} \label{subsection: gromov lawson on mfds with bd}
We will need to review some constructions from \cite{Wa 16}.
\begin{Construction}[Metrics on the collared half-disk] \label{Construction: metrics on the collared half-disk}
We will need to recall a special $\psc$ metric constructed in \cite[Section 5.1]{Wa 16}, used when preforming surgery on manifolds with boundary.
We need to work with a collared version of the half-disk,
$$
D^{d}_{+} \; = \; \{(x_{1}, \dots, x_{d}) \in \R^{d} \; | \; x_{1} \geq 0, \; \; x_{1}^{2} + \cdots + x_{d}^{2} \leq 1 \; \}.
$$
Fix a smooth function $\rho: [0, 1] \longrightarrow [0, 1]$ with $\rho(t) = 0$ for $t < 1/8$ and $\rho(t) = t$ for $t > 1/4$.
Consider the $d$-dimensional collared half-disk,
$$
D^{d}_{+, c} \; = \; \{ (x_{1}, \dots, x_{d}) \in \R^{d} \; | \; x_{1} \geq 0, \; \; \rho(x_{1})^{2} + x_{2}^{2} + \cdots + x_{d}^{2} \leq 1 \; \}.
$$
This is of course diffeomorphic to the manifold with corners $D^{c}_{+}(V^{+})$ defined in Construction \ref{Construction: half disk with collars}. 
The boundary decomposes as, 
$$\partial D^{d}_{+, c} = \partial_{0}D^{d}_{+, c}\cup\partial_{1}D^{d}_{+, c},$$ 
and we have,
$$
\partial D^{d}_{+, c}\cap([0, 1/8)\times\R^{d-1}) \; = \; \partial_{0}D^{d}_{+, c}\times[0, 1/8).
$$
We let $\hat{g}^{d}_{\torp}$ denote the $\psc$-metric on $D^{d}_{+, c}$ defined in \cite[Section 5.1]{Wa 16}. 
Identifying $\partial_{0}D^{d}_{+, c}$ and $\partial_{1}D^{d}_{+, c}$ with the $(d-1)$-dimensional disk $D^{d-1}$, 
this metric $\hat{g}^{d}_{\torp}$ has the following properties:
\begin{enumerate} \itemsep.2cm
\item[(i)] $\hat{g}^{d}_{\torp}|_{\partial_{0}D^{d}_{+, c}\times[0, 1/8)} \; = \; g^{d-1}_{\torp} + dt^{2}$;
\item[(ii)] $\hat{g}^{d}_{\torp}|_{\partial_{1}D^{d}_{+, c}} \; = \; g^{d-1}_{\torp}$.
\end{enumerate}
In the above conditions, $g^{d-1}_{\torp}$ is the \textit{torpedo metric} on $D^{d-1}$ used in Section \ref{defn: preliminary definitions psc} (see \cite{Wa 13} for the definition).
\end{Construction}

We now use the metric from the above construction to define a space of $\psc$-metrics on a manifold with boundary that are standard near a chosen embedding of $S^{k-1}\times D^{d-k+1}_{+, c}$.
\begin{defn} \label{defn: metrics compatible with half disk}
Let $W$ be a compact manifold of dimension $d$. 
Fix an embedding 
$$\phi: S^{k-1}\times D^{d-k+1}_{+, c} \longrightarrow W$$ 
that satisfies, 
$$\phi^{-1}(\partial W) = S^{k-1}\times\partial_{0}D^{d-k+1}_{+, c}.$$
The subspace $\mathcal{R}^{+}(W; \phi) \subset \mathcal{R}^{+}(W)$ consists of those $\psc$ metrics $g$ that satisfy, 
$$\phi^{*}g \; = \; g_{S^{k-1}} + \hat{g}^{d-k+1}_{\tor},$$ 
where $\hat{g}^{d-k+1}_{\tor}$ is the metric from Construction \ref{Construction: metrics on the collared half-disk}. 
\end{defn}
Let $\phi: S^{k-1}\times D^{d-k+1}_{+, c} \longrightarrow W$ be an embedding with $\phi^{-1}(\partial W) = S^{k-1}\times\partial_{0}D^{d-k+1}_{+, c}$ as in the above definition.
Consider the surgered manifold, 
$$
\widetilde{W} \; = \; W\setminus\phi(S^{k-1}\times D^{d-k+1}_{+, c})\bigcup D^{k}\times\partial_{1}D^{d-k+1}_{+, c}.
$$
There is a map 
\begin{equation} \label{equation: gromov lawson surgery map}
\mb{S}_{\phi}: \mathcal{R}^{+}(W; \phi) \longrightarrow \mathcal{R}^{+}(\widetilde{W})
\end{equation}
defined by sending a metric $g \in \mathcal{R}^{+}(W; \phi)$ to the metric on $\widetilde{W}$ that is equal to $g$ on the complement $W\setminus\phi(S^{k-1}\times D^{d-k+1}_{+, c})$, and equal to $g_{\tor}^{k} + g_{S^{d-k}}|_{\partial_{1}D^{d-k+1}_{+, c}}$ on $D^{k}\times \partial_{1}D^{d-k+1}_{+, c}$, where $g_{S^{d-k}}$ is the standard round metric on $S^{d-k}$.
The theorem stated below can be viewed as a version of Theorem \ref{theorem: walsh chernysh} for manifolds with boundary. 
\begin{theorem} \label{theorem: gromov lawson on boundary}
Suppose that $k, d-k \geq 3$.
Let $\phi: S^{k-1}\times D^{d-k+1}_{+, c} \longrightarrow W$ be an embedding as above ($\dim(W) = d$).
Then the map $\mb{S}_{\phi}: \mathcal{R}^{+}(W; \phi) \longrightarrow \mathcal{R}^{+}(\widetilde{W})$ from (\ref{equation: gromov lawson surgery map}) is a weak homotopy equivalence.  
\end{theorem}
Theorem \ref{theorem: gromov lawson on boundary} is proven by assembling several results and techniques developed by Walsh in \cite{Wa 16}.
The proof requires a few steps. 
Let $W$ and $\phi: S^{k-1}\times D^{d-k+1}_{+, c} \longrightarrow W$ be as in the statement of Theorem \ref{theorem: gromov lawson on boundary}.
We have,
$$
\widetilde{W} \; = \; W\setminus\phi(S^{k-1}\times D^{d-k+1}_{+, c})\bigcup D^{k}\times\partial_{1}D^{d-k+1}_{+, c},
$$
and we let
$$
\bar{\phi}: D^{k}\times\partial_{1}D^{d-k+1}_{+, c} \longrightarrow \widetilde{W}
$$
be the embedding given by the inclusion of the right-hand factor in the above union.
The map $\mb{S}_{\phi}: \mathcal{R}^{+}(W; \phi) \longrightarrow \mathcal{R}^{+}(\widetilde{W})$ has its image in the space $\mathcal{R}^{+}(\widetilde{W}; \bar{\phi})$, and therefore it factors as a composite,
\begin{equation} \label{equation: factorization of S(phi)}
\xymatrix{
\mathcal{R}^{+}(W; \phi) \ar[r] & \mathcal{R}^{+}(\widetilde{W}; \bar{\phi}) \ar@{^{(}->}[r] & \mathcal{R}^{+}(\widetilde{W}).
}
\end{equation}
To prove Theorem \ref{theorem: gromov lawson on boundary} it will suffice to show that each of these maps is a weak homotopy equivalence. 
The first map is easy.
\begin{proposition} \label{proposition: the first map}
The map $\mathcal{R}^{+}(W; \phi) \longrightarrow \mathcal{R}^{+}(\widetilde{W}; \bar{\phi})$ from (\ref{equation: factorization of S(phi)}) is a homeomorphism.
\end{proposition}
\begin{proof}
We define an inverse, 
$$
\mathcal{R}^{+}(\widetilde{W}; \bar{\phi}) \longrightarrow \mathcal{R}^{+}(W; \phi), 
$$
by sending a metric $g \in \mathcal{R}^{+}(\widetilde{W}; \bar{\phi})$, to the metric $\bar{g}$ on 
$$
W \; = \; \widetilde{W}\setminus\bar{\phi}(D^{k}\times\partial_{1}D^{d-k+1}_{+, c})\bigcup S^{k-1}\times D^{d-k+1}_{+, c},
$$
defined to be equal to $g$ on $\widetilde{W}\setminus\bar{\phi}(D^{k}\times\partial_{1}D^{d-k+1}_{+, c})$, and equal to $g_{S^{k-1}} + \hat{g}_{\tor}^{d-k+1}$ on the factor $S^{k-1}\times D^{d-k+1}_{+, c}$.
This map is clearly inverse to the first map in (\ref{equation: factorization of S(phi)}) thus finishes the proof of the proposition. 
\end{proof}

We now proceed to work on the second map in (\ref{equation: factorization of S(phi)}). 
We let $M$ denote the boundary, $\partial W$. 
Let $\phi_{0}: S^{k-1}\times\partial_{0}D^{d-k+1}_{+, c} \longrightarrow M$ denote the restriction of $\phi$ to $S^{k-1}\times\partial_{0}D^{d-k+1}_{+, c}$.
Similarly, we let $\widetilde{M}$ denote the boundary of $\widetilde{W}$. 
We let 
$$\bar{\phi}_{1}: D^{k}\times\partial_{0, 1}D^{d-k+1}_{+, c} \; \longrightarrow \; \widetilde{M}$$
be the embedding obtained by restricting $\bar{\phi}$. 
The subspace, 
$\mathcal{R}^{+}(\widetilde{M}; \bar{\phi}_{1}) \subset \mathcal{R}^{+}(\widetilde{M}),$ 
is defined in the same way as in Section \ref{defn: preliminary definitions psc}. 
By repeating the exact same argument used in the proof of the main theorem from \cite{Ch 06} (restated by us as Proposition \ref{proposition: chernysh restriction theorem}), it follows that the restriction map,
$$
\mathcal{R}^{+}(\widetilde{W}; \bar{\phi}) \longrightarrow \mathcal{R}^{+}(\widetilde{M}; \bar{\phi}_{1}), \quad g \mapsto g|_{\widetilde{M}},
$$
is a quasi-fibration (see also \cite[Lemma 8.2.2]{Wa 16}). 
Using this fact, we are now in a position to prove that the second map in the composite (\ref{equation: factorization of S(phi)}) is a weak homotopy equivalence. 
\begin{proposition} \label{proposition: inclusion of phi-bar metric space}
The inclusion, 
$
\mathcal{R}^{+}(\widetilde{W}; \bar{\phi}) \hookrightarrow \mathcal{R}^{+}(\widetilde{W}),
$
is a weak homotopy equivalence. 
\end{proposition}
\begin{proof}
Fix a metric $g \in \mathcal{R}^{+}(M; \bar{\phi}_{1})$. 
We have a map of fibre sequences, 
$$
\xymatrix{
\mathcal{R}^{+}(\widetilde{W}; \bar{\phi})_{g} \ar[d] \ar[r] & \mathcal{R}^{+}(\widetilde{W})_{g} \ar[d] \\
\mathcal{R}^{+}(\widetilde{W}; \bar{\phi}) \ar[r] \ar[d] & \mathcal{R}^{+}(\widetilde{W}) \ar[d] \\
\mathcal{R}^{+}(\widetilde{M}; \bar{\phi}_{1}) \ar[r] & \mathcal{R}^{+}(\widetilde{M}),
}
$$
where the horizontal maps are inclusions. 
The left column is a fibre sequence by Proposition \ref{proposition: chernysh restriction theorem} and the left-column is a fibre sequence by our observation above (using \cite[Lemma 8.2.2]{Wa 16}).
The bottom-horizontal map is a weak homotopy equivalence by Theorem \ref{theorem: walsh chernysh}. 
The top-horizontal map is a weak homotopy equivalence by what is proven in \cite[Section 7.6]{Wa 16}.
It follows from the five-lemma that the middle-horizontal map is a weak homotopy equivalence. 
This completes the proof of the proposition. 
\end{proof}

With the above proposition established, the proof of Theorem \ref{theorem: gromov lawson on boundary} is complete.

\subsection{The functor $\mathcal{W}^{\partial, \psc}_{\theta, d}(\--)$} \label{subsection: the functor W psc}
We now use the construction from the previous subsection to define a functor analogous to the one defined in Section \ref{subsection: W psc functor}. 
\begin{defn} \label{defn: standard metrics with boundary}
Let $\mb{t} \in \mathcal{K}_{d}$ and $(V, \sigma) \in (G^{\partial, \mf}_{\theta, d+1}(\R^{\infty})_{\loc})^{\mb{t}}$.
We let $g^{\std, -}_{\mb{t}, \partial}$ denote the metric on the manifold $S(V^{-})\times_{\mb{t}}D^{c}_{+}(V^{+})$ with the following property:
\begin{itemize} \itemsep.2cm
\item for each $t \in \mb{t}$ the restriction of $g^{\std, -}_{\mb{t}, \partial}$ to 
$S(V^{-}(i))\times D^{c}_{+}(V^{+}(t)) = S^{\delta(t)-1}\times D_{+, c}^{d-\delta(t) + 1}$ 
agrees with the metric, 
$$g_{S^{\delta(t)-1}} + \hat{g}^{d-\delta(t)+1}_{\tor}.$$
\end{itemize}
Similarly, we let $g^{\std, +}_{\mb{t}, \partial}$ denote the metric on $D(V^{-}(t))\times_{\mb{t}}\partial_{1}D_{+}^{c}(V^{+}(t))$
with the following property:
\begin{itemize} \itemsep.2cm
\item for each $t \in \mb{t}$ the restriction of $g^{\std, +}_{\mb{t}, \partial}$ to $D(V^{-}(t))\times \partial_{1}D_{+}^{c}(V^{+}(t))$ agrees with the metric, 
$$g^{\delta(t)}_{\tor} + g_{S^{d-\delta(t)}}|_{\partial_{1}D_{+}^{c}(V^{+}(t))}.$$ 
\end{itemize}
\end{defn}

The following definition is the main definition of this section. 
\begin{defn} \label{defn: metrics standard on multiple surgeries 2}
Fix $k \in \Z_{\geq -1}$.
Let $\mb{t} = (\mb{t}_{0}, \mb{t}_{1}) \in \widehat{\mathcal{K}}_{d}^{k}$. 
Let $(M, (V, \sigma), e) \in \mathcal{W}^{\partial, k}_{\theta, d}(\mb{t})$ and let 
$$\begin{aligned}
e^{-}_{0}: S(V^{-})\times_{\mb{t}_{0}}D^{c}_{+}(V^{+}) \longrightarrow M, \\
e^{-}_{1}: S(V^{-})\times_{\mb{t}_{1}}D(V^{+}) \longrightarrow M,
\end{aligned}$$
denote the embeddings obtained by restricting $e_{0}$ and $e_{1}$ respectively.
The subspace 
$$
\mathcal{R}^{+}(M; \mb{t}, e) \subset \mathcal{R}^{+}(M)$$
is defined to consist of all metrics $g \in \mathcal{R}^{+}(M)$ such that 
$$
(e^{-}_{0})^{*}g \; = \; g^{\std, -}_{\mb{t}_{0}, \partial} \quad \text{and} \quad 
(e^{-}_{1})^{*}g \; = \; g^{\std, -}_{\mb{t}_{1}}
$$
respectively.
We define the space 
$\mathcal{W}^{\psc, \partial, k}_{\theta, d}(\mb{t})$
to consist of all tuples 
$((M, (V, \sigma), e), g)$ 
such that 
$(M, (V, \sigma), e) \in \mathcal{W}^{k}_{\theta, d}(\mb{t})$ and $g \in \mathcal{R}^{+}(M; \; \mb{t}, e).$ 
\end{defn}

The correspondence $\mb{t} \mapsto \mathcal{W}^{\psc, \partial, k}_{\theta, d}(\mb{t})$ is made into a contravariant functor on $\widehat{\mathcal{K}}_{d}^{k}$ in the same way as was done in Section \ref{subsection: W psc functor}. 
To avoid redundancy we don't repeat the construction here. 
As before, we may form the homotopy colimit over the category $\widehat{\mathcal{K}}_{d}^{k}$. 
We denote,
$$
\mathcal{W}^{\psc, \partial, k}_{\theta, d} := \hocolim_{\mb{t} \in \widehat{\mathcal{K}}_{d}^{k}}\mathcal{W}^{\psc, \partial, k}_{\theta, d}(\mb{t}).
$$
The following theorem is analogous to Theorem \ref{theorem: fibresequence forget the metric}. 
\begin{theorem} \label{theorem: forgetful map fibration}
Let $k, d-k \geq 3$. 
The forgetful map 
$\mathcal{W}^{\partial, \psc, k-1}_{\theta, d} \longrightarrow \mathcal{W}^{\partial, k-1}_{\theta, d}$
is a quasi-fibration. 
Let $M \in \mathcal{M}^{\partial}_{\theta, d}$. 
The fibre over $(M, (\emptyset, \emptyset), \emptyset) \in \mathcal{W}^{\partial, k-1}_{\theta, d}$ is given by the space 
$\mathcal{R}^{+}(M)$.
\end{theorem}
\begin{proof}
As was done in the proof of Theorem \ref{theorem: fibresequence forget the metric}, we apply Proposition \ref{proposition: homotopy colimit fibre sequence}. 
Let $(j, \varepsilon): \mb{s} \longrightarrow \mb{t}$ be a morphism of $\widehat{\mathcal{K}}_{d}^{k}$. 
Let $(M, (V, \sigma), e) \in \mathcal{W}^{\partial}_{\theta, d}(\mb{t})$, and let $(M', (V,' \sigma'), e') \in \mathcal{W}^{\partial}_{\theta, d}(\mb{s})$ be the image of $(M, (V, \sigma), e)$ under the map,
$(j, \varepsilon)^{*}: \mathcal{W}^{\partial, k}_{\theta, d}(\mb{t}) \longrightarrow \mathcal{W}^{\partial, k}_{\theta, d}(\mb{s}).$
The map $(j, \varepsilon)^{*}$ yields a map of fibre sequences,
$$
\xymatrix{
\mathcal{R}^{+}(M; \mb{t}, e) \ar[d] \ar[rr] && \mathcal{R}^{+}(M'; \mb{s}, e') \ar[d] \\
\mathcal{W}^{\partial, \psc, k-1}_{\theta, d}(\mb{t}) \ar[d] \ar[rr] && \mathcal{W}^{\partial, \psc, k-1}_{\theta, d}(\mb{s}) \ar[d] \\
\mathcal{W}^{\partial, k-1}_{\theta, d}(\mb{t}) \ar[rr] && \mathcal{W}^{\partial, k-1}_{\theta, d}(\mb{s}).
}
$$
Since $k, d-k \geq 3$, the degrees of the surgeries associated to the embedding $e$ are within the range covered by Theorem \ref{theorem: gromov lawson on boundary}. 
It follows that the top-horizontal map is a weak homotopy equivalence. 
The theorem then follows by application of Proposition \ref{proposition: homotopy colimit fibre sequence}. 
\end{proof}

\begin{remark} \label{remark: main dimensional case of interest}
The particular dimensional case of the above theorem that we will have to use is $d = 2n \geq 6$ and $k = n$. 
We obtain a fibre sequence, 
$
\mathcal{R}^{+}(M) \longrightarrow \mathcal{W}^{\partial, \psc, n-1}_{\theta, 2n} \longrightarrow \mathcal{W}^{\partial, n-1}_{\theta, 2n},
$
for $M \in \mathcal{M}^{\partial}_{\theta, 2n}$.
\end{remark}

\subsection{Proof of Theorem \ref{theorem: commutativity of restrictions and transfer}} \label{subsection: proof of boundary map commutativity}
Fix $M \in \mathcal{M}^{\partial}_{\theta, 2n}$.
Denote $P := \partial M$. 
Fix an element $g_{P} \in \mathcal{R}^{+}(P)$.
By assembling Theorems \ref{theorem: fibresequence forget the metric}, \ref{theorem: boundary map is quasifibration d = 2n-1}, \ref{theorem: forgetful map fibration}, and Proposition \ref{proposition: chernysh restriction theorem}, we obtain the commutative diagram 
\begin{equation} \label{equation: three row commutative diagram}
\xymatrix{
\mathcal{R}^{+}(M)_{g_{P}} \ar[rr] \ar[d] && \mathcal{R}^{+}(M) \ar[rr] \ar[d] && \mathcal{R}^{+}(P) \ar[d]\\
\mathcal{W}^{\psc, n-1}_{\theta, 2n}(P) \ar[rr] \ar[d] && \mathcal{W}^{\partial, \psc, n-1}_{\theta, 2n} \ar[rr] \ar[d] && \mathcal{W}^{\psc, n-1}_{\theta, 2n-1} \ar[d] \\
\mathcal{W}^{n-1}_{\theta, 2n}(P) \ar[rr]  && \mathcal{W}^{\partial, n-1}_{\theta, 2n} \ar[rr]  && \mathcal{W}^{n-1}_{\theta, 2n-1}.
}
\end{equation}
By the results mentioned above, all columns are fibre sequences. 
Furthermore, the top and bottom rows are fibre sequences as well. 
The fibre transports of the columns yield the commutative diagram 
\begin{equation} \label{equation: fibre transport stage 1}
\xymatrix{
\Omega\mathcal{W}^{n-1}_{\theta, 2n}(P) \ar[d] \ar[rr]  && \Omega\mathcal{W}^{\partial, n-1}_{\theta, 2n} \ar[d] \ar[rr]  && \Omega\mathcal{W}^{n-1}_{\theta, 2n-1} \ar[d] \\
\mathcal{R}^{+}(M)_{g_{P}} \ar[rr] && \mathcal{R}^{+}(M) \ar[rr] && \mathcal{R}^{+}(P). 
}
\end{equation}
Both of the rows in (\ref{equation: fibre transport stage 1}) are fibre sequences. 
Thus, the fibre-transports yield the commutative diagram,
\begin{equation} \label{equation: the commutative diagram in question}
\xymatrix{
\Omega^{2}\mathcal{W}^{n-1}_{\theta, 2n-1} \ar[d]^{\Omega j_{2n-1}} \ar[rr]^{\Omega\iota_{2n}} && \Omega\mathcal{W}^{n-1}_{\theta, 2n}(P) \ar[d]^{j_{2n}} \\
\Omega\mathcal{R}^{+}(P) \ar[rr]^{T} && \mathcal{R}^{+}(M)_{g_{P}},
}
\end{equation}
and this is precisely the commutative diagram from the statement of Theorem \ref{theorem: commutativity of restrictions and transfer}.
This completes the proof of Theorem \ref{theorem: commutativity of restrictions and transfer}.

\appendix

\section{Proof of Theorem \ref{theorem: equivalence with long manifolds with boundary}} \label{section: long manifolds to W partial}
We now embark on the proof of Theorem \ref{theorem: equivalence with long manifolds with boundary}.
This proof involves simply mimicking what was done in \cite[Appendix A]{P 17} or \cite[Section 5]{MW 07}.
For this reason we just provide sketches of the proofs and omit the details so as to avoid too much redundancy.

\subsection{Configuration spaces} \label{subsection: configuration spaces}
Our main construction requires use of certain configuration spaces to be defined below. 
Recall from \cite{P 17} the space $\mathcal{D}^{\mf}_{\theta, \loc}$ that consists of tuples $(\bar{x}; (V, \sigma))$ where $\bar{x} \subset \R\times\R^{\infty}$ is a zero-dimensional submanifold (not necessarily compact), and $(V, \sigma): \bar{x} \longrightarrow G_{\theta}^{\mf}(\R\times\R^{\infty})_{\loc}$ is a map, subject to the following condition: for all $\delta > 0$, the set $\bar{x}\cap\left((-\delta, \delta)\times\R^{\infty})\right)$ is a finite set (or in other words is a compact $0$-manifold). 
For all $P \in \mathcal{M}_{\theta, d-1}$, there is a map 
$$\mb{L}: \mathcal{D}^{\mf}_{\theta, d+1}(P) \longrightarrow \mathcal{D}^{\mf}_{\theta, \loc}$$
defined by sending $W \in \mathcal{D}^{\mf}_{\theta, d+1}(P)$ to $(\bar{x}; (V, \sigma))$, where $\bar{x} \subset \R\times\R^{\infty}$ is the set of critical points of the height function $h_{W}: W \longrightarrow \R$, which is Morse. 
For each $x \in \bar{x}$, we have $V(x) = T_{x}W$ and $\sigma(x): T_{x}W\otimes T_{x}W \longrightarrow \R$ is the Hessian associated to $h_{W}$ at the critical point $x$.
We will need to work with a version of the space $\mathcal{D}^{\mf}_{\theta, \loc}$ adapted to receive a map from the space $\mathcal{D}_{\theta, d+1}^{\partial, \mf}(P)$.
Recall from Definition \ref{defn: grassmannian of half-planes} the space $G_{\theta, d+1}^{\partial, \mf}(\R\times\R^{\infty})_{\loc}$. 
\begin{defn} \label{defn: space of boundary configurations}
The space $\mathcal{D}^{\partial, \mf}_{\theta, d+1, \loc}$ consists of pairs $\left((\bar{x}_{0}; (V_{0}, \sigma_{0}, \phi_{0})), \; (\bar{x}_{1}; (V_{1}, \sigma_{1}))\right)$, where:
\begin{enumerate} \itemsep.2cm
\item[(i)] $(\bar{x}_{1}; (V_{1}, \sigma_{1}))$ is an element of $\mathcal{D}^{\mf}_{\theta, d+1, \loc}$;
\item[(ii)] $\bar{x}_{0} \in \R\times\R^{\infty}$ is a zero-dimensional submanifold, and $(V_{0}, \sigma_{0}, \phi_{0}): \bar{x}_{0} \longrightarrow G_{\theta, d+1}^{\partial, \mf}(\R\times\R^{\infty})_{\loc}$ is a map. 
The pair $(\bar{x}_{0}; (V_{0}, \sigma_{0}, \phi_{0}))$ satisfies the following condition: for all $\delta > 0$, the set $\bar{x}\cap\left((-\delta, \delta)\times\R^{\infty})\right)$ is a finite set.
\end{enumerate}
For elements of $\mathcal{D}^{\partial, \mf}_{\theta, \loc}$ we will often compress the notation and sometimes denote, 
$$(\bar{x}, (V, \sigma)) = \left((\bar{x}_{0}; (V_{0}, \sigma_{0}, \phi_{0})), \; (\bar{x}_{1}; (V_{1}, \sigma_{1}))\right).$$
\end{defn}
We have a map 
\begin{equation} \label{equation: localization map manifolds with boundary}
\mb{L}^{\partial}: \mathcal{D}^{\partial, \mf}_{\theta, d+1}(P) \longrightarrow \mathcal{D}^{\partial, \mf}_{\theta, d+1, \loc}
\end{equation}
that is defined by sending $W \in \mathcal{D}^{\partial, \mf}_{\theta, d+1}(P)$ to $\left((\bar{x}_{0}; (V_{0}, \sigma_{0}, \phi_{0})), \; (\bar{x}_{1}; (V_{1}, \sigma_{1}))\right)$, where:
\begin{enumerate} \itemsep.2cm
\item[(i)] $\bar{x}_{0} \in \{0\}\times\R\times\R^{\infty}$ is the set of critical points of $h_{W}$ that occur on the boundary of $W$. 
For each $x \in \bar{x}_{0}$, $V_{0}(x) = T_{x}W$ and $\sigma_{0}(x)$ is the Hessian at $T_{x}W$. 
The linear map $\phi_{0}(x): T_{x}W \longrightarrow \R$ is the defined to be the differential (at $x$) of the map 
$$W \hookrightarrow \R\times(-\infty, 0]\times\R^{\infty-1} \stackrel{\text{proj.}} \longrightarrow (-\infty, 0].$$
The fact that $W$ intersects $\R\times\{0\}\times\R^{\infty-1}$ transversally implies that $\phi_{0}$ is non-zero, and thus the triple $(V_{0}, \sigma_{0}, \phi_{0})$ is indeed an element of $G_{\theta, d+1}^{\partial, \mf}(\R\times\R^{\infty})_{\loc}$.
\item[(ii)] $\bar{x}_{1} \in (0, \infty)\times\R\times\R^{\infty}$ is the set of critical points of $h_{W}|_{\Int(W)}: \Int(W) \longrightarrow \R$, and $(V_{1}, \sigma_{1})$ is determined by the tangent space and Hessian at the points in $\bar{x}_{1}$.
\end{enumerate}
We will call this map $\mb{L}^{\partial}$ the \textit{localization map}.

\subsection{The space $\mathcal{L}^{\partial}_{\theta, d}$} \label{subsection: alternative model}
In this section we construct the space $\mathcal{L}^{\partial}_{\theta, d}$ appearing in the statement of Theorem \ref{theorem: equivalence with long manifolds with boundary}. 
The construction is very similar to what was done in \cite[Appendix A]{P 17}; the difference being that here we work with manifolds with boundary. 
The definition requires a few preliminary steps. 
\begin{Construction} \label{Construction: sadle set}
Let $(V, \sigma) \in G^{\mf}_{\theta, d+1}(\R^{\infty})_{\loc}$.  
Observe that the function 
\begin{equation} \label{equation: morse vector space}
f_{V}: V \longrightarrow \R, \quad f_{V}(v) := \sigma(v, v)
\end{equation}
is a Morse function on $V$ with exactly one critical point at the origin with index equal to $\text{index}(\sigma)$.
Let $(V, \sigma) \in G^{\mf}_{\theta}(\R^{\infty})_{\loc}$. 
The subspace $\sdl(V, \sigma) \subset V$ is defined by 
$$
\sdl(V, \sigma) = \{v \in V \; | \; ||v_{+}||^{2} ||v_{-}||^{2} \leq 1 \;\},
$$
where $v = (v_{-}, v_{+})$ is the coordinate representation of $v$ using the splitting $V = V^{-}\oplus V^{+}$ given by the negative and positive eigenspaces of $\sigma$.
The norm on $V$ is the one induced by the Euclidean inner product on the ambient space in which the subspace $V$ is contained.

If $(\bar{x}; (V, \sigma)) \in \mathcal{D}^{\mf}_{\theta, \loc}$, we may form the space, 
$$
\sdl(V, \sigma; \bar{x}) \; := \; \bigsqcup_{x \in \bar{x}}\sdl(V(x), \sigma(x)).
$$
Let $h_{\bar{x}}: \bar{x} \longrightarrow \R$ denote the height function, i.e.\ the projection $\bar{x} \hookrightarrow \R\times\R^{\infty} \stackrel{\text{proj}} \longrightarrow \R$. 
The function 
$$f_{V, \bar{x}}: \sdl(V, \sigma; \bar{x}) \longrightarrow \R$$
is defined by the formula,
$f_{V, \bar{x}}(z)  =  f_{V(x)}(z) + h_{\bar{x}}(x),$
if $z$ is contained in the path component
$$\sdl(V(x), \sigma(x)) \; \subset \; \sdl(V, \sigma; \bar{x}),$$
where $x \in \bar{x}$. 
\end{Construction}

We need a similar construction for elements of $G^{\partial, \mf}_{\theta, d+1}(\R^{\infty})_{\loc}$.
\begin{defn} \label{defn: sadle construction boundary}
Let $(V, \sigma, \phi) \in G^{\partial, \mf}_{\theta, d+1}(\R^{\infty})_{\loc}$.  
Let $f_{V}: V \longrightarrow \R$ be defined in the same way as in Construction \ref{Construction: sadle set}. 
Let $\bar{V} = \Ker(\phi)$.
We let 
$
\bar{f}_{V}: \bar{V} \longrightarrow \R
$
denote the restriction of $f_{V}$ to $\bar{V}$. 
Since $V^{-} \leq \bar{V}$, it follows that $\bar{f}_{V}$ is Morse with a single critical point at the origin with index equal to $\text{index}(\sigma) = \text{index}(\bar{\sigma})$.
We then define, 
$$
\textstyle{\sdl^{\partial}}(V, \sigma) = \sdl(V, \sigma)\cap\phi^{-1}([0, \infty)).
$$
It follows that $\sdl(\bar{V}, \bar{\sigma}) \subset \sdl^{\partial}(V, \sigma)\cap\phi^{-1}(0)$, and we denote,
$$\partial_{0}\textstyle{\sdl^{\partial}}(V, \sigma) \; := \; \sdl(\bar{V}, \bar{\sigma}).$$
\end{defn}
For $(\bar{x}; (V, \sigma)) := \left((\bar{x}_{0}; (V_{0}, \sigma_{0}, \phi_{0})), \; (\bar{x}_{1}; (V_{1}, \sigma_{1}))\right) \in \mathcal{D}^{\partial, \mf}_{\theta, d+1, \loc}$, we may form the space 
\begin{equation} \label{equation: saddle configuration}
\textstyle{\sdl^{\partial}}(V, \sigma; \bar{x}) \; = \; \left[\displaystyle{\bigsqcup_{x \in \bar{x}_{0}}}\textstyle{\sdl^{\partial}}(V_{0}(x), \sigma_{0}(x))\right]\bigsqcup\left[\displaystyle{\bigsqcup_{y \in \bar{x}_{1}}}\sdl(V_{1}(y), \sigma_{1}(y))\right].
\end{equation}
The first space in the above disjoint union we denote by $\sdl^{\partial}_{0}(V, \sigma; \bar{x})$, while the second space is denoted by $\sdl^{\partial}_{1}(V, \sigma; \bar{x})$.
The function 
$$f_{V, \bar{x}}: \textstyle{\sdl^{\partial}}(V, \sigma; \bar{x}) \longrightarrow \R$$
is defined in the same way as in Construction \ref{Construction: sadle set}. 
As before this function is Morse. 
Some of the critical points occur on the boundary.

Let $W \in \mathcal{D}^{\partial, \mf}_{\theta, d+1}(P)$ and $(\bar{x}; V, \sigma) \in \mathcal{D}^{\partial, \mf}_{\theta, d+1, \loc}$ be elements with $\mb{L}^{\partial}(W) = (\bar{x}; V, \sigma)$, where recall that $\mb{L}^{\partial}$ is the localization map (\ref{equation: localization map manifolds with boundary}). 
A \textit{coupling} between $W$ and $(\bar{x}; V, \sigma)$ is an embedding 
$$\lambda: \left(\textstyle{\sdl^{\partial}}(V, \sigma; \bar{x}), \; \partial_{0}\textstyle{\sdl^{\partial}}(V, \sigma; \bar{x})\right) \longrightarrow \left(W, \; \partial W\right)$$
that satisfies the following conditions:
\begin{enumerate} \itemsep.2cm
\item[(i)] $f_{V, \bar{x}} = h_{W}\circ\lambda$, where recall that $h_{W}: W \longrightarrow \R$ is the height function on $W$;
\item[(ii)] $\hat{\ell}_{\sdl(V, \sigma)} = \lambda^{*}\hat{\ell}_{W}$. 
\end{enumerate}

We are now ready to define the space $\mathcal{L}^{\partial}_{\theta, d+1}$ from the statement of Theorem \ref{theorem: equivalence with long manifolds with boundary}. 
\begin{defn} \label{defn: homotopy colim decomp of L}
Fix $\mb{t} \in \widehat{\mathcal{K}}_{d}$. 
The space $\mathcal{L}^{\partial}_{\theta, d+1}(\mb{t})$ consists of tuples 
$\left(W, (\bar{x}; V, \sigma), \lambda, \delta, \phi\right)$
where:
\begin{enumerate} \itemsep.2cm
\item[(i)] $W \in \mathcal{D}^{\partial, \mf}_{\theta, d+1}$, $(\bar{x}; V, \sigma) \in \mathcal{D}^{\mf}_{\theta, d+1, \loc}$ and $\mb{L}^{\partial}(W) = (\bar{x}; V, \sigma)$;
\item[(ii)] $\delta: \bar{x} \longrightarrow \{-1, 0, +1\}$ is a function subject to the following condition:
the height function,
$h_{\bar{x}}: \bar{x} \longrightarrow \R,$
admits a lower bound on $\delta^{-1}(+)$ and an upper bound on $\delta^{-1}(-1)$. 
\item[(iii)] $\phi: \mb{t} \stackrel{\cong} \longrightarrow \delta^{-1}(0)$ is a function over the set $\{0, 1, \dots, d+1\}$;
\end{enumerate}
\end{defn}
We need to describe how the correspondence $\mb{t} \mapsto \mathcal{L}^{\partial}_{\theta, d+1}$ defines a contravariant functor on $\widehat{\mathcal{K}}_{d}$. 
Let $(j, \varepsilon): \mb{s} \longrightarrow \mb{t}$ be a morphism in $\widehat{\mathcal{K}}_{d}$.
The induced map 
$$
(j, \varepsilon)^{*}: \mathcal{L}^{\partial}_{\theta, d+1}(\mb{t}) \longrightarrow \mathcal{L}^{\partial}_{\theta, d+1}(\mb{s})
$$
sends an element $(W, (\bar{x}; V, \sigma), \lambda, \delta, \phi) \in \mathcal{L}^{\partial}_{\theta, d+1}(\mb{t})$ to $(W, (\bar{x}; V, \sigma), \lambda, \delta', \phi')$ where $\phi' = \phi\circ j$ and 
\begin{equation} \label{equation: functoriality for delta}
\delta'(y) \; = \; 
\begin{cases}
\varepsilon(s) \quad \text{if $y = \phi(s)$ \; where $s \in \mb{s}\setminus j(\mb{t})$,}\\
\delta(y) \quad \text{otherwise.}
\end{cases}
\end{equation}
We then define,
$$
\mathcal{L}^{\partial}_{\theta, d+1} \; = \; \hocolim_{\mb{t} \in \widehat{\mathcal{K}}_{d}}\mathcal{L}^{\partial}_{\theta, d+1}(\mb{t}).
$$
Given an integer $k \in \Z_{\geq-1}$ and $\mb{t} \in \widehat{\mathcal{K}}_{d}^{k}$, the subspace $\mathcal{L}^{\partial, k}_{\theta, d+1}(\mb{t}) \subset \mathcal{L}^{\partial}_{\theta, d+1}(\mb{t})$ consists of those 
$$(W, (\bar{x}; V, \sigma), \lambda, \delta, \phi)$$ 
for which $W \in \mathcal{D}^{\mf, k}_{\theta, d+1}$. 
We may also take the homotopy colimit and define,
$$
\mathcal{L}^{\partial, k}_{\theta, d+1} \; = \; \hocolim_{\mb{t} \in \widehat{\mathcal{K}}_{d}}\mathcal{L}^{\partial, k}_{\theta, d+1}(\mb{t}).
$$
For all $k$ there is a forgetful map, 
$$
\mathcal{L}^{\partial, k}_{\theta, d+1}  \longrightarrow \mathcal{D}^{\partial, \mf, k}_{\theta, d+1}.
$$
The following proposition is proven by tracing through all of the steps that were covered in \cite[Section 5]{MW 07} verbatim. 
For this reason we omit the proof. 
\begin{proposition} \label{proposition: forgetful map is equivalence}
For all $k$, the forgetful map, 
$\mathcal{L}^{\partial, k}_{\theta, d+1}  \longrightarrow \mathcal{D}^{\partial, \mf, k}_{\theta, d+1},$
is a weak homotopy equivalence.
\end{proposition} 

\begin{remark} \label{remark: above proposition is not actually necessary}
As explained in Remark \ref{remark: necessity of above proposition}, in order to prove Theorem B (the main result of the paper) we don't actually need to use the fact that the forgetful map 
$\mathcal{L}^{\partial, k}_{\theta, d+1}  \longrightarrow \mathcal{D}^{\partial, \mf, k}_{\theta, d+1}$
is a weak homotopy equivalence; 
all that is needed is the existence of the map. 
What is more important is the weak homotopy equivalence $\mathcal{L}^{\partial, k}_{\theta, d+1} \stackrel{\simeq} \longrightarrow \mathcal{W}^{\partial, k}_{\theta, d+1}$ which we prove in the next section. 
\end{remark}

\subsection{Regularization} \label{subsection: regularization}
We now recall from \cite[Appendix A]{P 17} a construction used to define a map 
$$\mathcal{L}^{\partial}_{\theta, d+1} \longrightarrow \mathcal{W}^{\partial}_{\theta, d}.$$
Choose once and for all a diffeomorphism $\psi: \R \longrightarrow (-\infty, 0)$ such that $\psi(t) = t$ for $t < -\tfrac{1}{2}$, and a smooth nondecreasing function $\varphi: [0, 1] \longrightarrow [0,1]$ such that $\varphi(x) = x$ for $x$ close to $0$, and $\varphi(x) = 1$ for $x$ close to $1$. 
Define,
$$
\psi_{x}(t) = \varphi(x)t + (1 - \varphi(x))\psi(t)
$$
for $x \in [0,1]$. 
Then $\psi_{0} = \psi$ embeds $\R$ in $\R$ with image $(-\infty, 0)$, whereas each $\psi_{x}$ for $x > 0$ is a diffeomorphism $\R \longrightarrow \R$. 

Let $(V, \sigma, \phi) \in G_{\theta, d+1}^{\partial, \mf}(\R^{\infty})_{\loc}$. 
Recall form (\ref{equation: morse vector space}) the Morse function $f_{V}: V \longrightarrow \R$.
We define functions 
$$\begin{aligned}
f^{+}_{V}: \textstyle{\sdl^{\partial}}(V, \sigma)\setminus V^{+} &\longrightarrow \R, \\
 f^{-}_{V}: \textstyle{\sdl^{\partial}}(V, \sigma)\setminus V^{-} &\longrightarrow \R,
 \end{aligned}$$
by the formulae
\begin{equation} \label{equation: regularization functions}
\begin{aligned}
f^{+}_{V}(v) &= \psi^{-1}_{x}(t), \\
f^{-}_{V}(v) &= (-\psi_{x})^{-1}(-t),
\end{aligned}
\end{equation}
where $t = f_{V}(v)$ and $x = ||v_{-}||^{2}||v_{+}||^{2}$. 
The following proposition can be verified by hand.
\begin{proposition} \label{proposition: regularized functions}
The functions $f^{\pm}_{V}$ defined in (\ref{equation: regularization functions}) are proper submersions.
Furthermore, the restrictions of $f^{\pm}_{V}$ to  $\partial_{0}\textstyle{\sdl^{\partial}}(V, \sigma)\setminus\bar{V}^{+}$ are proper submersions as well.
\end{proposition}

Let $(\bar{x}; (V, \sigma)) \in \mathcal{D}^{\partial, \mf}_{\theta, d+1, \loc}$.
The functions (\ref{equation: regularization functions}) determine functions 
\begin{equation} \label{equation: regularized functions 2}
\begin{aligned}
f^{+}_{V, \bar{x}}: \textstyle{\sdl^{\partial}}(V, \sigma; \bar{x})\setminus V^{+} &\longrightarrow \R, \\
 f^{-}_{V, \bar{x}}: \textstyle{\sdl^{\partial}}(V, \sigma; \bar{x})\setminus V^{-} &\longrightarrow \R,
 \end{aligned}
\end{equation}
defined by 
$$
f^{\pm}_{V, \bar{x}}(z) \; = \; f^{\pm}_{V}(z) + h(x)
$$
for $z$ on $\sdl^{\partial}(V, \sigma)|_{x}$ for each $x \in \bar{x}$.
By Proposition \ref{proposition: regularized functions} it follows that the functions $f^{\pm}_{V, \bar{x}}$ are proper submersions as well. 
We now use the above construction to define the map,
$$\mathcal{L}^{\partial}_{\theta, d+1} \longrightarrow \mathcal{W}^{\partial}_{\theta, d}.$$
\begin{Construction} \label{construction: regularization}
Let $\mb{t} \in \widehat{\mathcal{K}}_{d}$ and 
let $(W, (\bar{x}; V, \sigma), \lambda, \delta, \phi)$ be an element of $\mathcal{L}^{\partial}_{\theta, d+1}(\mb{t})$.
Let $(V_{+1}, \sigma_{+1}), (V_{-1}, \sigma_{-1})$, and $(\widehat{V}, \widehat{\sigma})$ denote the restrictions of $(V, \sigma)$ to the subsets 
$$\delta^{-1}(+1), \; \delta^{-1}(-1), \; \delta^{-1}(0) \; \subset \; \bar{x}.$$
Let us denote the submanifold,
$$W^{\rg} = W\setminus\lambda(V^{+}_{+1}\cup \widehat{V}^{+}\cup V^{-}_{-1}).$$
We define a function 
$f^{\rg}: W^{\rg} \longrightarrow \R$ by
\begin{equation}
f^{\rg}(z) \; = \; 
\begin{cases}
f(z) &\quad \text{if $z \in \Image(\lambda)$,}\\
f^{+}_{V, \bar{x}}(v) &\quad \text{if $z = \lambda(v)$ and $v \in V_{+1}\cup \widehat{V}$,}\\
f^{-}_{V, \bar{x}}(v) &\quad \text{if $z = \lambda(v)$ and $v \in V_{-1}$.}
\end{cases}
\end{equation}
Let $j_{W^{\rg}}: W \longrightarrow (-\infty, 0]\times\R^{\infty-1}$ denote the composite 
$$W \hookrightarrow \R\times(-\infty, 0]\times\R^{\infty-1} \longrightarrow (-\infty, 0]\times\R^{\infty-1}.$$
We re-embed $W^{\rg}$ into $\R\times(-\infty, 0]\times\R^{\infty-1}$ using the product map
$$
f^{\rg}\times j_{W^{\rg}}: W^{\rg} \longrightarrow \R\times((-\infty, 0]\times\R^{\infty-1}),
$$
and we let $\widetilde{W} \subset \R\times(-\infty, 0]\times\R^{\infty-1}$ denote its image. 
By construction, the height function 
$$
h_{\widetilde{W}}: \widetilde{W} \longrightarrow \R
$$
is a proper submersion, and thus,
$$M := h^{-1}_{\widetilde{W}}(0),$$
is a compact $d$-dimensional submanifold of $\{0\}\times(-\infty, 0]\times\R^{\infty-1}$.  
The restriction of $\lambda$ gives embeddings 
$$\begin{aligned}
\widehat{e}_{0}: \left(S(\widehat{V}^{-}_{0})\times_{\mb{t}}D_{+}(\widehat{V}^{+}_{0}), \; S(\widehat{V}^{-}_{0})\times_{\mb{t}}\partial_{0}D_{+}(\widehat{V}^{+}_{0})\right)  &\longrightarrow (M, \partial M), \\
\widehat{e}_{1}: S(\widehat{V}^{-}_{1})\times_{\mb{t}}D(\widehat{V}^{+}_{1}) &\longrightarrow M.
\end{aligned}$$
These embeddings together with $(\widehat{V}, \widehat{\sigma})$ determines an element 
$$(M, (\widehat{V}, \widehat{\sigma}), e) \in \mathcal{W}^{\partial}_{\theta, d+1}(\mb{t}).$$ 
The correspondence 
$$
\left(W, (\bar{x}; V, \sigma), \lambda, \delta, \phi\right) \; \mapsto \; (M, (\widehat{V}, \widehat{\sigma}), e)
$$
defines a map 
$$
\mathcal{L}^{\partial}_{\theta, d+1}(\mb{t}) \longrightarrow \mathcal{W}^{\partial}_{\theta, d+1}(\mb{t}).
$$
These maps determine a natural transformation of functors in $\widehat{\mathcal{K}}_{d}$, and thus it induces a map 
\begin{equation} \label{equation: regularization map}
\text{Reg}: \mathcal{L}^{\partial}_{\theta, d+1} \longrightarrow \mathcal{W}^{\partial}_{\theta, d+1}.
\end{equation}
\end{Construction}

The following theorem is proven in the same way as \cite[Proposition A.13]{P 17}.
\begin{proposition} \label{proposition: level wise equivalence between L and W}
For all $k$ and $\mb{t}$ the map,
$\text{Reg}_{\mb{t}}: \mathcal{L}^{\partial, k}_{\theta, d+1}(\mb{t}) \longrightarrow \mathcal{W}^{\partial, k}_{\theta, d+1}(\mb{t}),$
is a weak homotopy equivalence, and thus the map (\ref{equation: regularization map}) induces the weak homotopy equivalence, 
$
\mathcal{L}^{\partial}_{\theta, d+1} \simeq \mathcal{W}^{\partial}_{\theta, d+1}.
$
\end{proposition}

Proposition \ref{proposition: level wise equivalence between L and W} is proven by constructing a homotopy inverse to $\text{Reg}_{\mb{t}}: \mathcal{L}^{\partial, k}_{\theta, d+1}(\mb{t}) \longrightarrow \mathcal{W}^{\partial, k}_{\theta, d+1}(\mb{t})$ for each $\mb{t} \in \widehat{\mathcal{K}}_{d}$. 
This is done implementing the \textit{long trace} construction which is essentially the same as what was done in \cite[Section A.2]{P 17} or \cite[Section 5]{MW 07}.
We give an outline of this in the following subsection. 
\subsection{Proof of Proposition \ref{proposition: level wise equivalence between L and W}} \label{subsection: the long trace}
We recall a construction from \cite[Section 5]{MW 07}.
We will use this construction to define a map $\mathcal{W}^{\partial}_{\theta, d}(\mb{t}) \longrightarrow \mathcal{L}^{\partial}_{\theta, d+1}(\mb{t})$ for each $\mb{t} \in \widehat{\mathcal{K}}_{d}$.
\begin{Construction} \label{Construction: the long trace}
Given $(V, \sigma) \in G^{\mf}_{\theta, d+1}(\R^{\infty})_{\loc}$, the formula
\begin{equation} \label{equation: formula for trace embedding}
v \mapsto \left(f_{V}(v), \; ||v_{-}||v_{+}, \;  ||v_{-}||^{-1}v_{-} \right)
\end{equation}
defines a smooth embedding
\begin{equation} \label{equation: embedding into Disk sphere}
\phi: \sdl(V, \sigma)\setminus V^{+} \longrightarrow \R\times S(V^{-})\times D(V^{+}),
\end{equation}
with complement 
$[0, \infty)\times\{0\}\times S(V^{-}).$
It respects boundaries and is a map over $\R$, where we use the restriction of $f_{V}$ on the source and the function $(t, x, y) \mapsto t$ on the target. 
Now, let $\mb{t} \in \mathcal{K}_{d}$, $(V, \sigma) \in \left(G^{\mf}_{\theta, d+1}(\R^{\infty})_{\loc}\right)^{\mb{t}}$ and consider the space, 
$$
\sdl(V, \sigma; \mb{t}) \; = \; \bigsqcup_{i \in \mb{t}}\sdl(V(i), \sigma(i)). 
$$
We let 
$V^{\pm}(\mb{t}) \subset \sdl(V, \sigma; \mb{t})$ 
denote the subspace given by the disjoint union $\coprod_{i \in \mb{t}}V^{\pm}(i)$. 
By applying the embedding $\phi$
over each component 
$\sdl(V(i), \sigma(i)) \subset \sdl(V, \sigma, \mb{t}),$ 
we obtain an embedding
$$
\phi_{\mb{t}}: \sdl(V, \sigma; \mb{t})\setminus V^{+}(\mb{t}) \longrightarrow \R\times D(V^{+})\times_{\mb{t}}S(V^{-}),
$$
which has complement $[0, \infty)\times\left(\{0\}\times_{\mb{t}}S(V^{-})\right)$.

We will need an analogue of the above construction for elements of $G^{\partial, \mf}_{\theta, d+1}(\R^{\infty})_{\loc}$.
Let $(V, \sigma) \in G^{\partial, \mf}_{\theta, d+1}(\R^{\infty})_{\loc}$ and consider the space $\sdl^{\partial}(V, \sigma)$.
The same formula (\ref{equation: formula for trace embedding}) defines an embedding 
\begin{equation} \label{equation: embedding into half disk sphere}
\phi^{\partial}: \textstyle{\sdl^{\partial}}(V, \sigma)\setminus V^{+} \longrightarrow \R\times S(V^{-})\times D_{+}(V^{+})
\end{equation}
that sends $\partial_{0}\sdl^{\partial}(V, \sigma)\setminus\bar{V}^{+}$ into $\R\times S(V^{-})\times\partial_{0}D_{+}(V^{+})$.
Furthermore, the restriction of $\phi^{\partial}$ to $\partial_{0}\sdl^{\partial}(V, \sigma)\setminus\bar{V}^{+}$ agrees with the embedding (\ref{equation: embedding into Disk sphere}). 
Let $\mb{t} \in \mathcal{K}_{d}$, let $(V, \sigma) \in \left(G^{\partial, \mf}_{\theta, d+1}(\R^{\infty})_{\loc}\right)^{\mb{t}}$, and consider the space, 
$$
\textstyle{\sdl^{\partial}}(V, \sigma; \mb{t}) \; = \; \bigsqcup_{i \in \mb{t}}\textstyle{\sdl^{\partial}}(V(i), \sigma(i)). 
$$
By applying the embedding $\phi^{\partial}$ with $i \in \mb{t}$
over each component 
$\sdl^{\partial}(V(i), \sigma(i)) \subset \sdl^{\partial}(V, \sigma, \mb{t}),$ 
we obtain an embedding
$$
\phi^{\partial}_{\mb{t}}: \textstyle{\sdl^{\partial}}(V, \sigma; \mb{t})\setminus V^{+}(\mb{t}) \longrightarrow \R\times S(V^{-})\times_{\mb{t}}D_{+}(V^{+}),
$$
which has complement, $[0, \infty)\times\left(S(V^{-})\times_{\mb{t}}\{0\}\right)$.
\end{Construction}

We now use the above construction to define a map $\mathcal{W}^{\partial}_{\theta, d}(\mb{t}) \longrightarrow \mathcal{L}^{\partial}_{\theta, d+1}(\mb{t})$. 
\begin{Construction} \label{Construction: long trace}
Let $(M, (V, \sigma), e) \in \mathcal{W}^{\partial}_{\theta, d}(\mb{t})$. 
The submanifold 
$$\textstyle{\Trc}^{\partial}(e) \subset \R\times\R^{\infty}$$ 
is defined to be the pushout of the diagram
\begin{equation} \label{equation: long trace}
\xymatrix{
\textstyle{\sdl^{\partial}}(V, \sigma; \mb{t})\setminus V^{+}(\mb{t}) \ar[d] \ar[rr] && \R\times M\setminus\left([0, \infty)\times e(S(V^{-}))\times_{\mb{t}}\{0\}\right) \\
\textstyle{\sdl^{\partial}}(V, \sigma; \mb{t}), && 
}
\end{equation}
where the horizontal map is defined by composing $\phi^{\partial}$ with the restriction of $e$ to 
$$
\R\times S(V^{-})\times_{\mb{t}}D_{+}(V^{+}) \setminus \left([0, \infty)\times S(V^{-})\times_{\mb{t}}\{0\}\right).
$$
The $\theta$-orientation $\hat{\ell}_{V}$ on $V$ together with the $\theta$-structure $\hat{\ell}_{M}$ on $M$ determines a $\theta$-structure $\hat{\ell}_{\Trc^{\partial}(e)}$ on $\Trc^{\partial}(e)$. 
The height function 
$$\textstyle{\Trc^{\partial}}(e) \hookrightarrow \R\times\R^{\infty} \longrightarrow \R$$ 
is Morse. 
It has one critical point for each $i \in \mb{t}$, with index equal to $\dim(V^{-}(i))$, all of which have value $0$.
It follows that $\Trc^{\partial}(e)$ equipped with its $\theta$-structure $\hat{\ell}_{\Trc^{\partial}(e)}$ determines an element of the space $\mathcal{D}^{\partial}_{\theta, d+1}$.
By remembering the critical points for the height function $\Trc^{\partial}(e) \longrightarrow \R$ and the embedding 
$$
\textstyle{\sdl^{\partial}}(V, \sigma; \mb{t}) \hookrightarrow \Trc^{\partial}(e),
$$
this construction determines an element of $\mathcal{L}^{\partial}_{\theta, d+1}(\mb{t})$, and thus we obtain a map, 
\begin{equation} \label{equation: long trace map}
\textstyle{\Trc^{\partial}}_{\mb{t}}: \mathcal{W}^{\partial}_{\theta, d}(\mb{t}) \longrightarrow \mathcal{L}^{\partial}_{\theta, d+1}(\mb{t}).
\end{equation}
\end{Construction}
The fact that $\textstyle{\Trc^{\partial}_{\mb{t}}}$ is a homotopy inverse to the map $\text{Reg}_{\mb{t}}$ 
for all $\mb{t}$ follows the same argument used in the proof of \cite[Proposition 5.3.6]{MW 07}.
Following their same proof then yields the proof of Proposition \ref{proposition: level wise equivalence between L and W}. 
Combining Proposition \ref{proposition: level wise equivalence between L and W} with Proposition \ref{proposition: forgetful map is equivalence} yields the weak homotopy equivalences 
$$
\xymatrix{
\mathcal{W}^{\partial}_{\theta, d+1} & \mathcal{L}^{\partial}_{\theta, d+1} \ar[l]_{\simeq} \ar[r]^{\simeq} & \mathcal{D}^{\partial, \mf}_{\theta, d+1}
}
$$
asserted in the statement of Theorem \ref{theorem: equivalence with long manifolds with boundary}. 
This finishes the proof of Theorem \ref{theorem: equivalence with long manifolds with boundary}.

\end{document}